\documentclass[10pt]{article}
\usepackage{latexsym,amsmath} 
\usepackage{times}

\newcommand{\rbal}{r_{\mathrm{bal}}}
\newcommand{\rupper}{r_{\mathrm{upper}}}
\newcommand{\rBP}{r_{\mathrm{BP}}}

\newcommand\pd{p_{\vec d}} 
\newcommand\wMAJ{w_{\mathrm{maj}}} 
\newcommand\bemph[1]{{\bf\em #1}}
\newcommand\mymap{p} 

\newcommand\sigmaMAJ{\sigma_{maj}} 
 
\newcommand\PHId{\PHI_{\vec d}} 
\newcommand\PHIdm{\PHI_{\vec d,\vec m}} 
\def\vec#1{\mathchoice{\mbox{\boldmath$\displaystyle#1$}}
{\mbox{\boldmath$\textstyle#1$}}
{\mbox{\boldmath$\scriptstyle#1$}}
{\mbox{\boldmath$\scriptscriptstyle#1$}}}

\newcommand{\qed}{\hfill$\Box$\smallskip}
\newenvironment{proof}{\emph{Proof.}}{}

\newcommand\SIGMA{\vec\sigma} 
\newcommand\OMEGA{\vec\omega} 
\newcommand\TAU{\vec\tau} 
\newcommand\Zgood{Z'}

\newtheorem{definition}{Definition}[section]

\newtheorem{remark}[definition]{Remark}
\newtheorem{theorem}[definition]{Theorem}
\newtheorem{lemma}[definition]{Lemma}
\newtheorem{proposition}[definition]{Proposition}
\newtheorem{corollary}[definition]{Corollary}

\newtheorem{fact}[definition]{Fact}

\newtheorem{conjecture}[definition]{Conjecture}

\usepackage[dvips]{epsfig} 
\graphicspath{{./Fig_eps/}{./Eps/}} 
\DeclareGraphicsExtensions{.ps,.eps} 
\usepackage{color} 
\usepackage{fullpage}

\newcommand\rk{r_{k\mathrm{-SAT}}}
\newcommand\sign{\mathrm{sign}}

\newcommand\inv{\mathrm{inv}}

\newcommand\id{\mathrm{id}}

\newcommand\dist{\mbox{dist}} 

\newcommand\PHI{\vec\Phi}

\newcommand\cA{\mathcal{A}} 
\newcommand\cB{\mathcal{B}} 
\newcommand\cC{\mathcal{C}} 
\newcommand\cD{\mathcal{D}}

\newcommand\cE{\mathcal{E}}

\newcommand\cH{\mathcal{H}} 
\newcommand\cS{\mathcal{S}} 
\newcommand\cT{\mathcal{T}}

\newcommand\cL{\mathcal{L}} 
\newcommand\cM{\mathcal{M}} 
\newcommand\cO{\mathcal{O}} 
\newcommand\cP{\mathcal{P}} 
 
\newcommand\cY{\mathcal{Y}} 
 
\newcommand\cW{\mathcal{W}} 
\newcommand\cZ{\mathcal{Z}} 

\def\cR{{\mathcal R}}
\def\cC{{\mathcal C}}
\def\cE{{\cal E}}

\newcommand\eul{\mathrm{e}} 
\newcommand\eps{\varepsilon} 
\newcommand\ZZ{\mathbf{Z}} 
\newcommand\Var{\mathrm{Var}} 
\newcommand\Erw{\mathrm{E}} 
\newcommand\pr{\mathrm{P}}

\newcommand{\vecone}{\vec{1}}

\newcommand{\Vol}{\mathrm{Vol}}

\newcommand{\Po}{{\rm Po}} 
\newcommand{\Bin}{{\rm Bin}} 
\newcommand{\Be}{{\rm Be}}

\newcommand{\bink}[2] {{{#1}\choose {#2}}}

\newcommand\ra{\rightarrow} 

\newcommand\bc[1]{\left({#1}\right)} 
\newcommand\cbc[1]{\left\{{#1}\right\}} 
\newcommand\bcfr[2]{\bc{\frac{#1}{#2}}} 
 
\newcommand\brk[1]{\left\lbrack{#1}\right\rbrack} 
 
\newcommand\norm[1]{\left\|{#1}\right\|} 
\newcommand\abs[1]{\left|{#1}\right|}

\newcommand\RR{\mathbf{R}} 
 
\newcommand{\Whp}{W.h.p.} 
\newcommand{\whp}{w.h.p.} 
 
\newcommand{\stacksign}[2]{{\stackrel{\mbox{\scriptsize #1}}{#2}}}

\newcommand{\Erdos}{Erd\H{o}s}
\newcommand{\Renyi}{R\'enyi}

\newcommand{\Exp}{\Erw}

\newcommand\Lem{Lemma}
\newcommand\Prop{Proposition}
\newcommand\Thm{Theorem}
\newcommand\Cor{Corollary}
\newcommand\Sec{Section}

\begin{document} 

\title{\bf Going after the k-SAT Threshold}

\author{
Amin Coja-Oghlan\thanks{%
Goethe University, Mathematics Institute, Frankfurt 60054, Germany,
  {\tt acoghlan@math.uni-frankfurt.de}. Supported by ERC Starting Grant 278857--PTCC (FP7).}
\and
Konstantinos Panagiotou\thanks{%
	University of Munich, Mathematics Institute, Theresienstr.\ 39, 80333 M\"unchen, Germany,
       {\tt kpanagio@math.lmu.de},  Supported by DFG grant PA 2080/2-1.}
}
\date{\today}

\maketitle

\begin{abstract}
Random $k$-SAT  is the single most intensely studied example of a random constraint satisfaction problem.
But despite substantial progress over the past decade,
the threshold for the existence of satisfying assignments is not known precisely for any $k\geq3$.
The best current results, based on the second moment method, yield upper and lower bounds
that differ by an additive $k\cdot \frac{\ln2}2$, a term that is unbounded in~$k$
	(Achlioptas, Peres: STOC~2003).
The basic reason for this gap is 
	the inherent asymmetry of the Boolean value `true' and `false'
		in contrast to the perfect symmetry, e.g., among the various colors in a graph coloring problem.
Here we develop a new \emph{asymmetric second moment method}
that allows us to tackle this issue head on for the first time in the theory of random CSPs.
This technique enables us to compute the
$k$-SAT threshold up to an additive $\ln2-\frac12+O(1/k)\approx 0.19$.
Independently of the rigorous work, physicists have developed a sophisticated but non-rigorous technique called the ``cavity method''
for the study of random CSPs (M\'ezard, Parisi, Zecchina: Science~2002).
Our result matches the best bound that can be obtained from the so-called ``replica symmetric'' version of the cavity method,
and indeed our proof directly harnesses parts of the physics calculations.
\end{abstract}

\section{Introduction}

Since the early 2000s physicists have developed a sophisticated but highly non-rigorous technique
called the ``cavity method''
for the study of random constraint satisfaction problems. 
This method allowed them to put forward a very detailed \emph{conjectured} picture according to which
various phase transitions affect both computational and structural properties of random CSPs.
In addition, the cavity method has inspired new message passing algorithms
called \emph{Belief/Survey Propagation guided decimation}.
Over the past few years there has been significant progress in turning bits and pieces of the physics picture into rigorous theorems.
Examples include results on the interpolation method~\cite{Energetic,BGT10} or the geometry of the solution space~\cite{Barriers,Molloy,Prasad}
and their algorithmic implications~\cite{DPLL,BetterAlg}.

In spite of this progress, substantial gaps remain.
Perhaps most importantly, in most random CSPs the threshold for the existence of solutions is not known precisely.
In 
the relatively simple case of the random $k$-NAESAT (``Not-All-Equal-Satisfiability'')
problem the difference between the best current lower and upper bounds is
as tiny as $2^{-\Omega(k)}$~\cite{Catching}.
By contrast, in random graph $k$-coloring, a problem already studied by \Erdos\ and \Renyi\ in the 1960s,
 the best current bounds differ by $\Theta(\ln k)$~\cite{AchNaor}.
Hence, the difference is \emph{unbounded} in terms of the number of colors.
Even worse, in random $k$-SAT the gap is as big as $\Theta(k)$~\cite{yuval}.
Yet 
random $k$-SAT is probably the single most important example of a random CSP, not least
due to the great amount of experimental and algorithmic work conducted on it (e.g.,~\cite{KirkpatrickSelman,Kroc}).

The reason for
the large gap in random $k$-SAT 
is that the satisfiability problem lacks a certain \emph{symmetry property}.
This property is vital to the current rigorous proof methods, particularly the \emph{second moment method}, on which most
of the previous work is  based (e.g., \cite{nae,AchNaor,yuval}).
More precisely, in random graph coloring the different colors all play the exact same role:
	for any proper coloring of a graph, another proper coloring can be obtained by simply permuting the color classes
		(e.g., color all red vertices blue and vice versa).
Similarly, in $k$-NAESAT, where the requirement is that in each clause at least one literal must be true and
at least one false,  the binary inverse of any NAE-solution is a NAE-solution as well.
By contrast, 
in $k$-SAT there is an inherent asymmetry between 
the Boolean values `true' and `false'.

As has been noticed in prior work~\cite{nae,yuval}, the  second moment method 
is fundamentally ill-posed to deal with such asymmetries.
Roughly speaking, the second moment method is based on the assumption that in a random CSP instance,
two randomly chosen solutions are perfectly uncorrelated.
But in random $k$-SAT, this is simply not the case.
Indeed, suppose that a variable $x$ appears much more often positively than negatively throughout the formula.
Then it seems reasonable to expect that most satisfying assignments set $x$ to `true', thereby satisfying all clauses where
$x$ appears positively.
More generally, define the \emph{majority vote} $\sigmaMAJ$ to be the assignment that sets variable $x$ to true
if it appears more often positively than negatively, and to false otherwise.
Then we expect that the satisfying assignments of a random formula
``gravitate toward''  $\sigmaMAJ$.
Unfortunately, the correlations among satisfying assignments induced by this drift toward $\sigmaMAJ$ doom the second moment method.
Previously this issue was sidestepped by symmetrizing the problem artificially~\cite{nae,yuval}.
But this inevitably leaves a $\Theta(k)$ gap.

The main contribution of the present work is 
a new \emph{asymmetric second moment method}
that enables us to tackle this problem  head on.
A key feature of this method is that we harness the Belief Propagation calculation from physics,
	called the ``replica symmetric case'' of the cavity method in physics jargon.
We are going to employ Belief Propagation directly as an ``educated guess'' 
in the design the random variable upon which our proof is based in order
to quantify how much a typical satisfying assignment leans toward $\sigmaMAJ$.

This is in contrast to most prior work on the subject, where
individual statements hypothesized on the basis of physics
arguments were proved via completely different methods (with the notable exception of the interpolation
technique~\cite{Energetic,BGT10,FranzLeone}).
Hence, we view the present work as a pivotal step in the long-term effort of providing a rigorous foundation
for the physicists' cavity method.
In fact, the general approach developed here does not hinge on particular
properties of the $k$-SAT problem, and thus we expect that the technique will extend
to other asymmetric problems as well.
Examples include not only other random CSPs that are asymmetric per se, but also
instances of random problems that arise at intermediate steps of message passing algorithms
such as \emph{Belief/Survey Propagation guided decimation}, even if the initial problem is symmetric.
In particular, we believe that getting a handle on asymmetric problems is a necessary step to
analyze such message passing algorithms  accurately.

To state our results precisely, we let $k\geq3,n>0$ be integers and we let $V=\cbc{x_1,\ldots,x_n}$ be a set
of $n$ Boolean variables.
Further, let $\PHI=\PHI_k(n,m)$ denote a Boolean formula with $m$ clauses of length $k$ over the variables $V$
chosen uniformly at random among all $(2n)^{km}$ such formulas.
Let $r=m/n$ denote the \emph{density}.
We say that an event occurs \emph{with high probability} (`\whp') if its probability tends to $1$ as $n\ra\infty$.

Friedgut~\cite{Ehud} showed that for any $k\geq3$ there exists a \emph{threshold sequence}%
	\footnote{It is widely conjecture but as yet unproved that $\rk(n)$ converges for any $k\geq3$.}
	 $\rk(n)$ such that
for any (fixed) $\eps>0$ \whp\ $\PHI$ is satisfiable if $m/n<(1-\eps)\rk(n)$, while for $m/n>(1+\eps)\rk(n)$
$\PHI$ is unsatisfiable \whp

Upper bounds on $\rk$ can be obtained via the \emph{first moment method}.
The best current ones~\cite{FranzLeone,KKKS} are
	\begin{equation}\label{XeqUpper}
	\rk\leq\rupper=2^k\ln2-\bc{1+\ln2}/2+o_k(1),
	\end{equation}
where 
$o_k(1)$ hides a term that tends to $0$ for large $k$.
The best prior lower bound is due to Achlioptas and Peres~\cite{yuval}, who used a ``symmetric'' second moment argument to show
	\begin{equation}\label{Xeqprevious}
	\rk\geq\rbal=2^k\ln2-k\cdot\frac{\ln2}2-\bc{1+\frac{\ln2}2}+o_k(1).
	\end{equation}
The bounds~(\ref{XeqUpper}) and~(\ref{Xeqprevious}) leave an additive gap of
 $k\cdot\frac{\ln2}2+\frac12+o_k(1)$,  i.e., the gap is unbounded in terms of $k$.

\begin{theorem}\label{XThm_main}
There is $\eps_k=o_k(1)$ such that
	\begin{equation}\label{XeqrBP}
	\rk\geq\rBP=2^k\ln2-\frac{3\ln2}{2}-\eps_k.
	\end{equation}
\end{theorem}

Achlioptas and Peres asked whether the gap $\rupper-\rk$ is bounded by an absolute constant (independent of $k$).
\Thm~\ref{XThm_main} answers this question, 
reducing the gap to $\ln2-\frac12\approx 0.19$.
No attempt at optimizing the error term $\eps_k$ has been made, but
 our proofs yield rather directly that $\eps_k=O(1/k)$.

Apart from the quantitative improvement,
the main point of this paper is that we manage to solve the problem of asymmetry in random CSPs
for the first time.
To explain this point, we start by discussing what we mean by asymmetry and
how it derails the second moment method.
That this is so was already intuited in~\cite{nae,yuval}.
In the next section, we are going to verify and elaborate on those discussions.

\section{Asymmetry and the second moment method}\label{XSec_bla}

\noindent
{\bf\em The second moment method.}
In general, the second moment method works as follows.
Suppose that $Z=Z(\PHI)$  is a non-negative random variable  such that $Z>0$ only if $\PHI$ is satisfiable.
Moreover, suppose that for some density $r>0$ there is a number $C=C(k)>0$ that may depend on $k$ but not on $n$ such that
	\begin{equation}\label{Xeqsmm}
	0<\Erw\brk{Z^2}\leq C\cdot \Erw\brk{Z}^2.
	\end{equation}
We claim that then $\rk\geq r$.
Indeed, the \emph{Paley-Zygmund inequality}
	\begin{equation}\label{XeqPZ}
	\pr\brk{Z>0}\geq\frac{\Erw\brk{Z}^2}{\Erw\brk{Z^2}}
	\end{equation}
implies that
	$\pr\brk{\PHI\mbox{ is satisfiable}}\geq\pr\brk{Z>0}\geq 1/C.$
Because the right hand side remains bounded away from $0$ as $n\ra\infty$,
the following simple consequence of Friedgut's sharp threshold result implies $\rk\geq r$.

\begin{lemma}[\cite{Ehud}]\label{XLemma_Friedgut}
Let $k\geq 3$.
If for some $r$ we have 
	$$\liminf_{n\ra\infty}\pr\brk{\PHI\mbox{ is satisfiable}}>0,$$
then $\rk\geq r-o(1)$.
\end{lemma}

Hence, we ``just'' need to find a random variable that satisfies~(\ref{XeqPZ}).
Let $\cS(\PHI)$ denote the set of satisfying assignments;
then certainly $Z=\abs{\cS(\PHI)}$ is the most obvious choice.
However,  this ``vanilla'' second moment argument turns out to fail spectacularly. 
We need to understand why.

\smallskip
\noindent
{\bf\em Asymmetry and the majority vote.}
The origin of the problem is that $k$-SAT is asymmetric in the following sense.
Suppose that all we know about the random formula $\PHI$ is for each variable $x$ the number $d_x$  of times that $x$
appears as a positive literal in the formula, and the number $d_{\neg x}$ of negative occurrences.
Then our best stab at constructing a satisfying assignment seems to be the ``majority vote'' assigment $\sigmaMAJ$
where we set $x$ to true if $d_x> d_{\neg x}$ and to false otherwise.
Indeed, by maximizing 
the total number of true literal occurrences, of which a satisfying assignment must put one
in every clause, $\sigmaMAJ$ also
 maximizes the probability of being satisfiable.

Our proof of \Thm~\ref{XThm_main} allows us to formalize this observation,
thereby verifying a conjecture from~\cite{yuval}.
Let $\dist(\cdot,\cdot)$ denote the Hamming distance.

\begin{corollary}\label{XCor_asymmetric}
There is a number $\delta=\delta(k)>0$ such that
for $2^k/k<r<\rBP$ \whp\ we have
	\begin{equation}\label{Xeqlean}
	\sum_{\sigma\in\cS(\PHI)}\frac{\dist(\sigma,\sigmaMAJ)}{\abs{\cS(\PHI)}}\leq \bc{\frac12-\delta}\cdot n.
	\end{equation}
\end{corollary}
Hence, the average Hamming distance of $\sigma\in\cS(\PHI)$ from $\sigmaMAJ$ is strictly smaller than $n/2$, i.e.,
the set $\cS(\PHI)$ is ``skewed toward''  $\sigmaMAJ$ \whp\

This asymmetry dooms the second moment method.
To see why,
let
	$$\wMAJ=\wMAJ(\PHI)=\sum_{x\in V}\frac{\max\cbc{d_x,d_{\neg x}}}{km}$$
denote the  \bemph{majority weight} of $\PHI$. 
Then the larger $\wMAJ$, the more likely $\sigmaMAJ$ and assignments close to it are to be satisfying.
In effect, 
the bigger $w_{maj}$, the more satisfying assignments we expect  to have.
The consequence of this is that the number $\abs{\cS(\PHI)}$ of satisfying assignments
behaves like a ``lottery'': its expectation is driven up by a tiny fraction of ``lucky'' formulas with $w_{maj}$  much bigger than expected.

Let us highlight this tradeoff, as it is characteristic of the kind of trouble that asymmetry causes.
For $\xi>0$ independent of $n$ but sufficiently small it turns out that for a certain constant $c>0$,
	\begin{equation}\label{Xeqwmaj1}
	\pr\brk{\wMAJ\sim\Erw\brk{\wMAJ}+\xi}=\exp\brk{-(c\xi^2+O(\xi^3))n}.
	\end{equation}
That is, the probability is exponentially small but,  like in the Chernoff bound,  the exponent is a \emph{quadratic} function of $\xi$.
By comparison,  increasing the majority weight by $\xi$ boosts the expected number of satisfying assignments by
a \emph{linear} exponential factor: there is $c'>0$ such that
	\begin{eqnarray}\label{Xeqwmaj2}
	\Erw\brk{\,|\cS(\PHI)|\ \big|\ \wMAJ\sim\Erw\brk{\wMAJ}+\xi}&=&
		\exp\brk{(c'\xi+O(\xi^2))n}\cdot \Erw\brk{|\cS(\PHI)|\,\big|\,\wMAJ\sim\Erw\brk{\wMAJ}}.
	\end{eqnarray}
The exponent in~(\ref{Xeqwmaj2}) is linear because
for a typical assignment $\tau$ at distance $(\frac12-\delta)n$ from $\sigmaMAJ$
increasing $\wMAJ$ by $\xi$ boosts the number of literals that are true under $\tau$ by $2\delta\xi\cdot km$,
a term that is linear in $\xi$.

Since the exponent is linear in~(\ref{Xeqwmaj2}) but quadratic in~(\ref{Xeqwmaj1}), there is a
 (small but) strictly positive $\xi>0$ such that the ``gain'' $\exp\brk{(c'\xi+O(\xi^2))n}$ in the expected number of satisfying assignments
 exceeds the ``penalty'' $\exp\brk{-(c\xi^2+O(\xi^3))n}$ for deviating from $\Erw\brk{\wMAJ}$.
With little extra work, this observation leads to
  
\begin{lemma}
	\label{XLemma_wmaj}
For any $k\geq 3$ and $r>2^k/k$ 
we have 
	$$\abs{\cS(\PHI)}\leq\exp\bc{-\Omega(4^{-k})\cdot n}\cdot\Erw\brk{\abs{\cS(\PHI)}}
		\quad\mbox{\whp}$$
\end{lemma}

\noindent
\Lem~\ref{XLemma_wmaj} entails rather easily that the ``vanilla'' second moment argument
fails dramatically.
Indeed, as already noticed in~\cite{nae,yuval}, we have
 $\Erw\brk{|\cS(\PHI)|^2}\geq \exp(\Omega(n))\cdot\Erw\brk{|\cS(\PHI)|}^2$.
Hence,  we miss our mark~(\ref{Xeqsmm}) by an exponential factor.
But \Lem~\ref{XLemma_wmaj} is witness to an even worse failure:
	not only does~(\ref{Xeqsmm}) fail to hold, but
	even the normally much more dependable \emph{first} moment overshoots the ``actual'' number of satisfying assignments by
	an exponential factor!
(\Lem~\ref{XLemma_wmaj} is an improvement of an observation from~\cite{Barriers},
	showing that $\abs{\cS(\PHI)}\leq\exp(-\xi n)\Erw\brk{|\cS(\PHI)|}$ \whp\ for some tiny $\xi=\xi(k)>0$;
		we conjecture that the $4^{-k}$ term in \Lem~\ref{XLemma_wmaj} is tight.)

In summary, the drift toward $\sigmaMAJ$ and the resulting fluctuations of the majority weight 
induce a tremendous source of variance, derailing the ``vanilla'' second moment argument.

\smallskip\noindent
{\bf\em Balanced assignments.}
A natural way to sidestep this issue is to work with
a `symmetric' subset of $\cS(\PHI)$.
Perhaps the most obvious choice is the set $\cS_{\mathrm{NAE}}(\PHI)$  of NAE-solutions.
In a landmark paper, Achlioptas and Moore~\cite{nae} proved that indeed 
there is $C=C(k)>0$ such that for $Z_{\mathrm{NAE}}=|\cS_{\mathrm{NAE}}(\PHI)|$ we have
	\begin{equation}\label{XeqAchMoore}
	\Erw\brk{Z_{\mathrm{NAE}}^2}\leq C\cdot\Erw\brk{Z_{\mathrm{NAE}}}^2\quad\mbox{for $r\leq 2^{k-1}\ln2-O_k(1)$}.
	\end{equation}
As we saw above (cf.\ \Lem~\ref{XLemma_Friedgut}), this implies that $\rk\geq 2^{k-1}\ln2-O(1)$.
However, a simple (first moment) calculation shows that for $r>2^{k-1}\ln2$, 
the set $\cS_{\mathrm{NAE}}(\PHI)$ is empty \whp\ 
Thus, the idea of working with NAE-solutions stops working at $r\sim2^{k-1}\ln2$, about a factor of two below
the satisfiability threshold.

Achlioptas and Peres~\cite{yuval} obtained a better bound by
precipitating symmetry in a more subtle manner.
Let us call $\sigma\in\cbc{0,1}^n$ \bemph{balanced} if under $\sigma$ 
out of the $km$ literal occurrences in $\PHI$ \emph{exactly half} are true (i.e., $\frac{km}2\pm1$).
Thus, balanced assignments are expressly forbidden from pandering to $\sigmaMAJ$.
Now, let $\cS_{\mathrm{bal}}(\PHI)$ be the set of all balanced satisfying assignments, and set $Z_{\mathrm{bal}}=|\cS_{\mathrm{bal}}(\PHI)|$.
Achlioptas and Peres used a clever weighting scheme to prove that
	\begin{equation}\label{XeqAP}
	\Erw\brk{Z_{\mathrm{bal}}^2}\leq C\cdot\Erw\brk{Z_{\mathrm{bal}}}^2\qquad\mbox{ for $r\leq\rbal\quad$ (cf.~(\ref{Xeqprevious}))}.
	\end{equation}
As before, this implies that $\rk\geq\rbal$ (\Lem~\ref{XLemma_Friedgut}).

Yet as in the case of NAE-solutions, balanced satisfying assignments cease to exist way before the satisfiability threshold.
Indeed, Achlioptas and Peres observed that $\cS_{\mathrm{bal}}(\PHI)=\emptyset$ for $r>2^k\ln2-k\frac{\ln2}2$ \whp\ 
In effect, to close in further on $\rk$ we will have to accommodate assignments that lean toward $\sigmaMAJ$.
How can this be accomplished?

\smallskip
\noindent
{\bf\em A quick fix?}
We saw that to make an asymmetric second moment argument work,
we need to rule out fluctuations of the majority weight.
A sensible way of implementing this is by actually fixing the entire vector $\vec d=(d_x,d_{\neg x})_{x\in V}$
that counts the positively/negatively occurrences of each variable.
More precisely, 
given a non-negative integer vector $\vec d=(d_x,d_{\neg x})_{x\in V}$ with $\sum_{x\in V}d_x+d_{\neg x}=km$ let
$\PHId$ denote a uniformly random $k$-CNF in which each variable $x$ appears $d_x$ times positively and $d_{\neg x}$
times negatively.
Then we can split the generation of a random formula $\PHI$ into two steps:
\begin{enumerate}
\item[] First, choose the occurrence vector $\vec d$ randomly from the ``correct'' distribution $\vec D$.
\item[]  Then, choose a random formula $\PHId$.
\end{enumerate}

The ``correct'' $\vec D$ is as follows.
Let ${\vec e}=(e_x,e_{\neg x})_{x\in V}$ be a family of independent Poisson variables with mean $kr/2$ each.
Moreover, let $\cE$ be the event that $\sum_{x\in V}e_x+e_{\neg x}=km$.
Let $\vec D$ be the conditional distribution of $\vec e$ given $\cE$.
Then standard arguments show that the outcome of first choosing $\vec d$ and then $\PHId$
is exactly the uniformly random  $\PHI$.

The point of generating $\PHI$ in two steps as above is that given the outcome $\vec d$ of the first step,
the majority weight is \emph{fixed}.
Hence, if we could show that \emph{given} a ``typical'' $\vec d$, 
the second moment succeeds for $|\cS(\PHId)|$
we would obtain a lower bound on $r_{k-SAT}$.
Unfortunately, matters are not so simple.
\begin{lemma}\label{XLemma_PHIdBad}
\Whp\ for a vector $\vec d$ chosen from $\vec D$ we have
	$\Erw[\abs{\cS(\PHId)}^2]\geq\exp\bc{\Omega(n)}\cdot\Erw\brk{\abs{\cS(\PHId)}}^2$.
\end{lemma}

Let us stress the two levels of randomness in \Lem~\ref{XLemma_PHIdBad}.
First, there is the choice of $\vec d$.
Then, for a \emph{given} $\vec d$, we compare $\Erw[\abs{\cS(\PHId)}^2]$ and $\Erw\brk{\abs{\cS(\PHId)}}^2$.
Of course, both of these quantities depend on $\vec d$, and we find that \whp\ $\vec d$ is such that the first exceeds
the second by an exponential factor.

The explanation for this is that even if we fix $\vec d$,
various other types of fluctuations remain, turning $\abs{\cS(\PHId)}$ into a ``lottery''.
For instance, even given $\vec d$ 
the number of clauses that are unsatisfied under $\sigmaMAJ$ fluctuates.
Hence, the inherent asymmetry of $k$-SAT puts not only the majority weight but also
various other parameters on a slippery slope.
What we need  is a way of controlling all these fluctuations simultaneously.
We will present our solution in \Sec~\ref{XSec_theRandomVar}.

\smallskip
\noindent{\bf\em Catching the $k$-SAT threshold?}
Before we come to that, let us discuss what it would take to eliminate the (small but non-zero) gap left by
\Thm~\ref{XThm_main}, i.e., how far we are from ``catching'' the $k$-SAT threshold.
The physicists' cavity method comes in two installments.
The (relatively speaking) simpler ``replica symmetric'' version is based on Belief Propagation.
\Thm~\ref{XThm_main} provides a rigorous proof of the best possible bound on the $k$-SAT threshold
that can be obtained from this version of the cavity method (up to possibly the precise error term $\eps_k$)~\cite{pnas}.

Unfortunately, for $r>\rBP$ the replica symmetric version (and in particular the Belief Propagation predictions
that we depend upon) are conjectured to break down.
 According to the more sophisticated ``1-step replica symmetry breaking'' (1RSB) version of the cavity method, 
the reason for this is that at $r\sim\rBP$ 
a new type of correlation amongst satisfying assignments arises.
To deal with these correlations, the physics methods replace Belief Propagation by the \emph{much} more intricate Survey Propagation technique.

In~\cite{Catching} we managed to prove rigorously that the 1RSB prediction for the random $k$-NAESAT
threshold is correct (up to an additive $2^{-\Omega(k)}$).
However, \cite{Catching} depends \emph{heavily} on the fact that $k$-NAESAT is symmetric.
While it would be very interesting to combine the merits of the present paper with those of~\cite{Catching},
this appears to be quite challenging.
Thus, putting the 1RSB calculation for random $k$-SAT on a rigorous foundation
remains an important open problem.
That said, we believe that any such attempt would need to build upon the techniques developed in this paper.

\section{Related work}

The interest in random $k$-SAT originated largely from the \emph{experimental} observation
that there seems to be a sharp threshold for satisfiability and, moreover,  that for certain densities $r<r_{k-SAT}$
no polynomial time algorithm is known to find a satisfying assignment \whp~\cite{KirkpatrickSelman,Kroc}.
Currently, the precise $k$-SAT threshold is known (rigorously) only  in two cases.
Chvatal and Reed~\cite{mick} and Goerdt~\cite{Goerdt} proved independently that $r_{2-\mathrm{SAT}}=1$.
Of course, $2$-SAT is special because there is a simple criterion for (un)satisfiability,
which enables the proofs of~\cite{mick,Goerdt}.
Unsurprisingly, these methods do not extend to $k>2$.
Additionally, the threshold is known precisely  when $k>\log_2n$, i.e.,
the clause length \emph{diverges} as a function of $n$~\cite{FriezeWormald}.
In this case, the problem of asymmetry evaporates because 
the majority weight is sufficiently concentrated for the ``vanilla'' second moment method to succeed.
(Note that \Prop~\ref{XLemma_wmaj} holds for any \emph{fixed} $k$, but not for $k=k(n)\ra\infty$.)
The issue of asymmetry also disappears in the case of \emph{strongly regular} formulas~\cite{Rathi} where
for some fixed $d$ we have $d_x=d_{\neg x}=d$ for all $x\in V$.

Also in random $k$-XORSAT (random linear equations mod 2) the threshold for the existence of solutions is known precisely~\cite{Dubois}.
The proof relies on computing the second moment of the number of solutions (after the instance has been stripped down to a suitable core).
In contrast to random $k$-SAT, the random $k$-XORSAT problem is symmetric
	(cf.\ Remark~\ref{XRem_sym} below), albeit in a more
subtle way than $k$-NAESAT.

Other problems where the second moment method succeeds are symmetric as well.
Pioneering the use of the second moment method in random CSPs,
Achlioptas and Moore~\cite{nae}
computed the random $k$-NAESAT threshold within an additive $1/2$.
By enhancing this argument with insights from physics  this gap
can be narrowed to a mere $2^{-\Omega(k)}$~\cite{Catching,Lenka}.
Moreover, the best current bounds on the random (hyper)graph $k$-colorability thresholds are based
on ``vanilla'' second moment arguments as well~\cite{AchNaor,DyerFriezeGreenhill}.
In summary, in all the previous second moment arguments,
the issue of asymmetry either did not appear at all by the nature of the
	problem~\cite{nae,AchNaor,Catching,Lenka,Dubois,DyerFriezeGreenhill,FriezeWormald},
or it was sidestepped~\cite{yuval}.

The best current algorithms for random $k$-SAT find satisfying assignments \whp\
for densities up to $1.817\cdot 2^k/k$ (better for small $k$) resp.\ $2^k\ln(k)/k$ (better for large $k$)~\cite{BetterAlg,FrSu},
a factor of $\Theta(k/\ln k)$ below the satisfiability threshold.
By comparison, the Lov\'asz Local Lemma and its algorithmic version 
succeed up to $r=\Theta(2^k/k^2)$~\cite{Moser}.

Apart from experimental work~\cite{Kroc}, 
 very little is known about the physics-inspired message passing algorithms (``Belief/Survey Propagation guided decimation'')~\cite{MPZ}.
The most basic variant of Belief Propagation guided decimation is known to fail \whp\ on random formulas if $r>c\cdot 2^k/k$
for some constant $c>0$~\cite{BP}.
However, it is conceivable that Survey Propagation and/or other variants of Belief Propagation perform better.

\section{Preliminaries}

We shall make repeated use of the following local limit theorem for the sums of independent random variables, see~\cite{FlajSed} and~\cite{Catching}.
\begin{lemma}
\label{lem:locallimit}
Let $X_1, \dots, X_n$ be independent random variables with support on $\mathbf{N}_0$ with probability generating function $P(z)$. Let $\mu = \Exp[X_1]$ and $\sigma^2 = \Var[X_1]$. Assume that $P(z)$ is an entire and aperiodic function. 
Then, uniformly for all $T_0 < \alpha < T_\infty$, where $T_x = \lim_{z\to x}\frac{zP'(z)}{P(z)}$,  as $n \to \infty$
\begin{equation}
\label{eq:locallimit}
	\Pr[X_1 + \dots + X_n = \alpha n] = (1 + o(1))\, \frac{1}{\zeta \sqrt{2\pi n \xi}} \, \left(\frac{P(\zeta)}{\zeta^\alpha}\right)^n, 
\end{equation}
where $\zeta$ and $\xi$ are the solutions to the equations
\begin{equation}
\label{eq:saddlepoints}
	\frac{\zeta P'(\zeta)}{P(\zeta)} = \alpha
	\qquad \text{ and } \qquad
	\xi = \frac{d^2}{dz^2}\left(\ln P(z) - \alpha \ln z\right) \Big|_{z = \zeta}.
\end{equation}
Moreover, there is a $\delta_0 > 0$ such that for all $0 \le |\delta| \le \delta_0$ the following holds. If $\alpha = \Exp[X_1] + \delta\sigma$, then
\begin{equation}
\label{eq:limitclose}
	\Pr[X_1 + \dots + X_n = \alpha n] = (1 + O(\delta)) \, \frac{1}{\sqrt{2\pi n \sigma}} \, e^{(-\delta^2/2 + O(\delta^3))n}.
\end{equation}
\end{lemma}

\noindent
From this we can rather easily derive the following well-known statement about the rate function of the binomial distribution.
\begin{lemma}\label{Lemma_binomial}
Let $0<p,q<1$.
Let
	\begin{eqnarray*}
	\psi(p,q)&=&
		-q\ln\bcfr qp-(1-q)\ln\bcfr{1-q}{1-p},
	\end{eqnarray*}
If $p,q$ remain fixed as $n\ra\infty$, then
	$$\pr\brk{\Bin(n,p)=qn}=\Theta(n^{-1/2})\exp\brk{\psi(p,q)n}.$$
\end{lemma}

\noindent
The following form of the chain rule will prove useful.
\begin{lemma}\label{Lemma_chainrule}
Let $g:\RR^a\ra\RR^b$ and $f:\RR^b\ra\RR$ be of class $C^2$, i.e, with continuous second derivatives.
Then for any $x_0\in\RR^a$ and with $y_0=g(x_0)$ we have for any $i,j\in\brk a$
\begin{eqnarray*}
\frac{\partial^2f\circ g}{\partial x_i\partial x_j}\bigg|_{x_0}&=&
	\sum_{k=1}^b\frac{\partial f}{\partial y_k}\bigg|_{y_0}\frac{\partial^2 g_k}{\partial x_i\partial x_j}\bigg|_{x_0}+
		\sum_{k,l=1}^b\frac{\partial^2f}{\partial y_k\partial y_l}\bigg|_{y_0}\frac{\partial g_k}{\partial x_i}\bigg|_{x_0}\frac{\partial g_l}{\partial x_j}\bigg|_{x_0}.
\end{eqnarray*}
\end{lemma}

\noindent
Finally, we need the following version of the inverse function theorem that states under which conditions a given system of equations can be solved around a specific point $u$. A detailed exposition can be found in~\cite{Spivak}. 

\begin{lemma}\label{Lemma_inverseFunction}
Let $U\subset\RR^h$ be open and let $f\in C^1(U)$.
Assume that $u\in U$ and $\lambda>0$ are such that
	$$\cbc{x\in\RR^h:\norm{x-u}_2\leq\lambda}\subset U.$$
Let $Df(x)$ be the Jacobian matrix of $f$ at $x$, $\id$ the identity matrix, and $\norm{\cdot}$ denote the operator norm over $L^2(\RR^h)$. Assume that $Df(u)=\id$ and
	$$\norm{Df(x)-\id}\leq\frac13\qquad\mbox{ for all $x\in\RR^h$ such that }\norm{x-u}_2\leq\lambda,$$
Then for each $y\in\RR^h$ such that $\norm{y-f(u)}\leq \lambda/2$ there is precisely one $x\in\RR^h$ such that $\norm{x-u}\leq r$
and $f(x)=y$.
Furthermore, the inverse map $f^{-1}$ is $C^1$ on $\cbc{x\in\RR^h:\norm{x-u}_2<\lambda}$, and
	$Df^{-1}(x)=(Df(x))^{-1}$ on this set.
\end{lemma}

\smallskip
\noindent
{\bf Notation.}
We will generally assume that $n>n_0,k>k_0$ for certain large enough constants $n_0,k_0$.
We are going to use the asymptotic symbols $O(f(x))$, $\Omega(f(x))$, etc.
It is understood that the asymptotic is with respect to the parameter $x$ of the function $f(x)$.
Thus, if $f$ is a function of $n$, then the asymptotic notation refers to the limit $n\ra\infty$,
and if $f$ is a function of $k$, then the notation refers to $k$ being large.
We use the following convention for the $O$-notation in the case that $f$ is a constant:
we let $O(1)$ be a term that remains bounded in the limit of large $n$, but that may by unbounded in terms of $k$.
By constrast, $O_k(1)$ refers to a term that remains bounded both in the limit of large $k$ and large $n$.
Expressions such as $o_k(1)$ are to be interpreted analogously.
Generally, all asymptotics are \emph{uniform} in the various other parameters (such as the degree sequence $\vec d$ or $r$).
For a function $f(k)>0$ use the symbol $\tilde O(f(k))$ to denote a function $g(k)$ such that for some constant $c>0$ we have
	$g(k)=O(f(k)\cdot\ln^c f(k))$.
For vectors $\xi,\eta$ we use the symbol
	$$\eta\doteq\xi$$
to denote the fact that $\norm{\xi-\eta}_\infty\leq O(1/n)$.

Let $V=\cbc{x_1,\ldots,x_n}$ and let $L=\cbc{x_1,\neg x_1,\ldots,x_n,\neg x_n}$.
For a literal $l\in L$ we let $\abs l$ denote the underlying variable.
Moreover, $\sign(l)=1$ if $l$ is a positive literal, and $\sign(l)=-1$ otherwise.
For a $k$-CNF $\Phi$ we let $\Phi_i$ denote the $i$th clause of $\Phi$
and $\Phi_{ij}$ the $j$th literal of $\Phi_i$.

From here on out, we let
\begin{equation}
\label{eq:rrho}
	r=2^{-k}\ln2-\rho \quad\text{with}\quad \rho=\frac32\ln2-\eps_k
\end{equation}
for some sequence $\eps_k=o_k(1)$ that tends to $0$ sufficiently slowly.

\section{The random variable}\label{XSec_theRandomVar}

\subsection{The construction}

Our goal is to
make the second moment method work for
 a random variable that counts ``asymmetric'' satisfying assignments.
In this section, we develop this random variable.
The starting point, and the key ingredient, is simply a map $\mymap:\ZZ\ra\brk{0,1}$.
For the sake of clarity, we start by setting up the framework for generic maps $p$;
	below we will use the Belief Propagation formalism to pick the ``optimal'' $p$.

The idea is that $p$ prescribes how strongly the assignments that we work with lean toward the majority vote.
Informally speaking, we are going to work with assignments such that
a variable $x$ that occurs $d_x$ times positively and $d_{\neg x}$ times negatively has
a $\mymap(d_x-d_{\neg x})$ chance of being set to `true'.
Before we give a formal definition, we need to fix the number of times that each variable appears positively or negatively.

\smallskip\noindent
{\bf\em Fixing the majority weight.}
As we saw in \Sec~\ref{XSec_bla}, in order to make the second moment argument work,  we need to rule out fluctuations of the majority weight. 
To achieve this, we follow the strategy outlined in \Sec~\ref{XSec_bla}.
That is, we are going to work with formulas $\PHId$ with a given vector $\vec d=(d_x,d_{\neg x})_{x\in V}$
of occurrence counts, where each variable $x$ appears precisely $d_x$ times positively
and $d_{\neg x}$ times negatively.
As in \Sec~\ref{XSec_bla}, we let $\vec D$ denote the (conditional Poisson) distribution over sequences $\vec d$ such that
first choosing $\vec d$ from $\vec D$ and then generating $\PHId$ is equivalent to choosing a $k$-CNF $\PHI$ uniformly at random.

\smallskip\noindent
{\bf\em Fixing the marginals.}
Now, fix one such vector $\vec d$.
Then the map $\mymap:\ZZ\ra\brk{0,1}$ induces a map
	$p_{\vec d}$ from the set
	$L=\cbc{x,\neg x:x\in V}$
	 of literals to $\brk{0,1}$ in the natural way: we let
	\begin{equation}\label{Xeqpd}
	\pd(x)=\mymap(d_x-d_{\neg x})\mbox{ and }\pd(\neg x)=1-\mymap(x).
	\end{equation}
The idea is that, given $\vec d$, we should set variable $x$ to `true' with probability $\pd(x)$.

To formalize this, 
we call $\pd(l)$ the \bemph{$\pd$-type} of the literal $l$.
Let 
 $\cT=\cT_{\vec d}=\cbc{\pd(l):l\in L}$ be the set of all possible $\pd$-types.
We say that $\sigma:V\ra\cbc{0,1}$ has \bemph{$\pd$-marginals}
	if for any type $t\in\cT_{\vec d}$ we have%
	$$\sum_{l\in L:\pd(l)=t}(\sigma(l)-t)\cdot d_l=O(1).$$
i.e., among all occurrences of literals of type $t$, 
a $t$ fraction is true under $\sigma$.
This definition captures the above idea that variable $x$ has a $\pd(x)$ chance of being `true'.

\smallskip\noindent
{\bf\em Fixing the clause types.}
We define the \bemph{$\pd$-type} of a clause $l_1\vee\cdots\vee l_k$ as the $k$-tuple
$(\pd\bc{l_1},\ldots,\pd\bc{l_k})\in\brk{0,1}^k$
comprising of the individual literal types.
Let $\cL=\cL_{\vec d}=\cT_{\vec d}^k$ be the set of all possible clause types.
For each $\ell\in\cL_{\vec d}$ let $M_{\PHId}(\ell)$ be the set of indices $i\in\brk m$ such that
the $i$th clause of $\PHId$ has type $\ell$, and let $m_{\PHId}(\ell)=\abs{M_{\PHId}(\ell)}$.

In addition to fluctuations of the majority weight, we also need to suppress fluctuations of the numbers $m_{\PHId}(\ell)$.
We are going to use the same trick as in the case of the majority weight.
Namely, we split the generation of a random formula $\PHId$ into two steps:
\begin{enumerate}
\item[]	First, choose a vector $\vec m=(m(\ell))_{\ell\in\cL}$ from the ``correct'' distribution $\vec M_{\vec d}$.
\item[]	Then, generate a formula $\PHIdm$ uniformly at random in which 
each variable $x$ appears exactly $d_x$ times positively and exactly  $d_{\neg x}$ times negatively
and that has exactly $m(\ell)$ clauses of type $\ell$ for all $\ell\in\cL$.
\end{enumerate}
Formally, the ``correct'' $\vec M_{\vec d}$ is
just the distribution of the random vector $\vec m_{\PHId}=(m_{\PHId}(\ell))_{\ell\in\cL}$
that counts the clauses by types in the ``unrestricted'' formula $\PHId$.
It is easily verified that the overall outcome of the above experiment is identical to $\PHId$.
From now on, we fix both $\vec d$ and $\vec m$.

Given $\vec d,\vec m$ there is a simple way of generating the random formula $\PHIdm$.
Namely, create $d_l$ clones of each literal $l$, and put all the clones of a given $\pd$-type on a pile.
Then the formula $\PHIdm$ is simply the result of matching the clones on the type $t$ pile 
randomly to all  the clauses where a literal of type $t$ is required.

An assignment $\sigma$ with $p_{\vec d}$-marginals splits each pile into two subsets,
namely the clones that are true under $\sigma$ and those that are false.
For each type $t$, among the clones in the type $t$ pile, a $t$-fraction are true, since $\sigma$ has $p_{\vec d}$-marginals. Therefore, we expect that under the random matching, for each clause type $\ell=(\ell_1, \dots,\ell_k)$ and each index $j$, in an $\ell_j$-fraction of  clauses the $j$th literal is matched to a `true' clone.

\smallskip\noindent
\noindent{\bf\em Judicious assignment.}
This observation motivates the following definition.
We say that an assignment $\sigma$ is 
\bemph{$\pd$-judicious} in $\PHIdm$
if for all clause types $\ell=(\ell_1,\ldots,\ell_k)\in\cL$ and all
$j\in\brk k$ we have 
	\begin{equation}\label{Xeqjudicious}
	\sum_{i\in M_{\PHIdm}(\ell)}\sigma(\PHI_{\vec d,\vec m,i,j})=m(\ell)\cdot \ell_j+O(1),
	\end{equation}
where $\PHI_{\vec d,\vec m,i,j}$ denotes the $j$th literal of the $i$th clause of $\PHIdm$,
and the sum is over all $i$ such that the $i$th clause has type $\ell$.
Let $\cS_p(\PHIdm)$ be the set of $p$-judicious satisfying assignments, and set $Z_p(\PHIdm)=\abs{\cS_p(\PHIdm)}$.

Given that $\sigma$ is $p$-judicious, 
in order for $\sigma$ to be satisfying we just need that for each type $\ell$ the `true' clones
are distributed so that each clause receives at least one.
Thus, the event of being satisfying is merely a matter of how exactly the `true' clones
are ``shuffled''  amongst the clauses of type $\ell$, while
for each $j$ the total number of `true' clones of type $\ell_j$ is fixed.
In particular, this shuffling occurs independently for each clause type.
Such random shuffling problems tend to be amenable to the second moment method.
Therefore, it seems reasonable to expect that a second moment argument succeeds for $Z_p(\PHIdm)$.
This is indeed the case for $r<\rBP-1+\ln2\approx\rBP-0.3$.
However,  to actually reach $\rBP$ we need to control one further parameter.

\smallskip\noindent
{\bf\em Fixing the cluster size.}
According to the physics predictions~\cite{pnas,MPZ},
for $\rbal<r<\rBP$ the set of satisfying assignments decomposes into an exponential number of well-separated `clusters'.
More precisely, we expect that  \whp\ for any two satisfying $\sigma,\tau$ either $\dist(\sigma,\tau)<0.01n$
(if $\sigma,\tau$ belong to the same cluster), or $\dist(\sigma,\tau)>0.49n$ (different clusters).
Formally, we simply define the \bemph{cluster of $\sigma$} as
	$$\cC_\sigma(\Phi)=\cbc{\tau\in\cS(\PHIdm):\frac{\dist(\sigma,\tau)}n\not\in\brk{\frac12-k^22^{-k/2},\frac12+k^22^{-k/2}}}.$$
The intuitive reason why the second moment argument for $Z_p(\PHIdm)$ breaks down for $r$ close to $\rBP$ is that
the cluster sizes $|\cC_{\sigma}(\PHIdm)|$ fluctuate.
A similar problem occurred in prior work on random $k$-NAESAT~\cite{Catching,Lenka}.

As in those papers, the problem admits a remarkably simple solution:
let us call an assignment $\sigma$ \bemph{good} in $\PHIdm$ if	
	\begin{equation}\label{Xeqgood}
	\abs{\cC_\sigma(\PHIdm)}\leq\Erw\brk{Z_p(\PHIdm)}.
	\end{equation}
Let $\cS_{p,\mathrm{good}}(\PHIdm)$ be the set of all good $\sigma\in\cS_p(\PHIdm)$.
To avoid fluctuations of the cluster size, we are just going to work with
	$Z_{p,\mathrm{good}}=|\cS_{p,\mathrm{good}}(\PHIdm)|$.

\smallskip\noindent
\noindent{\bf\em The second moment bound.}
We now face the task of estimating the first and the second moment of $Z_{p,\mathrm{good}}(\PHIdm)$.
The result can be summarized as follows.

\begin{theorem}\label{XThm_secondPlain}
Suppose $\rbal<r<\rBP$.
There exists $C=C(k)$ 
and a map $\mymap=\mymap_{\mathrm{BP}}:\ZZ\ra\brk{0,1}$
such that for $\vec d$ chosen from $\vec D$ 
and for $\vec m$ chosen from $\vec M_{\vec d}$
\whp\
		$$0<\Erw\brk{Z_{p,\mathrm{good}}(\PHIdm)^2}\leq C\cdot \Erw\brk{Z_{p,\mathrm{good}}(\PHIdm)}^2.$$
\end{theorem}

Together with Paley-Zygmund~(\ref{XeqPZ}),
	\Thm~\ref{XThm_secondPlain} shows that with $\vec d$ chosen from $\vec D$ and $\vec m$ chosen from $\vec M_{\vec d}$
	\whp
	\begin{eqnarray}			\label{Xeqfinish1}
	\pr\brk{\PHIdm\mbox{ is satisfiable}}
	&\geq&
	\pr\brk{Z_{p,\mathrm{good}}(\PHIdm)>0}\geq
			\frac{\Erw\brk{Z_{p,\mathrm{good}}(\PHIdm)}^2}{\Erw\brk{Z_{p,\mathrm{good}}(\PHIdm)^2}}
			\geq\frac1{C}.
	\end{eqnarray}
The construction of $\vec D$, $\vec M_{\vec d}$ ensures that
choosing $\PHI$ at random is the same as first picking $\vec d$ from $\vec D$ and $\vec m$ from $\vec M_{\vec d}$
and then generating $\PHIdm$.
Therefore, (\ref{Xeqfinish1}) implies 
$\liminf_{n\ra\infty}\pr\brk{\PHI\mbox{ is satisfiable}}>0$,
so that \Lem~\ref{XLemma_Friedgut} yields $\rk\geq\rBP$.
Hence, we are left to prove \Thm~\ref{XThm_secondPlain}.
We begin by constructing the map $p_{\mathrm{BP}}$.

\smallskip\noindent
{\bf\em Guessing the marginals.}
For a set $\emptyset\neq S\subset\cbc{0,1}^V$ and a variable $x$ we define the \bemph{$S$-marginal of $x$} as
	\begin{equation}\label{XeqMarginals}
	\mu_S(x)=\sum_{\sigma\in S}\frac{\sigma(x)}{|S|}.
	\end{equation}
The definition of `$\pd$-judicious' is guided by the idea that 
 $\pd(x)$ should prescribe 
the marginal
of $x$ in the set of all $\pd$-judicious satisfying assignments.
Hence, 
in order to make the set of $\pd$-judicious assignments as good an approximation
of the \emph{entire} set of satisfying assignments as possible,
we better pick $p$ so that $\pd(x)$
is a good approximation to the actual marginal $\mu_{\cS(\PHId)}(x)$
of $x$ in the set of \emph{all} satisfying assignments.
The problem is that, because of the asymmetry of the $k$-SAT problem, these marginals are highly non-trivial quantities.
Indeed, on general formulas $\Phi$ the marginals $\mu_{\cS(\Phi)}(x)$ are $\#P$-hard to compute.

However, according to the physicists' 
cavity method,
on random formulas with density $r<\rBP$ the marginals can be computed by means of an efficient message passing algorithm
called \emph{Belief Propagation}~\cite{pnas}.
While the mechanics of this are not important in our context, the result is.

\begin{conjecture}\label{XConj_marg}
Suppose that $\rbal<r<\rBP$.
Let $\vec d$ be chosen from $\vec D$ 
and let $x$ be a variable.
Then  \whp\ 
	\begin{equation}\label{Xeqmargs}
	\mu_{\cS(\PHId)}(x)
		=	
			\frac12+\frac{d_x-d_{\neg x}}{2^{k+1}}+O\bc{\frac{d_x-d_{\neg x}}{2^{k}}}^2.
	\end{equation}
\end{conjecture}
We observe that~(\ref{Xeqmargs}) is in line with the notion that $\cS(\PHId)$ is ``skewed toward'' $\sigmaMAJ$.
Indeed, the conjecture quantifies how much so.
Motivated by Conjecture~\ref{XConj_marg},  we define 
	\begin{equation}\label{XeqpBP}
	\mymap_{\mathrm{BP}}(z)=\left\{\hspace{-2mm}\begin{array}{cl}
		\displaystyle\frac12+\frac{z}{2^{k+1}}&\hspace{-2mm}\mbox{if }|z|\leq 10\sqrt{k2^k\ln k},\\[2mm]
		\displaystyle\frac12&\hspace{-2mm}\mbox{otherwise.}
		\end{array}\right.
	\end{equation}
Under the distribution $\vec D$, the random variables $d_x,d_{\neg x}$ are asymptotically independent
Poisson with mean $kr/2$ (cf.\ \Sec~\ref{XSec_bla}).
Therefore,
	$$\Erw_{\vec d}\brk{(d_x-d_{\neg x})^2}=kr\leq k2^k\ln2,$$
and standard concentration inequalities show that \whp\ there are no more than $n/k^{30}$
variables $x$ with $(d_x-d_{\neg x})^2>100k2^k\ln k$.
Hence, $\pd=p_{\mathrm{BP},\vec d}$ is (asymptotically) equal to the conjectured value on 
the bulk of variables \whp

In summary, 
the problem with the ``vanilla'' second moment argument is that the drift toward $\sigmaMAJ$
 induces correlations amongst the satisfying assignments.
Indeed, they are correlated with the majority assignment and thus with each other.
We circumvent this problem by 
explicitly prescribing the marginal probability that each variable is set to `true'.
One could think of this as working with the intersection of $\cS(\PHI)$ with a 
particular ``surface'' within the Hamming cube $\cbc{0,1}^n$, namely the assignments with $\pd$-marginals.
Within this surface,
all assignments are slanted equally toward $\sigmaMAJ$.
The Belief Propagation-informed definition of $p_{\mathrm{BP}}$ is meant to ensure
that the surface that we consider with is (about) the most populous one, i.e., the one with the largest number of satisfying assignments in it.
The core of our argument will be to show that 
\emph{with respect to the marginal distribution $p_{\mathrm{BP}}$}, i.e., within the surface 
that $p_{\mathrm{BP}}$ defines, two random elements of $\cS_p(\PHIdm)$
are typically uncorrelated.
But before we come to that, we need to compute the ``first moment'',
i.e., the  expected number of good $p_{BP}$-judicious satisfying assignments.

\begin{remark}\label{XRem_BP}
Belief Propagation actually leads to a stronger prediction than Conjecture~\ref{XConj_marg}.
Namely, it yields a conjecture for $\mu_{\cS(\PHId)}(x)$ up to an additive error then tends to $0$ as $n\ra\infty$.
However, (a) this stronger conjecture is not in explicit form, and (b) it does not only depend
on $d_x,d_{\neg x}$, but also on various other parameters.
In any case, even a more accurate prediction would not yield a better constant than $\frac32\ln2$ in \Thm~\ref{XThm_main}.
\end{remark}

\begin{remark}\label{XRem_balanced}
In the present framework, the notion of balanced satisfying assignments from~\cite{yuval} simply corresponds to
working with the constant map $p_{bal}:\ZZ\ra\brk{0,1},\ z\mapsto\frac12$.
This hightlights that the improvement that we obtain here stems from choosing
the non-constant map $p_{\mathrm{BP}}$ inspired by Belief Propagation.
\end{remark}

\begin{remark}\label{XRem_sym}
The definition~(\ref{XeqMarginals}) of the marginal of a set gives rise to a formal notion of `symmetric problem'.
Namely, we could call a (binary) random CSP \bemph{symmetric} if its set $\cS_{\mathrm{CSP}}(\PHI)$ of solutions
is such that for each variable $x$ \whp\ we have $\mu_x(\cS_{\mathrm{CSP}}(\PHI))=\frac12+o(1)$.
Clearly, $k$-NAESAT passes this test as $\mu_x(\cS_{\mathrm{NAE}}(\PHI))=\frac12$ for all $x$ with certainty.
Similarly, the problem of having a 
balanced satisfying assignment is symmetric~\cite{yuval}, as is 
random $k$-XORSAT.
\end{remark}

{\bf From here on out we keep the assumptions of \Thm~\ref{XThm_secondPlain}.
In particular, we assume $\rbal<r<\rBP$.
Let $\vec d$ be chosen from $\vec D$, and let $\vec m$ be chosen from $\vec M_{\vec d}$.
Let $p=p_{\mathrm{BP}}$ be as in~(\ref{XeqpBP})
and $\pd$ as in~(\ref{Xeqpd}).}

\subsection{Typical degree sequences}

We need to collect a few basic properties of the sequence $\vec d$ chosen from $\vec D$.
Let us call
a sequence $\vec d=(d_l)_{l\in L}$ of non-negative integers such that
	$\sum_{l\in L}d_l=km$  a \bemph{signed degree sequence}.
For a $k$-CNF $\Phi$ let $\vec d\bc{\Phi}=(d_{l}\bc\Phi)_{l\in L}$ denote the vector whose entry $d_{l}\bc\Phi$ is equal to
the number of times that literal $l$ occurrs in $\Phi$.
Then $\vec D=\vec D_k(n,m)$ is just the distribution of the signed degree sequence $\vec d(\PHI)$.

The \emph{signature} of a literal $l\in L$ with respect to a signed degree sequence $\vec d$ is the triple $(\sign(l),d_{\abs l},d_{\neg{\abs l}})$.
We omit the reference to $\vec d$ if it is clear from the context.
Let $T=T(\vec d)$ be the set of all possible signatures.
For each literal $l$ we let $T(l)$ denote its signature.
Furthermore, for a signature $\theta=(\sign(l),d_{\abs l},d_{\neg{\abs l}})\in T$ we let $\neg\theta=(-\sign(l),d_{\abs l},d_{\neg{\abs l}})$.

Let $\vec d$ be a signed degree sequence.
A $k$-CNF $\Phi$ over $V$ is \bemph{$\vec d$-compatible} if 
$\vec d(\Phi)=\vec d$.
Thus, 
	$$\PHId=\PHI_{\vec d,1}\wedge\cdots\wedge \PHI_{\vec d,m}$$
is a uniformly random $\vec d$-compatible $k$-CNF.

In the sequel we are going to prove statements about the random formula $\PHId$ for a ``typical'' signed degree sequence $\vec d$.
Formally, this means that we first choose $\vec d$ from the distribution $\vec D$ at random.
Then, conditioning on $\vec d$, we will study the random formula $\PHId$.
Thus, there are \emph{two levels} of randomness: the distribution of $\vec d$ and then, given $\vec d$,
the choice of the random formula $\PHId$.
When referring to the random choice of $\vec d$ we use the notation $\pr_{\vec d}\brk\cdot$, $\Erw_{\vec d}\brk\cdot$.
By contrast, if we choose $\PHId$ randomly for $\vec d$ fixed, then we use $\pr\brk\cdot$, $\Erw\brk\cdot$.

\begin{lemma}\label{Lemma_danneal}
\begin{enumerate}
\item Let $\cE$ be an event such that
		$\pr\brk{\PHI\in\cE}=o(1)$.
	Then \whp\ a signed degree sequence $\vec d$ chosen from the distribution $\vec D$ is such that
		$\pr\brk{\PHId\in\cE}=o(1)$.
	Conversely, if \whp\ for a random $\vec d$ chosen from $\vec D$ we have
		$\pr\brk{\PHId\in\cE}=o(1)$,
	then $\pr\brk{\PHI\in\cE}=o(1)$.
\item For any random variable $X\geq0$ and any $\eps>0$ we have
		$\pr_{\vec d}\brk{\Erw\brk{X(\PHId)}>\Erw\brk{X(\PHI)}/\eps}\leq\eps.$
\end{enumerate}
\end{lemma}
\begin{proof}
The first claim follows from Markov's inequality as
	$\pr\brk{\PHI\in\cE}=\Erw_{\vec d}\brk{\pr\brk{\PHId\in\cE}}$.
The second claim follows from from Markov's inequality as well because
	$\Erw\brk{X(\PHI)}=\Erw_{\vec d}\brk{\Erw\brk{X(\PHId)}}$.
\qed\end{proof}

\begin{lemma}\label{Lemma_tame}
For $\vec d$ chosen from $\vec D$ the following statements hold \whp
\begin{enumerate}
\item 	$\sum_{x\in V} (d_x-d_{\neg x})^2\sim km.$
\item	 $\frac1n\sum_{x\in V}|d_x-d_{\neg x}|=\tilde O(2^{k/2})$.
\item Let $\cM$ contain the $n$ literals of largest degree.
		Then
		$\frac1{km}\sum_{l\in\cM}d_l=\frac{1}2+\tilde O(2^{-k/2})$.
\end{enumerate}
\end{lemma}
\begin{proof}
We use the following description of the distribution $\vec D$.
Let $\vec e=(e_l)_{l\in L}$ be a family of indepedent $\Po(kr/2)$ variables.
Moreover, let $\cE$ be the event that $\sum_{l\in L}e_l=km$.
It is well known that $\vec e$ given $\cE$ has distribution~$\vec D$.
Furthermore, a simple calculation based on Stirling's formula yields
	\begin{equation}\label{eqDtame1}
	\pr\brk{\cE}=\Theta(n^{-1/2}).
	\end{equation}
Let $\hat e_l=\min\cbc{e_l,\ln^2n}$.
Employing Stirling's formula once more, we find that $\pr\brk{\hat e_l\neq e_l}\leq n^{-10}$ for all $l\in L$.
Hence, by the union bound,
	\begin{equation}\label{eqDtame2}
	\pr\brk{\forall l\in L:\hat e_l= e_l}\geq1-n^{-9}.
	\end{equation}
Furthermore, as $e_x,e_{\neg x}$ are independent for any $x\in V$, we have
	\begin{equation}\label{eqDtame3}
	\Erw\brk{(\hat e_x-\hat e_{\neg x})^2}=2\Var(\hat e_x)= 2\Var(e_x)+O(n^{-1})=kr+O(n^{-1}).
	\end{equation}
Because $\hat e_l\leq\ln^2n$ and the random variables $\cbc{(\hat e_x-\hat e_{\neg x})^2}_{x\in V}$ are mutually independent,
Azuma's inequality yields
	\begin{equation}\label{eqDtame4}
	\pr\brk{\abs{\sum_{x\in V}(\hat e_x-\hat e_{\neg x})^2-\Erw\sum_{x\in V}(\hat e_x-\hat e_{\neg x})^2}>n^{2/3}}
		\leq2\exp\brk{-\frac{n^{1/3}}{8\ln^8n}}\leq n^{-10}.
	\end{equation}
Combining (\ref{eqDtame1})--(\ref{eqDtame4}), we find
	\begin{eqnarray*}
	\pr_{\vec d}\brk{\abs{\sum_{x\in V} (d_x-d_{\neg x})^2-km}>n^{3/4}}&=&
		\pr\brk{\abs{\sum_{x\in V} (e_x-e_{\neg x})^2-km}>n^{2/3}\,\bigg|\,\cE}\\
		&\leq&\Theta(n^{1/2})\pr\brk{\abs{\sum_{x\in V} (e_x-e_{\neg x})^2-km}>n^{3/4}}\\
		&\leq&o(1)+\Theta(n^{1/2})\pr\brk{\abs{\sum_{x\in V} (\hat e_x-\hat e_{\neg x})^2-\Erw\sum_{x\in V} (\hat e_x-\hat e_{\neg x})^2}>n^{2/3}}\\
		&=&o(1),
	\end{eqnarray*}
thereby proving the first claim. The second claim follows from the first by means of the Cauchy-Schwarz inequality: \whp
	$$\brk{\frac1n\sum_{x\in V}|d_x-d_{\neg x}|}^2
		\leq\frac1n\sum_{x\in V}(d_x-d_{\neg x})^2\sim kr.$$
Finally, the third assertion is immediate from the second.
\qed\end{proof}

\noindent For a set $S\subset L$ we let $\Vol(S)=\Vol_{\vec d}(S)=\sum_{l\in S}d_l$.
\begin{lemma}\label{Lemma_degs}
Let $\vec d$ be chosen from $\vec D$.
Then \whp\ the following is true.
	\begin{equation}\label{eqdegs}
	\parbox[t]{13cm}{For any set $S\subset L$ of literals we have $\Vol(S)\leq10|S|\max\cbc{kr,\ln(n/|S|)}.$
		Furthermore, if $|S|\geq n2^{-0.8k}$, then $\Vol(S)\geq\frac13|S|kr.$}
	\end{equation}
\end{lemma}
\begin{proof}
We use the alternative description of $\vec D$ from the proof of \Lem~\ref{Lemma_tame}.
That is, $\vec e=(e_l)_{l\in L}$ is a family of indepedent $\Po(kr/2)$ variables, and
$\cE$ is the event that $\sum_{l\in L}e_l=km$.
Let $\lambda=kr/2$.
For any fixed set $S\subset L$ the random variable
	$X_S=\sum_{l\in S}e_l$ has distribution $\Po(|S|\lambda)$ (because the sum of two independent Poisson variables is Poisson).
Therefore, letting $\mu=10|S|\max\cbc{kr,\ln(n/|S|)}$, we obtain from Stirling's formula
	\begin{eqnarray}
	\pr\brk{X_S>\mu}&\leq&O(\sqrt n)\pr\brk{X_S=\lceil\mu\rceil}
		\leq O(\sqrt n)\cdot\frac{\lambda^\mu}{\mu!\exp(\lambda)}
		\leq O(\sqrt n)\cdot\bcfr{\eul\lambda}{\mu}^\mu\exp(-\lambda).
		\label{eqdegs1}
	\end{eqnarray}
For $1\leq s\leq 2n$ let $X_s=\sum_{S:\abs S=s}\vecone_{X_S>\mu}$.
Then~(\ref{eqdegs1}) yields
	\begin{eqnarray*}
	\Erw X_s&\leq&O(\sqrt n)\bink{2n}s\cdot\exp(-\lambda-\mu)
		\leq O(\sqrt n)\bcfr{2\eul n}{s}^s\cdot\exp(-\lambda-\mu)=o(1/n^2),
	\end{eqnarray*}
because $\mu\geq10 s\ln(n/s)$.
Thus, the first claim follows from~(\ref{eqDtame1}) and the union bound.

To prove the second claim, we use \Lem~\ref{Lemma_danneal}.
For $S\subset L$ we let $Y_S$ be the total number of occurrences of literals from $S$ in $\PHI$.
Then $Y_S$ has distribution $\Bin(km,|S|/2n)$ with mean $|S|kr/2$.
By the Chernoff bound,
	\begin{eqnarray}
	\pr\brk{Y_S<kr|S|/3}&\leq&\exp\brk{-\frac{kr|S|}{100}}.
		\label{eqdegs2}
	\end{eqnarray}
Hence, letting $Y_s=\sum_{S:|S|=s}\vecone_{Y_S<kr|S|/3}$, we get from~(\ref{eqdegs2}) for $s\geq n2^{-0.8k}$
	\begin{eqnarray*}
	\Erw\brk{Y_s}&\leq&\bink{2n}s\exp\brk{-\frac{krs}{100}}\leq\exp\brk{s(2+k)-\frac{kr s}{100}}=o(n^{-2}).
%		\label{eqdegs3}
	\end{eqnarray*}
Thus, by the union bound $\pr\brk{\forall s\geq n2^{-0.8k}:Y_s=0}=1-o(1/n)$.
Applying \Lem~\ref{Lemma_danneal} completes the proof.
\qed\end{proof}

For any $t\in\cT$ we let $n(t)$ be the number of variables $x\in V$ such that $\pd(x)=t$.

\begin{lemma}\label{Lemma_p}
Let $\vec d$ be chosen from $\vec D$.
Then \whp\
for any type $t\in\cT$ we have
	$$n(t)\geq 2^{-3k/4}n.$$
\end{lemma}
\begin{proof}
We use the alternative description of the distribution $\vec D$ from the proof of \Lem~\ref{Lemma_tame}.
That is, let $\vec e=(e_l)_{l\in L}$ be a family of indepedent $\Po(kr/2)$ variables,
and $\cE$ be the event that $\sum_{l\in L}e_l=km$.
For any $s,\Delta$ let $X\bc{s,\Delta}$ denote the number of literals $l$
such that $\sign(l)=s$ and $e_{\abs l}-e_{\neg\abs l}=\Delta$.
Since $\Var(e_l)=kr/2=\Omega_k(k2^k)$, for any $s\in\cbc{\pm1}$ and any $\Delta$ such that $\Delta^2\leq100 k2^k\ln k$ we have
	$\Erw\brk{X\bc{s,\Delta}}\geq n k^{-c}$ for some absolute constant $c>0$.
Furthermore, because the random variables $(e_l)_{l\in L}$ are mutually independent, the Chernoff bound implies that
	\begin{eqnarray}\label{eqLemma_p1}
	\pr\brk{X\bc{s,\Delta}\leq \frac12n k^{-c}}&\leq&\exp(-\Omega(n))\qquad\mbox{provided that $\Delta^2\leq100 k2^k\ln k$.}
	\end{eqnarray}
Similarly, if we let $X_s'$ denote the number of literals $l$
such that $\sign(l)=s$ and $|e_{\abs l}-e_{\neg\abs l}|>100 k2^k\ln k$, then $\Erw\brk{X'\bc{s}}\geq n k^{-c'}$ for some absolute constant $c'$ and
	\begin{eqnarray}\label{eqLemma_p2}
	\pr\brk{X'\bc{s}\leq \frac12n k^{-c'}}&\leq&\exp(-\Omega(n)).
	\end{eqnarray}
Thus, the assertion follows by combining~(\ref{eqDtame1}), (\ref{eqLemma_p1}) and~(\ref{eqLemma_p2}).
\qed\end{proof}

For each 
$t\in\cT$ we let $\pi(t)$ denote the fraction of literal occurrences of $p$-type $t$, i.e.,
\[
	\pi(t)=\sum_{l\in L:\pd(l)=t}\frac{d_l}{km}.
\]

For each $\ell\in \cL$ let
	$$\gamma_\ell=\frac{1}n\Erw\brk{m_{\PHId}(\ell)}.$$

\begin{lemma}\label{Lemma_gammaell}
Let $\vec d$ be chosen from $\vec D$.
Then \whp\
	$\gamma_\ell\sim\prod_{j=1}^k\pi(\ell_j)$
for all $\ell=(\ell_1,\ldots,\ell_k)\in\cL$.
\end{lemma}
\begin{proof}
By the linearity of expectation, we just need to compute the probability that the first clause $\PHI_{\vec d,1}$
has type $\ell$.
Since $\abs{\cT^{-1}(\vec t)}=\Omega(n)$ for all $\vec t\in\cT$, the types of the $k$ literals
of $\PHI_{\vec d,1}$ are asymptotically independent.
Thus, the assertion follows from the fact that $\pi(\ell_j)$ equals the marginal probability that a random lityal has type $\ell_j$.
\qed\end{proof}

\begin{lemma}\label{Lemma_Gamma}
\Whp\ for $\vec d$ chosen from $\vec D$ we have
	$\pr\brk{\forall\ell\in\cL:|m_{\PHId}(\ell)-\gamma_\ell n|\leq n^{2/3}}=1-o(1).$
\end{lemma}
\begin{proof}
Fix a type $\ell=(\ell_1,\ldots,\ell_j)$.
Because $p$ is a feasible marginal, for any $j\in\brk k$ there are $\Omega(n)$ literals $l$ with $p(l)=p(\ell_j)$.
Therefore, a straightforward calculation shows that
	\begin{eqnarray*}
	\pr\brk{\PHI_{d,i}\mbox{ has type }\ell|\PHI_{d,h}\mbox{ has type }\ell}
		=\pr\brk{\PHI_{d,i}\mbox{ has type }\ell}\cdot(1+O(1/n))\qquad\mbox{ for any }i\neq h.
	\end{eqnarray*}
Consequently, 
	$\Var(m_{\PHId}(\ell))\sim\Erw\brk{m_{\PHId}(\ell)}=O(n)$.
Hence, by Chebyshev's inequality
	\begin{eqnarray}\label{eqGamma1}
	\pr\brk{|m_{\PHId}(\ell)-\Erw\brk{m_{\PHId}(\ell)}|>n^{2/3}}=O(n^{-1/3})=o(1).
	\end{eqnarray}
Since $\abs\cL=O(1)$ as $n\ra\infty$ by the construction of $p$, the assertion follows from~(\ref{eqGamma1}) and the union bound.
\qed\end{proof}

\section{The first moment}

\subsection{Outline}

Let $\rho>\frac32\ln2$ be such that $r=2^k\ln2-\rho$.

\begin{proposition}\label{XProp_firstMoment}
\Whp\ $\vec d,\vec m$ are such that
	\begin{eqnarray*}
	\Erw\brk{Z_{p,\mathrm{good}}(\PHIdm)}
			&=&\exp\brk{\frac{n}{2^k}\bc{\rho-\frac{\ln 2}2+o_k(1)}}.
	\end{eqnarray*}
\end{proposition}

We begin by computing $\Erw\brk{Z_{p}(\PHIdm)}$.
By  definition, any assignment that is $\pd$-judicious has $\pd$-margi\-nals.
Thus, let $\cH_p(\vec d)\subset\cbc{0,1}^V$ denote the set of all assignments that have $\pd$-marginals.
Then by the linearity of expectation,
	\begin{eqnarray}\label{Xeqfirst1}
	\Erw\brk{Z_{p}(\PHIdm)}&=&\sum_{\sigma\in\cH_p(\vec d)}\pr\brk{\sigma\in\cS_p(\PHIdm)}.
	\end{eqnarray}
Hence, we need to compute $\abs{\cH_p(\vec d)}$ and the probability 
	$\pr\brk{\sigma\in\cS_p(\PHIdm)}$ for any $\sigma\in\cH_p(\vec d)$.
	Using basic properties of the entropy, we obtain

\begin{lemma}\label{XLemma_entropy}\label{Lemma_entropy}
Let $\chi(z)=-z\ln z-(1-z)\ln(1-z)$ denote the entropy function.
Then \whp\ $\vec d$ is such that
	$$\ln\abs{\cH_p(\vec d)}\sim n\cdot\sum_{x\in V}\chi(p(x)).$$
\end{lemma}

Taylor expanding $\chi(z)$ around $z=1/2$ and plugging in the definition~(\ref{XeqpBP}) of $p$, we obtain that \whp\ 
$\vec d$ is such that
	\begin{equation}\label{XeqCorentropy}
	\frac1n\ln\abs{\cH_p(\vec d)}=\ln2-\frac{k\ln2}{2^{k+1}}+o_k(2^{-k}).
	\end{equation}
As a next step, we compute the probability of 
 $\sigma\in\cS_p(\PHIdm)$ for $\sigma\in\cH_p(\vec d)$.

\begin{lemma}\label{XLemma_probability}
\Whp\ $\vec d$, $\vec m$ are such that
for any $\sigma\in\cH_p(\vec d)$,
	\begin{eqnarray}
	\frac1n\ln\pr\brk{\sigma\in\cS_p(\PHIdm)}&=&
		-\ln2+\frac{k\ln2}{2^{k+1}}+2^{-k}\brk{\rho-\frac{\ln2}2+o_k(1)}.
			\label{Xeqprobabilityexpression}
	\end{eqnarray}
\end{lemma}

Let us defer the proof of \Lem~\ref{XLemma_probability}, which is the core of the first moment computation, for a little while. 
Combining 
(\ref{Xeqfirst1})--(\ref{Xeqprobabilityexpression}), we see that
\whp\ over the choice of $\vec d,\vec m$ we have
	\begin{eqnarray}\nonumber
	\ln\Erw\brk{Z_{p}(\PHIdm)}&=&
		\ln\abs{\cH_p(\vec d)}+\ln\pr\brk{\sigma\in\cS_p(\PHIdm)}\\
		&\sim&2^{-k}\brk{\rho-\frac{\ln2}2+o_k(1)}\cdot n\label{Xeqfirst2}
	\end{eqnarray}
To obtain the expectation of $Z_{p,\mathrm{good}}$,  we show the following.
\begin{lemma}\label{XLemma_theLocalCluster}
\Whp\ over the choice of $\vec d,\vec m$ 
we have
	$$\Erw\brk{Z_{p,\mathrm{good}}(\PHIdm)}\sim \Erw\brk{Z_{p}(\PHIdm)}.$$
\end{lemma}
The proof of \Lem~\ref{XLemma_theLocalCluster} is based on arguments
developed in~\cite{Barriers} for analyzing the geometry of the set of satisfying assignments. 
Combining~(\ref{Xeqfirst2}) and \Lem~\ref{XLemma_theLocalCluster} yields \Prop~\ref{XProp_firstMoment}.

\subsection{Proof of \Lem~\ref{XLemma_probability}}

For a sequence $\vec m=(m(\ell))_{\ell\in\cL}$ of non-negative integers we let $\Gamma_{\vec m}$ denote the event
that $m_{\PHId}(\ell)=m(\ell)$ for all $\ell\in\cL$.
Let us call $\vec m$ \emph{feasible} if $\pr_{\vec m}\brk{\Gamma_{\vec m}}>0$ and $|m(\ell)-\gamma_\ell n|\leq n^{2/3}$ for all $\ell\in\cL$.
Let $Z$ be the number of $\pd$-judicious satisfying assignments.

\begin{proposition}\label{Prop_firstMoment}
Let $\vec d$ be chosen from $\vec D$.
Then \whp\ for any feasible $\vec m=(m(\ell))_{\ell\in\cL}$ the following statements hold.
\begin{enumerate}
\item We have
	\begin{eqnarray*}
	2^{-k}\brk{\rho-\frac{\ln2}2-k^{-9}}&\leq&
	\frac1n\ln\Erw\brk{Z(\PHIdm)}\leq
		2^{-k}\brk{\rho-\frac{\ln2}2+k^{-9}}
		.
	\end{eqnarray*}
\item 
	For any $t\in\cT$ we have
			$$\sum_{l\in L:\pd(l)=t}d_l^2\leq\frac{2km\pi(t)}{n(t)}.$$
\item For any $\sigma\in\cbc{0,1}^V$ with $p$-marginals we have
		$$\frac1{km}\sum_{l\in L}d_l\vecone_{\sigma(l)=1}=\frac12+O(2^{-k}).$$		
\end{enumerate}
\end{proposition}

The proof of \Prop~\ref{Prop_firstMoment} consists of two steps.
We defer the proof of the following lemma to \Sec~\ref{Sec_first}.

\begin{lemma}\label{Lemma_first}
With the assumptions of \Prop~\ref{Prop_firstMoment} and with $\delta,\delta'$ defined by
	\begin{eqnarray*}
	\frac1{km}\sum_{x\in V}\bc{p(x)-\frac12}^2&=&\bc{1+\delta}2^{-2k-2}\quad\mbox{and}\\
	\Sigma&=&\frac1{km}\sum_{x\in V}(1-2p(x))(d_x-d_{\neg x})=-(1+\delta')2^{-k}
	\end{eqnarray*}
we have \whp
	\begin{eqnarray*}
	\frac1n\ln\Erw\brk{Z(\PHIdm)}&=&2^{-k}\brk{\rho-\frac{\ln2}2}+O\bcfr{k(\delta+\delta')}{2^k}+\tilde O(2^{-3k/2}).
	\end{eqnarray*}
\end{lemma}

\medskip
\noindent
\emph{Proof of \Prop~\ref{Prop_firstMoment}.}
Let $\Delta=100k2^k\ln k$ and let $\delta,\delta'$ be as in \Lem~\ref{Lemma_first}.
Using the alternative description of the distribution $\vec D$ from the proof of \Lem~\ref{Lemma_tame} and applying Azuma's inequality,
one can easily verify that \whp
	\begin{eqnarray}\label{eqProp_firstMoment1}
	\sum_{x\in V}\vecone_{(d_x-d_{\neg x})^2\leq\Delta}\cdot(d_x-d_{\neg x})^2\geq
		(1-k^{-12})\sum_{x\in V}(d_x-d_{\neg x})^2.
	\end{eqnarray}
Therefore, \Lem~\ref{Lemma_tame} entails that \whp
	\begin{eqnarray*}
	\frac1{km}\sum_{x\in V}\bc{p(x)-\frac12}^2&=&\frac{1+O_k(k^{-12})}{km}\sum_{x\in V}\frac{(d_x-d_{\neg x})^2}{4^{k+1}}
		=\frac{1+O_k(k^{-12})}{4^{k+1}}.
	\end{eqnarray*}
Consequently, \whp\ we have
	\begin{eqnarray}\label{eqProp_firstMoment2}
	\delta&=&O_k(k^{-12}).
	\end{eqnarray}
Similarly, invoking~(\ref{eqProp_firstMoment1}) once more, we see that \whp
\[
	-\Sigma = \frac1{km}\sum_{x\in V}(2p(x)-1)(d_x-d_{\neg x})
		=\frac1{2^{k}km}\sum_{x\in V}\vecone_{(d_x-d_{\neg x})^2\leq\Delta}\cdot(d_x-d_{\neg x})^2
		=\frac{1+O_k(k^{-12})}{2^{k}},
\]
whence
	\begin{eqnarray}\label{eqProp_firstMoment3}
	\delta'&=&O_k(k^{-12})
	\end{eqnarray}
\whp\
Thus,
\Prop~\ref{Prop_firstMoment} is a direct consequence of \Lem s~\ref{Lemma_p} and~\ref{Lemma_first}
	and~(\ref{eqProp_firstMoment2}), (\ref{eqProp_firstMoment3}).
\qed

\subsection{Proof of \Lem~\ref{Lemma_first}}\label{Sec_first}

We begin by determining the number $\sigma\in\cbc{0,1}^V$ with $p$-marginals.
The following is an easy consequence of \Lem~\ref{XLemma_entropy}.

\begin{corollary}\label{Cor_entropy}
\Whp\ for $\vec d$ chosen from $\vec D$ we have
	$$\frac1n\ln\abs{\cH(p)}=\ln2-\frac2n\sum_{x\in V}\bc{p(x)-\frac12}^2+\tilde O(2^{-3k/2}).$$
\end{corollary}
\begin{proof}
This follows from \Lem~\ref{XLemma_entropy} by Taylor expanding $\chi(\cdot)$ around $\frac12$.
\qed\end{proof}

We need to compute the probability that an assignment $\sigma\in\cbc{0,1}^V$ with $p$-marginals is a $p$-judicious satisfying assignment.
To this end, we introduce a new probability space $(\hat\Omega,\hat\pr)$.
Let $\vec q=(q_{\ell,j})_{\ell\in\cL,j\in\brk k}$ be a matrix with entries in $\brk{0,1}$.
The elements of our new probability space $\hat\Omega$ are all $0/1$ vectors
	$$(\hat\SIGMA_{ij}(\ell))_{\ell\in\cL,i\in\brk{m(\ell)},j\in\brk k}.$$
The distribution $\hat\pr$ is such that the entries $\hat\SIGMA_{ij}(\ell)$ are mutually independent, and 
	for each $\ell=(\ell_1,\ldots,\ell_k)\in\cL$, $i\in\brk{m(\ell)}$, $j\in\brk k$ we let
		$\hat\SIGMA_{ij}(\ell)=\Be(q_{\ell,j})$
	be a Bernoulli random variable.
(It may be helpful to think of $\hat\SIGMA_{ij}(\ell)$ as the truth value of the $j$th literal of the $i$th clause
of type $\ell$ in a random formula $\PHIdm$.)

For $\ell=(\ell_1,\ldots,\ell_k)\in\cL$  let $S_i(\ell)$ be the event that
	$$\max_{j\in\brk k}\hat\SIGMA_{ij}(\ell)=1$$
(the intuition is that this corresponds to the event that the clause $i$ of type $\ell$ is satisfied).
Let $S(\ell)=\bigcap_{i\in\brk{m(\ell)}}S_i(\ell)$ and $S=\bigcap_{\ell\in\cL}S(\ell)$.
Moreover, for $j\in\brk k$ let $B(\ell,j)$ be the event that
	$$\frac1{m(\ell)}\sum_{i\in\brk{m(\ell)}}\hat\SIGMA_{ij}(\ell)\doteq p(t).$$
Let $B(\ell)=\bigcap_{j=1}^kB(\ell,j)$ and $B=\bigcap_{\ell\in\cL}B(\ell)$.
The connection between the probability space $\hat\Omega$ and \Lem~\ref{Lemma_first} is as follows.

\begin{lemma}\label{Lemma_switch}
Suppose that $\sigma\in\cbc{0,1}^V$ has $p$-marginals.
Let $\cS(\sigma)$ be the event that $\sigma$ is a 
satisfying assignment of $\PHIdm$
and let $\cB(\sigma)$ be the event that $\sigma$ is $\pd$-judicious.
Then
	$\pr\brk{\cS(\sigma)|\cB(\sigma)}=\hat\pr\brk{S|B}.$ 
\end{lemma}
\begin{proof}
Note that in $\pr\brk{\cS(\sigma)|\cB(\sigma)}$ probability is taken over the choice of the
random formula $\PHIdm$, while in $\hat\pr\brk{S|B}$ probability
is taken over $\hat\SIGMA$ chosen from the above distribution.
Thus, we need to relate the two probability spaces.

For any $\vec d$-compatible formula $\Phi\in\Gamma_{\vec m}$ we can define a map
	$$\sigma\in\cbc{0,1}^V\mapsto\hat\sigma\big|_{\Phi}=\bc{\sigma_{ij}(\ell)\big|_{\Phi}}_{\ell\in\cL,i\in\brk{m(\ell)},j\in\brk k,}$$
by letting $\hat\sigma_{ij}(\ell)|_{\Phi}$ be the truth value of the $j$th literal of the $i$th clause of type $\ell$ in $\Phi$ under $\sigma$.
In other words, $\hat\sigma|_{\Phi}$ is the string of truth values that we get by ``plugging the assignment $\sigma$ into $\Phi$''.
Then $\sigma$ is judicious iff $\hat\sigma_{\Phi}\in B$.
Furthermore, $\sigma$ is satisfying iff $\hat\sigma_{\Phi}\in S$.
Finally, if $\sigma$ has $p$-marginals, then $\hat\sigma|_{\PHIdm}$ becomes a random vector.
Given $\cB(\sigma)$ its distribution is identical to the conditional distribution of $\hat\SIGMA$ given $B$.
\qed\end{proof}

\begin{corollary}\label{Cor_switch}
With the notation of \Lem~\ref{Lemma_switch} we have
	$\pr\brk{\cS(\sigma)\cap\cB(\sigma)}=\hat\pr\brk{S|B}\exp(o(n))$.
Moreover, for any $\sigma$ with $p$-marginals we have $\pr\brk{\cB(\sigma)}=\Theta\bc{n^{(\abs\cT-k\abs\cL)/2}}$.
\end{corollary}
\begin{proof}
Since the total number $\abs\cL$ of clause types is bounded,
the assertion follows from  a repeated application of \Lem~\ref{lem:locallimit} (the local limit theorem).
\qed\end{proof}

\noindent
Thus, we have reduced the proof of \Lem~\ref{Lemma_first} to the computation of $\hat\pr\brk{S|B}$.
The benefit of the probability space $\hat\Omega$ is that $S,B$ can be decomposed easily
into independent events.
Indeed, for any $\ell\in\cL$ and any $i\in\brk{m(\ell)}$ we have
	$$\hat\pr\brk{S_i(\ell)}=1-\prod_{j=1}^k1-q_{\ell,j},$$
because the $\hat\SIGMA_{ij}(\ell)$ are independent.
Moreover, 
due to independence and because $\vec m$ is feasible,
	$$\frac1n\ln\hat\pr\brk{S(\ell)}=\frac1n\sum_{i\in\brk{m(\ell)}}\ln\hat\pr\brk{S_i(\ell)}\sim\gamma_\ell\ln\brk{1-\prod_{j=1}^k1-q_{\ell,j}}$$
and thus
	\begin{equation}\label{eqProbSat}
	\frac1n\ln\hat\pr\brk{S}\sim\sum_{\ell\in\cL}\gamma_\ell\ln\brk{1-\prod_{j=1}^k1-q_{\ell,j}}.
	\end{equation}
Similarly,
	\begin{equation}\label{eqB}
	\frac1n\ln\hat\pr\brk{B}=\frac1n\sum_{\ell\in\cL}\ln\hat\pr\brk{B(\ell)}=\frac1n\sum_{\ell\in\cL}\sum_{j=1}^k\ln\hat\pr\brk{B(\ell,j)}.
	\end{equation}

A further benefit of the space $\hat\Omega$ is that we are free to choose the vector $\vec q$ as we please
	(subject only to the condition that $\hat\pr\brk B>0$).
To facilitate the computation of $\hat\pr\brk{S|B}$, we are going to choose $\vec q$ such that
	\begin{equation}\label{eqqgoal}
	\hat\pr\brk{B|S}=\exp(o(n)).
	\end{equation}
For if~(\ref{eqqgoal}) holds, then
	$$\hat\pr\brk{S|B}=\frac{\hat\pr\brk S}{\hat\pr\brk B}\cdot\exp(o(n)),$$
where $\hat\pr\brk S$, $\hat\pr\brk B$ can be calculated rather easily via~(\ref{eqProbSat}) and~(\ref{eqB}).
Thus, as a next step we need to find $\vec q$ such that~(\ref{eqqgoal}) is true.
To this end, we define
	\begin{equation}\label{eqtheExpecationWorksOut}
	\hat q_{\ell,j}=\hat\Erw\brk{\hat\SIGMA_{ij}(\ell)|S_i(\ell)}=\frac{q_{\ell,j}}{1-\prod_{l=1}^k1-q_{\ell,l}}\qquad\qquad(\ell\in\cL,j\in\brk k).
	\end{equation}

\begin{lemma}\label{Lemma_fixedpoint}
There exists $\vec q$ such that
	$\hat q_{\ell,j}=\ell_j$ for all $\ell=(\ell_1,\ldots,\ell_k)\in\cL,j\in\brk k$.
Furthermore, this $\vec q$ satisfies
	\begin{equation}\label{eqqofchoice}
	q_{\ell,j}=\ell_j-2^{-k-1}+\tilde O(2^{-3k/2}).
	\end{equation}
\end{lemma}
\begin{proof}
For any $\ell,j$ we have
	\begin{eqnarray*}
	\frac{\partial\hat q_{\ell,j}}{\partial q_{\ell,j}}&=&\frac{1-(1-2q_{\ell,j})\prod_{l\neq j}1-q_{\ell,l}}{\bc{1-\prod_l1-q_{\ell,l}}^2}
				,\\
	\frac{\partial\hat q_{\ell,j}}{\partial q_{\ell,h}}&=&-\frac{q_{\ell,j}\prod_{l\neq h}1-q_{\ell,l}}{\bc{1-\prod_l1-q_{\ell,l}}^2}
		\qquad\qquad\qquad(h\neq j).
	\end{eqnarray*}
Hence, for $k$ large enough and $0.01<q_j<0.99$ for all $j$, the $k\times k$ matrix $D\hat q$ is close to $\id$.
In particular, this is true for $q_j$ close to $1/2$.
Therefore, the assertion follows from the inverse function theorem
	(\Lem~\ref{Lemma_inverseFunction}). % Genauer: Zusatz 2.2, Brocker.
\qed\end{proof}

\begin{corollary}\label{Cor_fixedpoint}
With $\vec q$ from \Lem~\ref{Lemma_fixedpoint} we have
$\hat\pr\brk{B|S}=\Theta(n^{-k\abs\cL/2})=\exp(o(n))$ and thus~(\ref{eqqgoal}).
\end{corollary}
\begin{proof}
Equation~(\ref{eqtheExpecationWorksOut}) shows that
for the vector $\vec q$ from \Lem~\ref{Lemma_fixedpoint} we have $\hat\Erw\brk{\hat\SIGMA_{ij}(\ell)|S_i(\ell)}=\ell_j$ for all $\ell,j,i$.
Therefore, a repeated application of \Lem~\ref{lem:locallimit} yields
	$\hat\pr\brk{B|S}=\Theta(n^{-k\abs\cL/2})=\exp(o(n))$.
\qed\end{proof}

\medskip\noindent
{\bf\em From this point on we fix $\vec q$ as in \Lem~\ref{Lemma_probS}.}

\begin{lemma}\label{Lemma_probS}
Letting
	\begin{equation}\label{eqdefSigma}
	\Sigma=\frac1{km}\sum_{x\in V}(1-2\pd(x))(d_x-d_{\neg x}),
	\end{equation}
we have
	$
	\frac1n\ln\hat\pr\brk{S}
		=-\ln2+2^{-k}\brk{\rho-\frac{\ln2}2-k\ln2}-k\Sigma\ln2+\tilde O(2^{-3k/2}).
	$
\end{lemma}
\begin{proof}
Starting from~(\ref{eqProbSat}), we obtain
	\begin{eqnarray}\label{eqWeNeedThis2}
	\frac1n\ln\hat\pr\brk{S}&\sim&\sum_{\ell\in\cL}\gamma_\ell\ln\brk{1-\prod_{j=1}^k1-q_{\ell,j}}\\
		&=&-\sum_{\ell\in\cL}\gamma_\ell\brk{\bc{\prod_{j=1}^k1-q_{\ell,j}}+\frac12\bc{\prod_{j=1}^k1-q_{\ell,j}}^2+\tilde O(8^{-k})},\nonumber
	\end{eqnarray}
where we used the approximation $\ln(1+x)=x-\frac12x^2+O(x^3)$. Thus, \Lem~\ref{Lemma_gammaell} yields
	\begin{eqnarray*}
	\frac1n\ln\hat\pr\brk{S}
		&=&-\sum_{\ell\in\cL}\gamma_\ell\brk{\bc{\prod_{j=1}^k1-q_{\ell,j}}+\frac12\cdot 4^{-k}+\tilde O(2^{-5k/2})}\\
		&=&-\frac{r}2\cdot 4^{-k}+\tilde O(2^{-3k/2})-r\sum_{\ell\in\cL}
				\prod_{j=1}^k\pi(\ell_j)(1-q_{\ell,j}).
	\end{eqnarray*}
Further, by~(\ref{eqqofchoice})
	\begin{eqnarray*}
	\frac1n\ln\hat\pr\brk{S}
		&=&-\frac{r}2\cdot 4^{-k}+\tilde O(2^{-3k/2})-r\sum_{\ell\in\cL}
				\prod_{j=1}^k\pi(\ell_j)(1-\ell_j+2^{-k-1})\\
		&=&-\frac{r}2\cdot 4^{-k}+\tilde O(2^{-3k/2})-r\brk{\sum_{t\in\cT}
				\pi( t)(1- t+2^{-k-1})}^k\\
		&=&-\frac{r}2\cdot 4^{-k}+\tilde O(2^{-3k/2})-r\brk{2^{-k-1}+\sum_{t\in\cT}\pi(t)(1-t)}^k\\
		&=&-\frac{r}2\cdot 4^{-k}-kr\cdot4^{-k}+\tilde O(2^{-3k/2})-r\brk{\sum_{t\in\cT}\pi(t)(1-t)}^k\\
		&=&-\frac{r}2\cdot 4^{-k}-kr\cdot4^{-k}+\tilde O(2^{-3k/2})-r\brk{
			\frac12-\sum_{t\in\cT}
				\pi(t)\bc{t-\frac12}}^k.
	\end{eqnarray*}
Now,
	\begin{eqnarray*}
	\sum_{ t\in\cT}\pi(t)\bc{t-\frac12}&=&
		\sum_{x\in V}\frac{d_x}{km}\bc{p(x)-\frac12}+\frac{d_{\neg x}}{km}\bc{\frac12-p(x)}\\
		&=&\frac1{km}\sum_{x\in V}\bc{d_x-d_{\neg x}}\bc{p(x)-\frac12}=-\Sigma/2.
	\end{eqnarray*}
Hence,
	\begin{eqnarray*}
	\frac1n\ln\hat\pr\brk{S}&=&-\frac{r}2\cdot 4^{-k}-kr\cdot4^{-k}-r\bcfr{1+\Sigma}2^k+\tilde O(2^{-3k/2})\\
		&=&-r\cdot 2^{-k}-\frac{r}2\cdot 4^{-k}-kr\cdot4^{-k}-kr\Sigma 2^{-k}+\tilde O(2^{-3k/2}).
	\end{eqnarray*}
Plugging in $r=2^k\ln2-\rho$, we get
	\begin{eqnarray*}
	\frac1n\ln\hat\pr\brk{S}
		&=&-\ln2+2^{-k}\brk{\rho-\frac{\ln2}2-k\ln2}-k\Sigma\ln2+\tilde O(2^{-3k/2}),
	\end{eqnarray*}
as claimed.
\qed\end{proof}

\begin{lemma}
We have
	$\frac1n\ln\hat\pr\brk B=-\frac{k\ln2}{2^{k+1}}+\tilde O(2^{-3k/2}).$
\end{lemma}
\begin{proof}
Due to~(\ref{eqB}) we just need to estimate $\ln\hat\pr\brk{B(\ell,j)}$ for any $\ell=(\ell_1,\ldots,\ell_k)\in\cL$ and $j\in\brk k$.
By construction,
	$$\hat\pr\brk{B(t,\ell)}=\pr\brk{\Bin(m(\ell),q_{\ell,j})=\ell_j m(\ell)}.$$
By \Lem~\ref{Lemma_fixedpoint} we have $q_{\ell,j}=\ell_j-2^{-k-1}+\tilde O(2^{-3k/2})=\frac12+\tilde O(2^{-k/2})$.
Hence, using \Lem~\ref{Lemma_binomial}, we find
	\begin{eqnarray}\label{eqWeNeedThis1}
	\frac1{m(\ell)}\ln\hat\pr\brk{B(\ell,j)}&\sim&\psi(q_{\ell,j},\ell_j)\\
		&=&-\frac{(q_{\ell,j}-\ell_j)^2}{2q_{\ell,j}}-\frac{(\ell_j-q_{\ell,j})^2}{2(1-q_{\ell,j})}+O_k(8^{-k})\nonumber\\
		&=&-\frac{(q_{\ell,j}-\ell_j)^2}{2}\bc{\frac1{q_{\ell,j}}+\frac1{1-q_{\ell,j}}}+O_k(8^{-k})\nonumber\\
		&=&-\bc{2+\tilde O(2^{-k/2})}(q_{\ell,j}-\ell_j)^2\nonumber\\
		&=&-\bc{2+\tilde O(2^{-k/2})}(2^{-k-1}+\tilde O(2^{-3k/2}))^2=-\bc{\frac12+\tilde O(2^{-k/2})}2^{-2k}.\nonumber
	\end{eqnarray}
Hence, (\ref{eqB}) yields
	\begin{eqnarray*}
	\frac1n\ln\hat\pr\brk{B}
		&=&-\sum_{\ell\in\cL}\sum_{j=1}^k\frac{m(\ell)}n\cdot\brk{\frac12\cdot2^{-2k}+\tilde O(2^{-5k/2})}\\
		&=&-kr\cdot\brk{\frac12\cdot2^{-2k}+\tilde O(2^{-5k/2})}
		=-\frac{k\ln2}{2^{k+1}}+\tilde O(2^{-3k/2}),
	\end{eqnarray*}
as claimed.
\qed\end{proof}

\begin{remark}\label{Remark_weNeedThis}
In the second moment calculation we will need to know that
	$$\ln\hat\pr\brk{S|B}=
			\sum_{\ell\in\cL}m(\ell)\brk{\ln\bc{1-\prod_{j=1}^k1-q_{\ell,j}}-\sum_{j=1}^k\psi(q_{\ell,j},\ell_j)}$$
which follows from (\ref{eqWeNeedThis2}) and~(\ref{eqWeNeedThis1}).
\end{remark}

\begin{corollary}\label{Cor_firstCalc}
Let $\delta,\delta'>0$ be such that
	\begin{eqnarray*}
	\sum_{x\in V}\bc{\pd(x)-\frac12}^2&=&\frac{(1+\delta)km}{2^{2k+2}},\\
	\Sigma&=&(1+\delta')2^{-k}\qquad\mbox{with $\Sigma$ from~(\ref{eqdefSigma}).}
	\end{eqnarray*}
Then with $r=2^{-k}\ln2-c$ we have
	$$\frac{\ln\abs{\cH(p)}+\ln\hat\pr\brk{S|B}}n=2^{-k}\brk{\rho-\frac{\ln2}2}+O\bcfr{k(\delta+\delta')}{2^k}+\tilde O(2^{-3k/2}).$$
\end{corollary}
\begin{proof}
By the above,
	\begin{eqnarray*}
	\frac1n\ln\hat\pr\brk S&=&-\ln2+2^{-k}\bc{\rho-\frac{\ln2}2-k\ln2}-k\Sigma\ln2+\tilde O(2^{-3k/2})\\
			&=&-\ln2+2^{-k}\bc{\rho-\frac{\ln2}2}+\frac{k\delta'\ln2}{2^k}+\tilde O(2^{-3k/2}),\\
	\frac1n\ln\hat\pr\brk B&=&-\frac{k\ln2}{2^{k+1}}+\tilde O(2^{-3k/2}),\\
	\frac1n\ln\abs{\cH(p)}&=&\ln2-\frac2n\sum_{x\in V}\bc{\pd(x)-\frac12}^2
		=\ln2-\frac{k\ln2}{2^{k+1}}-\frac{\delta k\ln2}{2^{k+1}}.
	\end{eqnarray*}
Summing up yields the result.
\qed\end{proof}

\medskip
\noindent\emph{Proof of \Lem~\ref{Lemma_first}.}
\Lem~\ref{Lemma_first} is a direct consequence of Corollaries~\ref{Cor_entropy},
	\ref{Cor_switch}, \ref{Cor_fixedpoint} and~\ref{Cor_firstCalc}.
\qed

\subsection{Proof of \Lem~\ref{XLemma_theLocalCluster}}

Assume that $\vec m$ is feasible.
Let $\cZ$ denote the number of good $p$-satisfying assignments.

\begin{proposition}\label{Prop_firstGood}
Let $\vec d$ be chosen from $\vec D$ and let $\vec m$ be chosen from $\vec M_{\vec d}$.
Then
	$\Erw\brk{\cZ(\PHIdm)}\sim\Erw\brk{Z(\PHIdm)}$  \whp\ 
\end{proposition}

The proof of \Prop~\ref{Prop_firstGood} is based on three lemmas. 

\begin{lemma}\label{Lemma_dmanneal}
Let $\vec d$ be chosen from $\vec D$ and let $\vec m$ be chosen from $\vec M_{\vec d}$.
\begin{enumerate}
\item Let $\cE$ be an event such that
		$\pr\brk{\PHI\in\cE}=o(1)$.
	Then
		$\pr\brk{\PHIdm\in\cE}=o(1)$.
\item For any random variable $X\geq0$ and any $\eps>0$ we have
		$\pr_{\vec d,\vec m}\brk{\Erw\brk{X(\PHId)}>\Erw\brk{X(\PHI)}/\eps}\leq\eps.$
\end{enumerate}
\end{lemma}
\begin{proof}
This follows from a similar application of Markov's inequality as in the proof of \Lem~\ref{Lemma_danneal}.
\qed\end{proof}

\begin{lemma}\label{Lemma_frozen}
With the assumptions of \Prop~\ref{Prop_firstGood}
the random variable 
	$$\Zgood=\abs{\cbc{\sigma\in\cS(\PHIdm):\abs{\cbc{\tau\in\cS(\PHIdm):\dist(\sigma,\tau)<2^{-0.99k}n}}\leq \Erw\brk{Z(\PHIdm)}}}$$
satisfies
	$\Erw\brk{\Zgood(\PHIdm)}\sim\Erw\brk{Z(\PHIdm)}$ \whp
\end{lemma}

\noindent
The proof of \Lem~\ref{Lemma_frozen} can be found in \Sec~\ref{Sec_frozen}.
Moreover, in \Sec~\ref{Sec_off} we prove the following.

\begin{lemma}\label{Lemma_off}
Suppose that $r\leq2^k\ln2$.
Let $\xi=k 2^{-k/2}$.
Let $Z''$ be the number of pairs $(\sigma,\tau)\in\cS(\PHI)^2$
such that
	$$\dist(\sigma,\tau)\in\brk{k2^{-k},1}\setminus\brk{\frac12-\xi,\frac12+\xi}.$$
Then
	$\Erw\brk{Z''}=o(1).$ 
\end{lemma}
Finally, \Prop~\ref{Prop_firstGood} follows immediately from \Lem s~\ref{Lemma_dmanneal}, \ref{Lemma_frozen} and \ref{Lemma_off}.

\subsection{Proof of \Lem~\ref{Lemma_frozen}}\label{Sec_frozen}

Let $\Phi$ be a $k$-CNF and $\sigma\in\cS(\Phi)$.
We say that a variable $x$ is \emph{$\xi$-rigid} in $(\Phi,\sigma)$ if for any $\tau\in\cS(\Phi)$ with $\tau(x)\neq\sigma(x)$
we have $\dist(\sigma,\tau)\geq\xi n$.
Let $\lambda=kr/(2^k-1)$.

\begin{lemma}\label{Lem_supporting}
\begin{enumerate}
\item The expected number of $\sigma\in\cS(\PHI)$ in which
		more than $k^{12}2^{-k}n$ variables support at most 12 clauses is 
			$\leq\exp(-nk^9/2^k)\Erw\abs{\cS(\PHI)}$.
\item The expected number of $\sigma\in\cS(\PHI)$ in which
		more than $(1+1/k^2)2^{-k}n$ variables support no clause at all
		is $\leq\exp(-n/(k^62^k))\Erw\abs{\cS(\PHI)}$.
\end{enumerate}
\end{lemma}
\begin{proof}
Fix an assignment $\sigma\in\cbc{0,1}^V$, say $\sigma=\vecone$.
Then the number of clauses supported by each $x\in V$ is asymptotically Poisson with mean $\lambda$.
Let $\cE_x$ be the event that $x$ supports no more than 12 clauses.
Then
	$$\pr\brk{\cE_x}\leq\lambda^{12}\exp(-\lambda)\leq \frac12k^{12}2^{-k}.$$
The events $(\cE_x)_{x\in V}$ are negatively correlated.
Therefore, the total number $X$ of variables $x\in V$ for which $\cE_x$ occurs is stochastically
dominated by a binomial variable $\Bin(n,\frac12k^{12}2^{-k})$.
Hence, the first assertion follows from Chernoff bounds.

With respect to the second assertion, let $\cE_x'$ be the event that $x$ supports no clause at all.
Then $\pr\brk{\cE_x}\leq\exp(-\lambda)$.
Using negative correlation and Chernoff bounds once more completes the proof.
\qed\end{proof}

Let us call a set $S\subset V$ \emph{self-contained} if each variable in $S$
supports at least ten clauses that consist of variables in $S$ only.
There is a simple process that yields a (possibly empty) self-contained set $S$.
\begin{enumerate}
\item[$\bullet$] For each variable $x$ that supports at least one clause, choose such a clause $C_x$ randomly.
\item[$\bullet$] Let $R$ be the set of all variables that support at least 12 clauses.
\item[$\bullet$] While there is a variable $x\in R$ that supports fewer than ten clauses $\PHI_i\neq C_x$ 
		that consist of variables of $R$ only, remove $x$ from $R$.
\end{enumerate}
The clauses $C_x$ will play a special role later.

\begin{lemma}\label{Lemma_selfcontained}
The expected number of solutions $\sigma\in\cS(\PHI)$ for which the above process
yields a set $R$ of size $|R|\leq(1-k^{15}/2^k)n$ is bounded by $\exp(-nk^3/2^k)\Erw\abs{\cS(\PHI)}$.
\end{lemma}
\begin{proof}
Let $\sigma\in\cbc{0,1}^V$ be an assignment, say $\sigma=\vecone$.
Let $Q$ be the set of all variables that support fewer than 12 clauses.
By \Lem~\ref{Lem_supporting} we may condition on $|Q|\leq k^{12}2^{-k}n$.
Assume that $|R|\leq(1-k^{15}/2^k)n$.
Then there exists a set $S\subset V\setminus(R\cup Q)$ of size $\frac12k^{15} n/2^k\leq S\leq k^{15} n/2^k$
such that each variable in $S$ supports ten clauses that contain another variable from $S\cup Q$.
With $s=|S|/n$ the probability of this event is bounded by
	\begin{eqnarray*}
	\bink{m}{10sn}\brk{\frac{2^{1-k}}{1-2^{1-k}}\cdot\frac{k^2|S\cup Q|^2}{n^2}}^{10sn}
		&\leq&\brk{4\eul k^2s}^{10sn}.
	\end{eqnarray*}
Hence, the expected number of set $S$ for which the aforementioned event occurs is bounded by
	\begin{eqnarray*}
	\bink ns\brk{4\eul k^2s}^{10sn}&\leq&\brk{\frac\eul s\cdot(4\eul k^2s)^2}^{sn}
		\leq\exp(-sn),
	\end{eqnarray*}
which implies the assertion.
\qed\end{proof}

Let us call a variable $x$ is \emph{attached} if $x$ supports a clause whose other $k-1$ variables belong to $R$.

\begin{corollary}\label{Cor_attached}
The expected number of $\sigma\in\cS(\PHI)$ in which more than
$n/(k^22^k)$ variables $x\not\in R$ that support at least one clause are not attached is bounded by
	$\Erw\abs{\cS(\PHI)}\cdot\exp(-n/(k^62^k))$.
\end{corollary}
\begin{proof}
Let $F=V\setminus R$.
By \Prop~\ref{Lemma_selfcontained} we may assume that
	$|F|\leq nk^{15}/2^k$.
Therefore, for each of the ``special'' clause $C_x$ that we reserved for each $x$ that supports at least one clause
the probability of containing a variable from $F\setminus\cbc x$ is bounded by
	$$(1+o_k(1))k\cdot\frac{|F|}{n}\leq\frac{3k^{16}}{2^k}.$$
Furthermore, these events are independent (because the clauses $C_x$ were disregarded in the construction of $R$).
Hence, the number of variables $x\not\in R$ that support at least one clause but that are not attached
is dominated by $\Bin(|F|,\frac{3k^{16}}{2^k})$.
The assertion thus follows from Chernoff bounds.
\qed\end{proof}

Let us call $S\subset V$ \emph{dense} if each variable in $S$ supports at least ten clauses and at most $2k$ clauses
such that at least ten of them feature another variable from $S$.

\begin{lemma}\label{Lemma_dense}
For $\vec d$ chosen from $\vec D$, $\vec m$ chosen from $\vec M_{\vec d}$ and any $\sigma\in\cbc{0,1}^V$ the following holds \whp\ 
Let $\cA$ be the event that $\sigma$ is a $p$-satisfying assignment of $\PHIdm$.
Then
	\begin{eqnarray*}
	\pr\brk{\mbox{$\PHIdm$ has a dense $S\subset V$, $|S|\leq n2^{-0.99k}$}~|~\cA}=o(1).
	\end{eqnarray*}
\end{lemma}
\begin{proof}
We may assume that $\vec d$ satisfies~(\ref{eqdegs}); we emphasize that this is a property of $\vec d$ only,
	regardless of $\vec m$ or the event $\cA$.
Let $\cD(S)$ be the event that $S\subset V$ is dense.
We may fix (i.e., condition on) the specific clauses supported by each variable $x\in S$.
Let $x\in S$ and let $i\in\brk m$ be the index of a clause supported by $x$.
Let $\ell$ be the type of clause $i$.
For each $t\in\cT$ let $V_t$ be the set of literals $l$ of type $t$.
Then the probability that clause $i$ contains another variable from $S$ is bounded by
	$$\sum_{j\in\brk k}\frac{\Vol(V_{\ell_j}\cap\sigma^{-1}(0)\cap S)}{\Vol(V_{\ell_j}\cap\sigma^{-1}(0))}
		\leq k\max_{t\in\cT}\bcfr{\Vol(V_t\cap\sigma^{-1}(0)\cap S)}{\Vol(V_t\cap\sigma^{-1}(0))}.$$
Since $|V_t|\geq n2^{-0.8k}$ for all $t$ \whp\ by \Lem~\ref{Lemma_p}, we have
	$\Vol(V_t)\geq\frac13 kr|V_t|\geq 30\cdot2^{0.2k}n$.
Furthermore, $\Vol(V_t\cap\sigma^{-1}(0))\geq\frac13\Vol(V_t)$ by the choice of $p(t)$.
Hence, (\ref{eqdegs}) yields
	$$\frac{\Vol(V_t\cap S\cap\sigma^{-1}(0))}{\Vol(V_t\cap\sigma^{-1}(0))}\leq
		\frac{\Vol(S)}{\frac13\Vol(V_t)}\leq
		\frac{\max\cbc{kr,\ln(n/|S|)}|S|}{2^{0.2k}n}.$$
Due to negative correlation, in total we obtain
	\begin{eqnarray*}
	\pr\brk{\PHIdm\in\cD(S)|\cA}&\leq&\bink{2k}{10}^{|S|}\cdot\bcfr{k\max\cbc{kr,\ln(n/|S|)}|S|}{2^{0.2k}n}^{10|S|}.
	\end{eqnarray*}
(The factor $\bink{2k}{10}^{|S|}$ accounts for the number of ways to choose $10$ out of the at most $2k$ clauses
that each variable in $S$ supports.)

For $0<s\leq 1/k^5$ let $X_s$ be the number of sets $S$ of size $|S|=sn$ for which $\cD(S)$ occurs.
Then
	\begin{eqnarray*}
	\Erw\brk{X_s|\cA}&\leq&\bink{n}{sn}
		\bink{2k}{10}^{|S|}\bcfr{k\max\cbc{kr,\ln(n/|S|)}|S|}{2^{0.2k}n}^{10|S|}
		\leq\brk{\frac{\eul\bc{k^2\max\cbc{kr,-\ln(s)}s}^{10}}{s4^{k}}}^{sn}\\
		&=&\brk{\frac{\eul k^{20}\max\cbc{s^9(kr)^{10},s^9\ln^{10}(s)}}{4^{k}}}^{sn}.
	\end{eqnarray*}
There are two cases to consider.
First, if $s\leq\ln(n)/n$, then the term in the brackets is clearly $o(1)$.
Second, if $s\geq\ln(n)/n$, then we have the following bound.
Since $s\leq s_{\max}=2^{-0.99k}$ and as $x\mapsto x^9\ln^{10}x$ is monotonically increasing for $x<0.1$,
we have
	$$s^9\ln^{10}(s)\leq s_{\max}^9\ln^{10}s_{\max}\leq s_{\max}^9(kr)^{10}\leq 2^{10k-8.91k}=2^{1.09k}.$$
Hence, the entire bracket is bounded by $2^{-k/2}$.
Summing over all possible $s$ and using Markov's inequality completes the proof.
\qed\end{proof}

Let us call a variable $x\in V$ \emph{$\xi$-rigid} in  $\sigma\in\cS(\Phi)$
if for any $\tau\in\cS(\Phi)$ with $\tau(x)\neq\sigma(x)$ we have $\dist(\sigma,\tau)\geq\xi n$.

\begin{corollary}\label{Cor_rigid}
\Whp\ for $\vec d$ chosen from $\vec D$ and for $\vec m$ chosen from $\vec M_{\vec d}$  the following is true.
Let $\sigma\in\cbc{0,1}^V$ and 
let $\cA$ be the event that $\sigma$ is a $p$-satisfying assignment of $\PHIdm$.
Moreover, let $Y$ be the number of variables that are not $2^{-0.99k}$-rigid.
Then %\whp\ 
	$$\pr\brk{Y(\PHIdm)\leq (1+2k^{-2})2^{-k}n~|~\cA}=1-o(1).$$
\end{corollary}
\begin{proof}
Let $\xi=2^{-0.99k}$.
We condition on the event $\cA$.
Consider a variable $z$ that is either attached or in $R$.
Let $\tau\in\cS(\PHIdm)$ be such that $\tau(z)\neq\sigma(x)$ and $\dist(\sigma,\tau)< n/2^{0.99k}$.
Because $z$ is attached or in $R$, the set
	$$\Delta=\cbc{x\in R:\tau(x)\neq\sigma(x)}$$
is non-empty.
Moreover, $\Delta$ is dense by the construction of $R$.
Thus, \Lem~\ref{Lemma_dense} shows that $\dist(\sigma,\tau)\geq\abs{\Delta}\geq n/2^{0.99k}$ \whp\
Hence, \whp\ all $z$ that are either attached or in $R$ are $\xi$-rigid.

Further, let $\cR$ be the event that 
\begin{itemize}
\item no more than $(1+1/k^2)2^{-k}n$ variables support no clause at all and
\item at most
		$n/(k^22^k)$ variables $x\not\in R$ that support at least one clause are not attached
\end{itemize}
Then \Lem~\ref{Lem_supporting} and \Cor~\ref{Cor_attached} imply 
together with \Prop~\ref{Prop_firstMoment} that
	$$\pr\brk{\PHIdm\in\cR~|~\cA}=1-o(1).$$
Hence, the total number of vertices that either do not support a clause or that are not attached is bounded by 
	$(1+2/k^2)2^{-k}n$
\qed\end{proof}

\noindent\emph{Proof of \Lem~\ref{Lemma_frozen}.}
Suppose that $r=2^k\ln2-c$.
By \Prop~\ref{Prop_firstMoment}  we have
	$$\frac1n\ln\Erw\brk{Z(\PHIdm)}\geq 2^{-k}\bc{c-\frac{\ln2}2+o_k(1)}$$
\whp\
Now, assume that in $\sigma\in\cS(\PHIdm)$ all but at most
$(1+2k^{-2})2^{-k}n$ variables are $\xi$-rigid with $\xi=2^{-0.99k}$.
If $\tau\in\cS(\PHIdm)$ is such that $\dist(\sigma,\tau)\leq\xi n$, then $\sigma,\tau$ agree on all
$\xi$-rigid variables of $\sigma$.
Hence,
	$$\frac1n\ln\cbc{\tau\in\cS(\PHIdm):\dist(\sigma,\tau)\leq\xi n}\leq(1+2k^{-2})2^{-k}=(1+o_k(1))2^{-k}\ln2.$$
As $c-\frac{\ln2}2+o_k(1)>(1+o_k(1))\ln2$ for $c>\frac32\ln2+\eps$ and $k$ large enough, the assertion follows.
\qed

\subsection{Proof of \Lem~\ref{Lemma_off}}\label{Sec_off}

By Markov's inequality, it suffices to bound the expected number of paris
$(\sigma,\tau)\in\cS(\PHI)$ at the given Hamming distances.
More precisely, let $\cZ_x$ be the number of pairs $(\sigma,\tau)\in\cS(\PHI)$ such that $\dist(\sigma,\tau)/n=x$.
Let $h(x)=-x\ln x-(1-x)\ln(1-x)$ and set
	$$q(x)=r\cdot\ln\bc{1-2^{1-k}+2^{-k}(1-x)^k}.$$
Then
	\begin{equation}\label{eqoff1}
	\frac1n\ln\Erw\brk{\cZ_x}\leq \ln2+h(x)+q(x).
	\end{equation}
We consider several cases.
\begin{description}
\item[Case 1: $k2^{-k}\leq x\leq(2k)^{-1}$.]
	We have
		\begin{eqnarray*}
		h(x)+q(x)+\ln2&\leq&\ln2+x(1-\ln x)+r\bc{-2^{1-k}+2^{-k}(1-x)^k}\\
			&\leq&\ln2+x(1-\ln x)+2^k\ln2\bc{-2^{1-k}+2^{-k}(1-x)^k}+c2^{1-k}\quad\mbox{[as $r=2^k\ln2-c$]}\\
			&\leq&x(1-\ln x)-\ln2+(1-x)^k\ln2\\
			&\leq&x(1-\ln x)-\ln2+(1-kx+k^2x^2)\ln2\\
			&\leq&x(1-\ln x)-kx+k^2x^2=x\brk{1-\ln x-k+k^2x}.
		\end{eqnarray*}
	If $k2^{-k}\leq x\leq k^{-2}$, then $1-\ln x-k+k^2x\leq1-\ln k+1<0$.
	Moreover, if $k^{-2}\leq x\leq (2k)^{-1}$, then
		$1-\ln x-k+k^2x\leq1+2\ln k-\frac34k<0$.
\item[Case 2: $(2k)^{-1}<x<0.01$.]
	We have
		\begin{eqnarray*}
		h(x)+q(x)+\ln2&\leq&\ln2+x(1-\ln x)+r\bc{-2^{1-k}+2^{-k}(1-x)^k}\\
			&\leq&\ln2+x(1-\ln x)-\frac{r}{2^{k-1}}+\frac r{2^k}\exp(-kx)\\
			&\leq&x(1-\ln x)-\ln2+\frac c{2^{k-1}}+\exp(-kx)\ln2\\
			&\leq&x(1-\ln x)+\frac c{2^{k-1}}+(\exp(-1/2)-1)\ln2\\
		\end{eqnarray*}
	The last expression is negative for $x<0.05$ (and $k$ not too small).
\item[Case 3: $0.01<x<\frac12-k2^{-k/2}$.]
	We have
		\begin{eqnarray*}
		h'(x)&=&-\ln x+\ln(1-x),\\
		q'(x)&=&-\frac{kr(1-x)^{k-1}}{2^k-2+(1-x)^k}\geq -\frac{kr(1-x)^{k-1}}{2^k-2}=\exp(-\Omega(k)).
		\end{eqnarray*}
	Hence, for $0.01\leq x<\frac12-k^{-2}$ we have
		$h'(x)+q'(x)>0$.
	Thus, $h(x)+q(x)+\ln2$ is monotonically increasing in this interval.
	Now, let $x=\frac12-\eps$ for $k^{-2}\leq\eps\leq k2^{-k/2}$.
	Then
		\begin{eqnarray*}
		h(x)&=&\ln2-2\eps^2+O(\eps^3),\\
		q(x)&=&(2^k\ln2-c)\bc{-2^{1-k}+2^{1-2k}+2^{-k}\bc{\frac12-\eps}^k+\tilde O(8^{-k}}\\
			&=&-2\ln2+2^{1-k}\bc{c+\ln2}+\tilde O(4^{-k}).
		\end{eqnarray*}
	Consequently,
		\begin{eqnarray*}
		h(x)+q(x)+\ln2&=&-2\eps^2+O(\eps^3)+O(2^{-k})<0.
		\end{eqnarray*}
\item[Case 4: $\frac12+k2^{-k/2}\leq x<1$.]
	The function $h(x)$ satisfies $h(1-y)=h(y)$ for $0<y<1/2$.
	Furthermore, $q(x)$ is monotonically decreasing.
	Therefore, for any $x\geq \frac12+k2^{-k/2}$ we have
		$$\ln2+h(x)+q(x)\leq \ln2+h\bc{\frac12-k2^{-k}}+q\bc{\frac12-k2^{-k}}<0.$$
\end{description}
In each case we have $\ln2+h(x)+q(x)<0$.
Thus, the assertion follows from~(\ref{eqoff1}) and Markov's inequality.

\section{The second moment}\label{Sec_secondMoment}

\emph{Throughout this section we assume that 
	 $r=2^{-k}\ln2-\rho$ with $\rho=\frac32\ln2-\eps_k$  for some sequence $\eps_k=o_k(1)$ that tends to $0$ sufficiently slowly.
	We also assume that $k\geq k_0$ for a large enough constant $k_0>3$.
	We let $\vec d$ denote a signed degree sequence $\vec d$ chosen from $\vec D$ and we let $\vec m$ denote
	a vector chosen from $\vec M_{\vec d}$.
	By \Lem~\ref{Lemma_Gamma} we may assume that 
		$|m(\ell)-\gamma_\ell n|\leq n^{2/3}$ for all $\ell$.
	Let $\SIGMA,\TAU\in\cbc{0,1}^V$ denote a pair of assignments chosen uniformly and independently from the set of all assignments
	with $p$-marginals.
	Finally, let $\xi=k2^{-k/8}$. 
	}

\subsection{Outline}
\label{ssec:secMomOutline}

The \bemph{overlap} of two assignments $\sigma,\tau\in\cbc{0,1}^V$ is the vector
	$$\cO(\sigma,\tau)=
		\brk{\frac1{km \pi(t)}\sum_{l\in L:\pd(l)=t}
			 d_{l}\cdot\vecone_{\sigma(l)=1}\cdot\vecone_{\tau(l)=1}}_{t\in\cT}.$$
In words, $\cO(\sigma,\tau)$ captures the fraction of occurrences of literals of each type $ t$ that are
true under both $\sigma,\tau$.
Since $\SIGMA,\TAU$ are independent and have $p$-marginals, we have
	$$\Erw\brk{\cO(\SIGMA,\TAU)}=\brk{t^2}_{t\in\cT}.$$
Set $\cO^*=\brk{t^2}_{t\in\cT}$.

Let $Z''$ be the number of pairs $(\sigma,\tau)$ of $\pd$-judicious satisfying assignments of $\PHIdm$ 
such that
	\begin{equation}\label{eqdistCritical}
	\dist(\sigma,\tau)\in\brk{\frac12-k^22^{-k/2},\frac12+k^22^{-k/2}}.
	\end{equation}
Moreover, let 
	$Z'$ 
be the number of pairs $(\sigma,\tau)$ of $\pd$-judicious  satisfying assignments of $\PHIdm$ such that 
	$$\norm{\cO(\sigma,\tau)-\cO^*}_\infty\leq\xi.$$

\begin{proposition}\label{Prop_in}
\Whp\ $\vec d,\vec m$ are such that
	$\Erw\brk{Z''(\PHIdm)}\leq \Erw\brk{Z'(\PHIdm)}+o(1)
			.$
\end{proposition}

\noindent
The proof of \Prop~\ref{Prop_in} can be found in \Sec~\ref{Sec_in}.
Let $Z$ denote the number of $p$-satisfying assignments of $\PHId$.
Furthermore, let $\cZ$ signify the number of good $p$-satisfying assignments of $\PHId$.
In \Sec~\ref{Sec_centre} we are going to establish the following.

\begin{proposition}\label{Prop_centre}
\Whp\ $\vec d,\vec m$ are such that
	$\Erw\brk{Z'(\PHIdm)}\leq
			C\cdot \Erw\brk{Z(\PHIdm)}^2.$
\end{proposition}

\begin{corollary}\label{Cor_centre}
\Whp\ $\vec d,\vec m$ are such that
	$\Erw\brk{\cZ^2(\PHIdm)}\leq C'\cdot \Erw\brk{\cZ(\PHIdm)}^2.$
\end{corollary}
\begin{proof}
Let $Y$ be the number of pairs $(\sigma,\tau)$ of good $p$-satisfying assignments of $\PHIdm$
such that
	\begin{equation}\label{eqCorCentre0}
	\dist(\sigma,\tau)\not\in\brk{\frac12-k^22^{-k},\frac12+k^22^{-k}}.
	\end{equation}
By definition, for any good $\sigma$ there are at most $\Erw\brk{Z(\PHIdm)}$ $p$-satisfying $\tau$ such that~(\ref{eqCorCentre0}) holds.
Therefore,
	\begin{equation}\label{eqCorCentre1}
	\Erw\brk{Y(\PHIdm)}\leq 
			\Erw\brk{Z(\PHIdm)}^2.
	\end{equation}
Combining~(\ref{eqCorCentre1}) with \Prop~\ref{Prop_in} and~\ref{Prop_centre}, we obtain for $\vec d$ chosen from $\vec D$ \whp
	\begin{eqnarray}\nonumber
	\Erw\brk{\cZ^2(\PHIdm)}&\leq&\Erw\brk{(Y+Z'')(\PHIdm)}\\
		&\leq& \Erw\brk{(Y+Z')(\PHIdm)}+o(1)\leq(C+1)\Erw\brk{Z(\PHIdm)}^2+o(1).
		\label{eqCorCentre2}
	\end{eqnarray}
By \Prop~\ref{Prop_firstMoment} we have $\Erw\brk{Z(\PHIdm)}=\exp(\Omega(n))$.
Furthermore, \Prop~\ref{Prop_firstGood} yields $\Erw\brk{Z(\PHIdm)}\sim\Erw\brk{\cZ(\PHIdm)}$.
Consequently, (\ref{eqCorCentre2}) implies
	$\Erw\brk{\cZ^2(\PHIdm)}\leq(C+2)\Erw\brk{\cZ(\PHIdm)}^2$, as desired.
\qed\end{proof}

\noindent
The second part of \Thm~\ref{XThm_secondPlain} follows directly from \Cor~\ref{Cor_centre}.

\subsection{Proof of \Prop~\ref{Prop_in}}\label{Sec_in}

We begin by relating the overlap to the Hamming distance.

\begin{lemma}\label{Lemma_totaloverlap}
\Whp\ $\vec d,\vec m$ are such that
for all pairs $\sigma,\tau\in\cbc{0,1}^V$ satisfying~(\ref{eqdistCritical}) we have
	$$\bar\cO(\sigma,\tau)=\frac1{km}\sum_{l\in L}d_l\vecone_{\sigma(l)=1}\vecone_{\tau(l)=1}
			=\frac14+\tilde O(2^{-k/2}).$$
\end{lemma}
\begin{proof}
By \Lem~\ref{Lemma_tame} \whp
	\begin{eqnarray*}
	\cO&=&\frac1{km}\sum_{x\in V}\frac{d_x}2\vecone_{\sigma(x)=\tau(x)}+O(|d_x^+-d_x^-|)\\
		&=&\tilde O(2^{-k/2})+\frac1{km}\sum_{x\in V}\frac{d_x}2\vecone_{\sigma(x)=\tau(x)}\\
		&=&\tilde O(2^{-k/2})+\frac1{km}\sum_{x\in\cM}\frac{d_x}2
			=\frac14+\tilde O(2^{-k/2}),
	\end{eqnarray*}
as claimed.
\qed\end{proof}

\begin{lemma}\label{Lemma_crude}
\Whp\ $\vec d,\vec m$ are such that
for any $\sigma,\tau\in\cbc{0,1}^V$ that satisfy~(\ref{eqdistCritical}) and that have $p$-marginals
we have
	$$\frac1n\ln\pr\brk{\sigma,\tau\in\cS(\PHId)}\leq-2\ln2+O(k2^{-k}).$$
\end{lemma}
\begin{proof}
Much as in the first moment calculation in \Sec~\ref{Sec_first}, here it is convenient to work with a different probability space.
Namely, we let $\hat\Omega$ be the set of all vectors $(\hat\sigma_{ij},\hat\tau_{ij})_{i\in\brk m,j\in\brk k}$
of $0/1$ pairs.
We define a probability distribution on $\hat\Omega$ 
in which the \emph{pairs} $(\hat\sigma_{ij},\hat\tau_{ij})_{i\in\brk m,j\in\brk k}$ are mutually independent random variables.
For any $i\in\brk m,j\in\brk k$ we let 
	$\hat\pr\brk{(\hat\sigma_{ij},\hat\tau_{ij})=(a,b)}=q^{ab}$,
where the parameters $q^{ab}$ are chosen so that the following equations hold:
	\begin{eqnarray*}
	q^{11}&=&\bar\cO(\sigma,\tau),\\
	q^{10}&=&q^{01},\\
	q^{11}+q^{10}&=&\frac1{km}\sum_{l\in L}d_l\vecone_{\sigma(l)=1},\\
	\sum_{a,b=0}^1q^{ab}&=&1.
	\end{eqnarray*}
Let $(\hat\SIGMA,\hat\TAU)$ denote a random pair chosen from this distribution.

\Prop~\ref{Prop_firstMoment} and \Lem~\ref{Lemma_totaloverlap} ensure that \whp\ $\vec d$ is such that
	\begin{equation}\label{eqcrude1}
	q^{11}=\frac14+\tilde O(2^{-k/2}),\qquad q^{11}+q^{10}=\frac12+O(2^{-k}).
	\end{equation}
Thus, we may assume that~(\ref{eqcrude1}) holds.

Let $B$ be the event that 
	$$\sum_{i,j}\hat\SIGMA_{ij}=\sum_{l\in L}d_l\vecone_{\sigma(l)=1},\quad
		\sum_{i,j}\hat\TAU_{ij}=\sum_{l\in L}d_l\vecone_{\tau(l)=1}\quad\mbox{and}\quad
			\sum_{i,j}\hat\SIGMA_{ij}\hat\TAU_{ij}
				=km\bar\cO(\sigma,\tau).$$
In addition, let $S$ be the event that $\max_{j\in\brk k}\hat\SIGMA_{ij}=\max_{j\in\brk k}\hat\TAU_{ij}$ for all $i\in\brk m$.
We claim that
	\begin{equation}\label{eqcrude2}
	\pr\brk{\sigma,\tau\in\cS(\PHId)}=\pr\brk{S|B}. 
	\end{equation}
Indeed, any $d$-compatible formula $\Phi$ induces a pair $(\hat\sigma|_\Phi,\hat\tau|_\Phi)\in\hat\Omega$
defined by $\hat\sigma_{ij}|_\Phi=\sigma(\Phi_{ij})$, $\hat\tau_{ij}|_\Phi=\tau(\Phi_{ij})$.
Clearly, the distribution of the random pair $(\hat\sigma|_{\PHId},\hat\tau|_{\PHId})$ is identical to the distribution of $(\hat\SIGMA,\hat\TAU)$ given $B$.

Due to independence, the probability of the event $S$ is easy to compute.
Indeed, with $q=q^{10}+q^{11}$ inclusion/exclusion yields
	$$\hat\pr\brk{S}=\brk{1-2q^k+(1-2q+q^{11})^k}^m$$
Furthermore, $\hat\pr\brk{B}=\exp(o(n))$ by the local limit theorem for the multinomial distribution.
Hence, (\ref{eqcrude2}) yields
	\begin{eqnarray*}
	\frac1n\ln\pr\brk{\sigma,\tau\in\cS(\PHId)}&=&
	\frac1n\ln\hat\pr\brk{S|B}\leq o(1)+\frac1n\ln\frac{\hat\pr\brk{S}}{\hat\pr\brk B}=o(1)+\frac1n\ln\hat\pr\brk{S}\\
		&\sim&r\ln\brk{1-2q^k+(1-2q+q^{11})^k}\leq-r\brk{2q^k-(1-2q+q^{11})^k}.
	\end{eqnarray*}
Using~(\ref{eqcrude1}) and simplifying completes the proof.
\qed\end{proof}

\begin{lemma}\label{Lemma_toverlap}
Let $\lambda>2^{-k}$ and $t\in\cT$.
For $\vec d$ chosen from $\vec D$ the following is true \whp\
Let $\cH''$ be the set of all pairs $\sigma,\tau\in\cbc{0,1}^V$ such that
	$|\cO_t(\sigma,\tau)-1/4|>\lambda.$
Then 
	$$\abs{\cH''}\leq4^n\exp\brk{-\frac{\lambda^2n(t)}{18}}.$$
\end{lemma}
\begin{proof}
Let $\SIGMA'',\TAU''\in\cbc{0,1}^V$ be chosen uniformly and independently.
Then $\Erw\brk{\cO_t(\SIGMA'',\TAU'')}=\frac14$.
Furthermore, $\cO_t$ satisfies the following Lipschitz condition.
\begin{quote}
If $\sigma',\tau'',\sigma'',\tau''\in\cbc{0,1}^V$ are such that
	there is a literal $l_0$ with $\cT(l_0)=t$ such that
		$\sigma''(l)=\sigma'(l),\tau''(l)=\tau'(l)$ for all $l\not\in\cbc{l_0,\neg l_0}$,
then
	$$\abs{\cO_t(\sigma'',\tau'')-\cO_t(\sigma',\tau')}\leq\frac{2d_{l_0}}{km\pi(t)}.$$
\end{quote}
Therefore, by Azuma's inequality for any $\lambda>0$ we have
	\begin{eqnarray*}
	\pr\brk{\abs{\cO_t(\SIGMA'',\TAU'')-\frac14}>\lambda}
		\leq\exp\brk{-\frac{\lambda^2(km\pi(t))^2}{9\sum_{l\in\cL:\cT(l)=t}d_l^2}}
			\leq\exp\brk{-\frac{\lambda^2n(t)}{18}},
	\end{eqnarray*}
where the last step follows from part 2 of \Prop~\ref{Prop_firstMoment}.
\qed\end{proof}

\noindent\emph{Proof of \Prop~\ref{Prop_in}.}
Let $H''$ be the set of pairs $(\sigma,\tau)$ such that
\begin{itemize}
\item $\sigma,\tau$ satisfy~(\ref{eqdistCritical}) and have $p$-marginals, and
\item $\norm{\cO(\sigma,\tau)-\cO^*}_{\infty}>\xi$.
\end{itemize}
Then by \Lem~\ref{Lemma_toverlap}
	and the second part of \Prop~\ref{Prop_firstMoment}
	 \whp\ (over the choice of $\vec d$) we have
	\begin{eqnarray}\label{eqtoverlap1}
	\abs{H''}&\leq&4^n\exp\brk{-\frac{\xi^2n(t)}{36}}\leq4^n\exp\brk{-\frac{k^2n}{36\cdot 2^k}}.
	\end{eqnarray}
Furthermore, by \Lem~\ref{Lemma_crude} \whp\ (again over the choice of $\vec d$) we have
	\begin{eqnarray}\label{eqtoverlap2}
	\pr\brk{\sigma,\tau\in\cS(\PHId)}&\leq&4^{-n}\exp\brk{\frac{O(k)}{2^k}}\qquad\mbox{for any }(\sigma,\tau)\in H''.
	\end{eqnarray}
Combining~(\ref{eqtoverlap1}) and~(\ref{eqtoverlap2}), we obtain that \whp\ $\vec d$ is such that
	\begin{eqnarray*}
	\Erw\brk{(Z''-Z')(\PHId)}&\leq&\sum_{(\sigma,\tau)\in H''}\pr\brk{\sigma,\tau\in\cS(\PHId)}
			\leq\abs{H''}4^{-n}\exp\brk{\frac{O(k)}{2^k}}=o(1).
	\end{eqnarray*}
Therefore, the definition of the distribution $\vec M_{\vec d}$ entails that \whp\ $\vec d$ is such that
	$$\Erw_{\vec m}\brk{\Erw\brk{(Z''-Z')(\PHIdm)}}=\Erw\brk{(Z''-Z')(\PHId)}=o(1).$$
Thus, the assertion follows from Markov's inequality.
\qed

\section{Proof of \Prop~\ref{Prop_centre}}\label{Sec_centre}

\emph{We keep the notation and the assumptions of \Sec~\ref{Sec_secondMoment}.
	}

\subsection{Overview}

For two assignments $\sigma,\tau$ and a formula $\Phi$ with signed degree distribution $\vec d$ we 
define a matrix
	$$\omega(\sigma,\tau,\Phi)=(\omega_{\ell,j}(\sigma,\tau,\Phi))_{\ell\in\cL,j\in\brk k}$$
by letting
	$\omega_{\ell,j}(\sigma,\tau,\Phi)$ be equal to the fraction of clauses of type $\ell$
		whose $j$th literal is true under both $\sigma,\tau$.
We call $\omega_{\ell,j}(\sigma,\tau,\Phi)$ the \bemph{overlap matrix} of $\sigma,\tau$ in $\Phi$.
Recalling that $\SIGMA,\TAU$ denote two independent uniformly distributed assignments with $p$-marginals,
we define $\vec\omega=\omega(\SIGMA,\TAU,\PHIdm)$; thus, $\vec\omega$ is a random matrix.
We use the symbol $\omega$ to denote (fixed, non-random) matrices $\omega=(\omega_{\ell,j})_{\ell\in\cL,j\in\brk k}$ with entries in $\brk{0,1}$.
Furthermore, we let $\omega_\ell=(\omega_{\ell,j})_{j\in\brk k}$ denote the $\ell$-row of such a matrix $\omega$.
Finally, let $\omega^*=(\omega^*_{\ell,j})$ be the  matrix with entries $\omega^*_{\ell,j}=\ell_j^2$ for all $\ell,j$.

In addition, let $\cS(\ell)$ be the event that both $\SIGMA,\TAU$ satisfy all clauses of type $\ell$ of $\PHIdm$.
Let $\cS=\bigcap_{\ell\in\cL}\cS(\ell)$.
Further, let $\cB(\ell,j)$ be the event that under both
	$$\frac1{m(\ell)}\sum_{i\in M_{\PHIdm}(\ell)}\SIGMA(\PHI_{\vec d,\vec m,i,j})\TAU(\PHI_{\vec d,\vec m,i,j})\doteq \ell_j,$$
i.e., the fraction of clauses of type $\ell$ whose $j$th literal is true equals $\ell_j+O(1/n)$.
Let
	$$\cB=\bigcap_{\ell\in\cL,\,j\in\brk k}\cB(\ell,j).$$
In \Sec~\ref{Sec_Pcentre} we are going to prove the following.

\begin{proposition}\label{Prop_Pcentre}
\Whp\ $\vec d,\vec m$ are such that the following holds.
Let $\cL'\subset\cL$ be a set of clause types and let $\cS'=\bigcap_{\ell\in\cL'}\cS(\ell)$.
\begin{enumerate}
\item For all $\omega=(\omega_{\ell,j})$ such that
	$|\omega_{\ell,j}-\omega_{\ell,j}^*\big|\leq k^{-12}$
for all $\ell\in\cL',j\in\brk k$
 we have the bound
	\begin{eqnarray*}
	\pr\brk{\cS'|\vec\omega\doteq\omega,\,\cB}
	&\leq&
			\pr\brk{\cS'|\vec\omega\doteq\omega^*,\,\cB}
			\exp\brk{\tilde O(4^{-k})\sum_{\ell\in\cL'}m(\ell)\norm{\omega_{\ell}-\omega_{\ell}^*}_2^2}.
	\end{eqnarray*}
\item We have
			\begin{eqnarray*}
		\pr\brk{\cS|\vec\omega\doteq\omega^*,\,\cB}
		&\leq&\pr\brk{\cS'|\vec\omega\doteq\omega^*,\,\cB}
			\exp\brk{-\Theta(2^{-k})\sum_{\ell\not\in\cL'}m(\ell)}. 
	\end{eqnarray*}
\item 
	For any assignment $\sigma$ with $p$-marginals we have
		\begin{eqnarray*}
		\pr\brk{\cS|\vec\omega\doteq\omega^*,\,\cB}&\leq&O(1)\cdot\pr\brk{\sigma\in\cS_p(\PHIdm)|\sigma\mbox{ is $p$-judicious}}^2.
		\end{eqnarray*}
\end{enumerate}
\end{proposition}

For $\omega=(\omega_{\ell,j})$ define $\cO(\omega)\in\brk{0,1}^\cT$ by letting
	$$\cO_t(\omega)=\sum_{\ell\in\cL}\sum_{j\in\brk k:\ell_j=t}\frac{m(\ell)\omega_{\ell,j}}{km\pi(t)}.$$
We also let $\bar\omega$ denote the matrix with entries $\bar\omega_{\ell,j}=\cO_{\ell_j}(\omega)$ for all $\ell,j$.
We say that $\omega$  is \bemph{compatible} with $\cO\in\brk{0,1}^\cT$ if $\cO=\cO(\omega)$.
In \Sec~\ref{Sec_omegacentre} we are going to prove the following.

\begin{proposition}\label{Prop_omegacentre}
\Whp\ $\vec d$ has the following property.
For any $\omega=(\omega_{\ell,j})$ such that $\norm{\cO(\omega)-\frac14\vecone}_\infty\leq2\xi$ 
we have
	\begin{eqnarray*}
	\pr\brk{\vec\omega\doteq\omega|\cO(\vec\omega)\doteq\cO(\omega),\,\cB}&\leq&
			\pr\brk{\vec\omega\doteq{\bar\omega}|\cO(\vec\omega)\doteq\cO(\omega),\,\cB}\exp\brk{
				-\Omega_k\bc{1}\cdot\sum_{\ell\in\cL}m(\ell)\norm{\omega_\ell-\bar\omega_\ell}_2^2}.
	\end{eqnarray*}
\end{proposition}

\noindent
Recall that $\cO^*=(t^2)_{t\in\cT}$.
In \Sec~\ref{Cor_omegacentre} we will prove the following.

\begin{corollary}\label{Cor_omegacentre}
\Whp\ $\vec d,\vec m$ are such that the following holds.
For any $\cO=(\cO_t)_{t\in\cT}$ such that
	$\norm{\cO-\frac14\vecone}_\infty\leq 2\xi$
we have
	\begin{eqnarray*}
	\pr\brk{\cS|\cO(\SIGMA,\TAU)\doteq\cO,\,\cB}
			&\leq&O(1)\cdot\pr\brk{\cS|\OMEGA\doteq\omega^*
				,\,\cB}
				\exp\brk{n\cdot\tilde O(2^{-k})\sum_{ t\in\cT}\pi( t)(\cO_{ t}-\cO_{ t}^*)^2}.
	\end{eqnarray*}
\end{corollary}

\noindent
In \Sec~\ref{Sec_Bprob} we will show the following.

\begin{proposition}\label{Prop_Bprob}
There exists a constant $\eta>0$ such that 
\whp\ $\vec d,\vec m$ are such that the following holds.
For all $\cO$ with
	$\norm{\cO-\frac14\vecone}_\infty\leq 2\xi$
we have
	\begin{eqnarray*}
	\pr\brk{\cB|\cO(\SIGMA,\TAU)\doteq\cO}
			&\leq&\eta\cdot \pr\brk{\cB|\cO(\SIGMA,\TAU)\doteq\cO^*}.
	\end{eqnarray*}
Furthermore,
	$\pr\brk{\cB|\cO(\SIGMA,\TAU)=\cO^*}=\Theta(n^{\abs\cT-k\abs\cL}).$
\end{proposition}

\noindent
Recall that $n(t)$ is the number of variables of type $t\in\cT$.
In \Sec~\ref{sec:enumeration} we are going to prove the following.

\begin{proposition}\label{Prop_entropy}
\Whp\ $\vec d,\vec m$ are such that the following holds.
For all vectors 
$\lambda=(\lambda_t)_{t\in\cT}$ with $\norm{\lambda}_\infty\leq1/8$ we have
	\begin{eqnarray*}
	\pr\brk{\forall t\in\cT:|\cO_t(\SIGMA,\TAU)-\cO_t^*|\geq\lambda_t}
		&\leq&
			\exp\brk{-n\cdot\Omega_k(1)\sum_{t\in\cT}\pi(t)\lambda_t^2}.
	\end{eqnarray*}
\end{proposition}

\noindent\emph{Proof of \Prop~\ref{Prop_centre}.}
Suppose that $\cO\in\brk{0,1}^\cT$ satisfies $\norm{\cO-\cO^*}_\infty\leq \xi$.
By \Prop~\ref{Prop_Bprob} \whp
	\begin{eqnarray}\nonumber
	\pr\brk{\cS,\cB|\cO(\SIGMA,\TAU)\doteq\cO}&=&
		\pr\brk{\cS|\cB,\cO(\SIGMA,\TAU)\doteq\cO}\pr\brk{\cB|\cO(\SIGMA,\TAU)\doteq\cO}\\
		&\leq&\eta\cdot\pr\brk{\cS|\cB,\cO(\SIGMA,\TAU)\doteq\cO}\pr\brk{\cB|\cO(\SIGMA,\TAU)\doteq\cO^*}.\label{eqProp_centre1}
	\end{eqnarray}
Furthermore, by \Cor~\ref{Cor_omegacentre} \whp
	\begin{eqnarray}\label{eqProp_centre2}
	\pr\brk{\cS|\cO(\SIGMA,\TAU)\doteq\cO,\cB}&\leq&O(1)\cdot\pr\brk{\cS|\OMEGA\doteq\omega^*,\cB}
		\exp\brk{n\tilde O(2^{-k})\sum_{t\in\cT}\pi(t)(\cO_{ t}-\cO_{t}^*)^2}.
	\end{eqnarray}
Combining~(\ref{eqProp_centre1}) and~(\ref{eqProp_centre2}), we see that
	\begin{eqnarray}\nonumber
	\pr\brk{\cS,\cB|\cO(\SIGMA,\TAU)\doteq\cO}&\leq&O(1)\cdot\pr\brk{\cS|\cB,\OMEGA\doteq\omega^*}\cdot\pr\brk{\cB|\cO(\SIGMA,\TAU)\doteq\cO^*}\\
		&&\qquad\qquad\cdot\exp\brk{n\tilde O(2^{-k})\sum_{t\in\cT}\pi(t)(\cO_{t}-\cO_{ t}^*)^2}.
		\label{eqProp_centre3a}
	\end{eqnarray}
For an assignment $\sigma$ with $p$-marginals let
	$$b=\pr\brk{\sigma\mbox{ is $p$-judicious in $\PHIdm$}},\qquad s=\pr\brk{\sigma\in\cS_p(\PHIdm)|\sigma\mbox{ is $p$-judicious in $\PHIdm$}}.$$
Then by part~3 of \Prop~\ref{Prop_Pcentre}, \Cor~\ref{Prop_Bprob} and \Cor~\ref{Cor_switch} we have
	$$\pr\brk{\cS|\cB,\OMEGA\doteq\omega^*}\cdot\pr\brk{\cB|\cO(\SIGMA,\TAU)\doteq\cO^*}\leq O(1)\cdot(bs)^2.$$
Therefore, (\ref{eqProp_centre3a}) yields
	\begin{eqnarray}
	\pr\brk{\cS,\cB|\cO(\SIGMA,\TAU)\doteq\cO}&\leq&O(1)\cdot(bs)^2\exp\brk{n\tilde O(2^{-k})\sum_{t\in\cT}\pi(t)(\cO_{t}-\cO_{ t}^*)^2}.
		\label{eqProp_centre3}
	\end{eqnarray}

For a vector $\lambda=(\lambda_t)_{t\in\cT}$ let
	$$h(\lambda)=\pr\brk{\forall t\in\cT:\abs{\cO_t(\SIGMA,\TAU)-\cO_t^*}\geq\lambda_t}.$$
Moreover, for $c=c(k)>0$ a sufficiently large number let
	$\Lambda=\frac{c}{\sqrt n}\ZZ_{\geq 0}^{\cT}$
be the positive $\cT$-dimensional grid scaled by a factor of $c/\sqrt n$.
In addition, let $h$ be the number of assignments $\sigma$ with $p$-marginals.
Then by \Prop~\ref{Prop_entropy} and~(\ref{eqProp_centre3}) there is a number $\zeta=\zeta(k)>0$ such that
	\begin{eqnarray*}
	\frac{\Erw\brk{Z'(\PHIdm)}}{\Erw\brk{Z(\PHIdm)}^2}&\leq&O(1)\cdot\frac{\Erw\brk{Z'(\PHIdm)}}{(bhs)^2}
		\leq O(1)\cdot\sum_{\lambda\in\Lambda}h(\lambda)\exp\brk{n\tilde O(2^{-k})\sum_{t\in\cT}\pi(t)(\lambda_t+c/\sqrt n)^2}\\
		&\leq&O(1)\cdot\sum_{\lambda\in\Lambda}h(\lambda)\exp\brk{n\tilde O(2^{-k})\sum_{t\in\cT}\pi(t)\lambda_t^2}\\
		&\leq&O(1)\cdot\sum_{\lambda\in\Lambda}\exp\brk{n\sum_{t\in\cT}\pi(t)\lambda_t^2\brk{\tilde O(2^{-k})-\Omega_k(1)}}\\
		&\leq&O(1)\cdot\sum_{\lambda\in\Lambda}\exp\brk{-n\cdot\Omega_k(1)\sum_{t\in\cT}\pi(t)\lambda_t^2}
			\leq O(1)\cdot\sum_{\lambda\in\Lambda}\exp\brk{-\zeta n\norm{\lambda}_2^2}\\
		&\leq&O(1)\cdot\sum_{\vec z\in \ZZ_{\geq 0}^{\cT}}\exp\brk{-\zeta c^2 \norm{\vec z}_2^2}
			=O(1)\brk{\sum_{z=0}^\infty\exp\brk{-\zeta c^2 z^2}}^{\abs\cT}=O(1),
	\end{eqnarray*}
as desired.
\qed

\paragraph{Notation for the proofs of \Prop s~\ref{Prop_Pcentre}--\ref{Prop_Bprob}.}
It will be convenient to work with a different probability space.
Namely, let $\hat\Omega$ be the set of all pairs $(\hat\sigma,\hat\tau)$ of $0/1$ vectors
	$$(\hat\sigma,\hat\tau)=(\hat\sigma_{ij}(\ell),\hat\tau_{ij}(\ell))_{\ell\in\cL,i\in\brk{m(\ell)},j\in\brk k}.$$
Let $B_{\ell,j}\subset\hat\Omega$ be the event that
	$$\frac1{m(\ell)}\sum_{i\in\brk{m(\ell)}}\hat\sigma_{ij}(\ell)\doteq \ell_j\quad\mbox{ and }\quad
			\frac1{m(\ell)}\sum_{i\in\brk{m(\ell)}}\hat\tau_{ij}(\ell)
			\doteq  \ell_j$$
for all $\ell\in\cL,j\in\brk k$.
Let $B_\ell=\bigcap_{j\in\brk k}B_{\ell,j}$ and let $B=\bigcap_{\ell\in\cL}B_{\ell}$.

To define a measure $\hat\pr$ on $\hat\Omega$, 
let $\vec q=(q^{ab}_{\ell,j})_{a,b\in\cbc{0,1},\ell\in\cL,j\in\brk k}$ be a vector with entries in $\brk{0,1}$ such that
	\begin{equation}\label{eqqvector1}
	\sum_{a,b=0}^1q^{ab}_{\ell,j}=1,\qquad q^{01}_{\ell,j}=q^{10}_{\ell,j}
	\end{equation}
for all $\ell,j$.
Define 
	\begin{equation}\label{eqqvector2}
	q_{\ell,j}=q^{11}_{\ell,j}+q^{10}_{\ell,j}
	\end{equation}
so that
	\begin{equation}\label{eqqvector3}
	q^{00}_{\ell,j}=1-2q_{\ell,j}+q^{11}_{\ell,j}.
	\end{equation}
We define a measure $\hat\pr=\hat\pr_{\vec q}$ on $\hat\Omega$ as follows.
\begin{quote}
For any $\ell=(\ell_1,\ldots,\ell_k)\in\cL$, $i\in\brk{m(\ell)}$ and $j\in\brk k$ independently we choose a pair
of values $(\hat\SIGMA_{ij}(\ell),\hat\TAU_{ij}(\ell))\in\cbc{0,1}^2$ such that 
	$$\hat\pr\brk{(\hat\SIGMA_{ij}(\ell),\hat\TAU_{ij}(\ell))=(a,b)}=q^{ab}_{\ell,j}$$
for any $a,b\in\cbc{0,1}$.
\end{quote}
This probability space induces a random matrix $\hat{\OMEGA}=(\hat{\OMEGA}_{\ell,j})_{\ell,j}$ with entries
	$$\hat\OMEGA_{\ell,j}=\frac1{m(\ell)}\sum_{i\in\brk{m(\ell)}}\hat\SIGMA_{ij}(\ell)\hat\TAU_{ij}(\ell).$$

We will use the probability space $(\hat\Omega,\hat\pr)$ several times in the proof of the various propositions below,
with various choices of $\vec q$.

\subsection{Proof of \Prop~\ref{Prop_omegacentre}}\label{Sec_omegacentre}

Consider any $\omega=(\omega_{\ell,j})$ such that
	$\norm{\cO(\omega)-\frac14\vecone}_\infty\leq2\xi$.
We use the probability space $(\hat\Omega,\hat\pr)$ with the vector $\vec q$ defined by 
	$$q^{11}_{\ell,j}=\bar\omega_{\ell,j},\qquad q_{\ell,j}=\ell_j\qquad\mbox{ for all }\ell,j;$$
the remaining entries of $\vec q$ are determined by~(\ref{eqqvector1})--(\ref{eqqvector3}).
Then the following is immediate from the construction.

\begin{fact}\label{Fact_omegacentre}
We have
	$\pr\brk{\vec\omega\doteq\omega|\cO(\vec\omega)\doteq\cO(\omega),\,\cB}=
		\pr\brk{\hat\OMEGA\doteq\omega|\cO(\hat\OMEGA)\doteq\cO(\omega),\,B}.$
\end{fact}

Now,
	\begin{eqnarray}\nonumber
	\pr\brk{\hat\OMEGA\doteq\omega|\cO(\hat\OMEGA)\doteq\cO(\omega),\,B}&=&
		\frac{\pr\brk{\hat\OMEGA\doteq\omega,\,\cO(\hat\OMEGA)\doteq\cO(\omega),\,B}}{\pr\brk{\cO(\hat\OMEGA)\doteq\cO(\omega),\,B}}
			\doteq\frac{\pr\brk{\hat\OMEGA\doteq\omega,\,B}}{\pr\brk{\cO(\hat\OMEGA)\doteq\cO(\omega),\,B}}\\
		&=&\frac{\pr\brk{B|\hat\OMEGA\doteq\omega}}{\pr\brk{\cO(\hat\OMEGA)\doteq\cO(\omega),\,B}}\cdot\pr\brk{\hat\OMEGA\doteq\omega}\nonumber\\
		&\leq&O(1)\cdot
			\frac{\pr\brk{B|\hat\OMEGA\doteq\bar\omega}}{\pr\brk{\cO(\hat\OMEGA)\doteq\cO(\bar\omega),\,B}}\cdot\pr\brk{\hat\OMEGA\doteq\omega},
			\nonumber 
	\end{eqnarray}
The last step follows from the local limit theorem for the multinomial distribution because 
	$$\Erw\brk{\sum_{i\in\brk{m(\ell)}}\hat\SIGMA_{ij}(\ell)\bigg|\hat\OMEGA\doteq\bar\omega}=
		\Erw\brk{\sum_{i\in\brk{m(\ell)}}\hat\TAU_{ij}(\ell)\bigg|\hat\OMEGA\doteq\bar\omega}
			=m(\ell)\cdot \ell_j$$
for all $\ell,j$.
Hence,
	\begin{eqnarray*}\nonumber
	\frac{\pr\brk{\hat\OMEGA\doteq\omega|\cO(\hat\OMEGA)\doteq\cO(\omega),\,B}}{\pr\brk{\hat\OMEGA\doteq\bar\omega|\cO(\hat\OMEGA)\doteq\cO(\omega),\,B}}
		&\leq&\frac{\pr\brk{\hat\OMEGA\doteq\omega}}{\pr\brk{\hat\OMEGA\doteq\bar\omega}}.
			\nonumber 
	\end{eqnarray*}
For each $\ell\in\cL$, $j\in\brk k$ the sum	 $\sum_{i\in\brk{m(\ell)}}\hat\SIGMA_{ij}(\ell)\hat\TAU_{ij}(\ell)$ has a binomial distribution
	$\Bin(m(\ell),\bar\omega_{\ell,j})$.
Furthermore, these random variables are mutually independent.
Therefore, Chernoff bounds yield
	\begin{eqnarray*}
	\frac{\pr\brk{\hat\OMEGA\doteq\omega}}{\pr\brk{\hat\OMEGA\doteq\bar\omega}}&\leq&
		\exp\brk{-\Omega_k(1)\sum_{\ell\in\cL}\sum_{j\in\brk k}m(\ell)(\omega_{\ell,j}-\bar\omega_{\ell,j})^2},
	\end{eqnarray*}
whence the assertion follows.

\subsection{Proof of \Cor~\ref{Cor_omegacentre}}\label{Sec_Cor_omegacentre}

Let $\omega$ be an overlap matrix such that $\cO\doteq\cO(\omega)$.
Let $\cL'=\cL'(\vec\omega)$ be the set of all $\ell\in\cL$ such that $|\omega_{\ell,j}-1/4|\leq\xi$ for all $j\in\brk k$.
Let $\cS'=\bigcap_{\ell\in\cL'}\cS(\ell)$.
Then
	\begin{eqnarray*}
	P(\omega)&=&\pr\brk{\cS',\,\vec\omega\doteq\omega|\cO(\vec\omega)\doteq\cO,\,\cB}\\
		&=&\pr\brk{\cS'|\vec\omega\doteq\omega,\,\cB}
		\cdot
			\pr\brk{\vec\omega\doteq\omega|\cO(\vec\omega)\doteq\cO,\,\cB}.
	\end{eqnarray*}
Let
	\begin{eqnarray*}
	\bar P&=&
			\pr\brk{\cS'|\vec\omega\doteq\omega^*,\,\cB}\cdot
			\pr\brk{\vec\omega\doteq{\bar\omega}|\cO(\vec\omega)\doteq\cO,\,\cB};
	\end{eqnarray*}
observe that $\bar P$ depends on $\cO$ but not on the specific choice of $\omega$.
Then by \Prop s~\ref{Prop_Pcentre} and~\ref{Prop_omegacentre}
	\begin{eqnarray*}
	P(\omega)&\leq&\bar P\cdot\exp\brk{
		\sum_{\ell\in\cL}m(\ell)\brk{\vecone_{\ell\in\cL'}\cdot\tilde O(4^{-k})\norm{\omega_\ell-\omega^*_\ell}_2^2
			-\Omega_k(1)\norm{\bar\omega_\ell-\omega_\ell}_2^2}}\\
		&\leq&\bar P
			\cdot\exp\brk{\sum_{\ell\in\cL}m(\ell)\brk{\vecone_{\ell\in\cL'}\cdot\tilde O(4^{-k})\bc{\norm{\bar\omega_\ell-\omega_\ell}_2^2+
				\norm{\bar\omega_\ell-\omega^*_\ell}_2^2}
			-\Omega_k(1)\norm{\bar\omega_\ell-\omega_\ell}_2^2}}\\
		&\leq&\bar P\cdot\exp\brk{\sum_{\ell\in\cL}m(\ell)\brk{
			\tilde O(4^{-k})
				\norm{\bar\omega_\ell-\omega^*_\ell}_2^2
			-\Omega_k(1)\norm{\bar\omega_\ell-\omega_\ell}_2^2}}.
	\end{eqnarray*}
By the second part of \Prop~\ref{Prop_Pcentre},
	\begin{eqnarray*}
	\frac1n\ln\frac{\pr\brk{\cS'|\vec\omega\doteq\omega^*,\,\cB}}
		{\pr\brk{\cS|\vec\omega\doteq\omega^*,\,\cB}}
		&=&\Theta(2^{-k})\sum_{\ell\not\in\cL'}\frac{m(\ell)}n
			\leq\sum_{\ell\not\in\cL'}\frac{m(\ell)}{kn}\xi^2
		\leq\frac1k\sum_{\ell\not\in\cL'}\frac{m(\ell)}n\norm{\bar\omega_\ell-\omega_\ell}_2^2.
	\end{eqnarray*}
Hence, letting
	$
	\tilde P=	\pr\brk{\cS|\vec\omega\doteq\omega^*,\,\cB}\cdot
			\pr\brk{\vec\omega\doteq{\bar\omega}|\cO(\vec\omega)\doteq\cO,\,\cB},$
we obtain
	\begin{eqnarray*}
	P(\omega)
		&\leq&\tilde P\exp\brk{\tilde O(4^{-k})\sum_{\ell\in\cL}m(\ell)
				\norm{\bar\omega_\ell-\omega^*_\ell}_2^2-\sum_{\ell\in\cL}\Omega_k(1)m(\ell)\norm{\bar\omega_\ell-\omega_\ell}_2^2}.
	\end{eqnarray*}
To proceed, we note that
	\begin{eqnarray*}
	\sum_{\ell\in\cL}m(\ell)\norm{\bar\omega_\ell-\omega_\ell^*}^2&=&
		\sum_{j\in\brk k}\sum_{\ell\in\cL}m(\ell)(\cO_{\ell_j}^*-\cO_{\ell_j})^2\\
		&=&\sum_{t\in\cT}\sum_{\ell\in\cL}\sum_{j\in\brk k}m(\ell)(\cO_{ t}^*-\cO_{ t})^2\cdot\vecone_{\ell_j= t}
			=km\sum_{ t\in\cT}\pi( t)(\cO_{ t}^*-\cO_{ t})^2.
	\end{eqnarray*}
Thus,
	\begin{eqnarray*}
	P(\omega)&\leq& {\tilde P}
		\exp\brk{n\tilde O(2^{-k})\sum_{t\in\cT}\pi( t)(\cO_{ t}^*-\cO_{ t})^2
				-\sum_{\ell\in\cL}\Omega_k(1)m(\ell)\norm{\bar\omega_\ell-\omega_\ell}^2}.
	\end{eqnarray*}
Summing  over all possible overlap matrices $\omega$ of assignments with $p$-marginals, we get
	\begin{eqnarray*}
	P&=&\sum_{\omega:\cO(\omega)\doteq\cO}P(\omega)=
			\pr\brk{\cS'|\cO(\vec\omega)\doteq\cO,\,\cB}\geq \pr\brk{\cS|\cO(\vec\omega)\doteq\cO,\,\cB},
	\end{eqnarray*}
which we can bound by
	\begin{eqnarray*}
	P&\leq&\tilde P\cdot\exp\brk{n\tilde O(2^{-k})\sum_{ t\in\cT}\pi( t)(\cO_{ t}^*-\cO_{ t})^2}
			\sum_{\omega:\cO(\omega)\doteq\cO}\exp\brk{-\sum_{\ell\in\cL}\Omega_k(1)m(\ell)\norm{\bar\omega_\ell-\omega_\ell}_2^2}\\
			&=&\pr\brk{\cS|\vec\omega\doteq\omega^*,\,\cB}
				\exp\brk{n\tilde O(2^{-k})\sum_{t\in\cT}\pi( t)(\cO_{ t}^*-\cO_{ t})^2}\\
			&&\quad\cdot
				\sum_{\omega:\cO(\omega)\doteq\cO}
					\exp\brk{-\sum_{\ell\in\cL}\Omega(1)m(\ell)\norm{\bar\omega_\ell-\omega_\ell}_2^2}
									\pr\brk{\vec\omega\doteq{\bar\omega}|\cO(\vec\omega)\doteq\cO,\,\cB}\\
			&\leq&O(1)\cdot\pr\brk{\cS|\vec\omega\doteq\omega^*,\,\cB}
				\exp\brk{n\tilde O(2^{-k})\sum_{t\in\cT}\pi( t)(\cO_{ t}^*-\cO_{ t})^2},
	\end{eqnarray*}
as desired.

\subsection{Proof of \Prop~\ref{Prop_Bprob}}\label{Sec_Bprob}

Let $\omega$ be such that $\cO\doteq\cO(\omega)$ and $\norm{\omega-\bar\omega}_\infty\leq n^{-1/3}$.
We are going to work with the probability space $(\hat\Omega,\hat\pr)$ defined by letting
	$$q^{11}_{\ell,j}=\omega_{\ell,j},\qquad q_{\ell,j}=\ell_j.$$
We claim that there exist numbers $0<c_k<c_k'$ (independent of $\omega$) such that \whp\ $\vec d$ is such that
	\begin{equation}\label{eqBprob1}
	c_k\leq n\pr\brk{B_{\ell,j}|\hat\OMEGA\doteq\omega}\leq c_k'\qquad\mbox{for all }\ell,j.
	\end{equation}
Indeed, given $\hat\OMEGA_{\ell,j}\doteq\omega_{\ell,j}$ the total number of indices $i\in\brk{m(\ell)}$ such that
	$(\hat\SIGMA_{ij}(\ell),\hat\TAU_{ij}(\ell))=(1,0)$ has distribution
		$$\Bin\bc{(1-\omega_{\ell,j})m(\ell),\frac{\ell_j-\omega_{\ell,j}}{1-\omega_{\ell,j}}}.$$
Therefore, the probability that the total number of such $i$ equals its expectation is in the interval $\brk{c_{k,1}n^{-1/2},c_{k,2}n^{-1/2}}$
for certain $c_{k,2}>c_{k,1}>0$.
Furthermore, given this event, the 
number of $i\in\brk{m(\ell)}$ such that $(\hat\SIGMA_{ij}(\ell),\hat\TAU_{ij}(\ell))=(0,1)$ has distribution
	$$\Bin\bc{(1-\ell_j)m(\ell),\frac{\ell_j-\omega_{\ell,j}}{1-\ell_j}}.$$
Once more, the conditional probability that this random variable equals its expectation lies in 
	$\brk{c_{k,3}n^{-1/2},c_{k,4}n^{-1/2}}$ for certain $c_{k,4}>c_{k,3}>0$.
Hence, setting $c_k=c_{k,1}c_{k,3}$ and $c_k'=c_{k,2}c_{k,4}$, we obtain~(\ref{eqBprob1}).

Summing~(\ref{eqBprob1}) over all (finitely many) possible $\omega$ with $\pr\brk{\omega\doteq\hat\OMEGA}>0$ and $\cO(\omega)\doteq\cO$
and invoking \Prop~\ref{Prop_omegacentre}, we find that \whp\ over the choice of $\vec d$,
	\begin{eqnarray*}
	\pr\brk{B_{\ell,j}|\cO(\hat\OMEGA)\doteq\cO}&=&\sum_{\omega}\pr\brk{B_{\ell,j}|\hat\OMEGA\doteq\omega}\pr\brk{\hat\OMEGA\doteq\omega}\\
		&\leq&o(1/n)+\sum_{\omega:\norm{\omega-\bar\omega}_\infty\leq n^{-1/3}}\pr\brk{B_{\ell,j}|\hat\OMEGA\doteq\omega}\pr\brk{\hat\OMEGA\doteq\omega}\\
		&\leq&o(1/n)+c_k'/n\leq2c_k'/n.
	\end{eqnarray*}
A similar calculation shows $\pr\brk{B_{\ell,j}|\cO(\hat\OMEGA)\doteq\cO}\geq \frac12c_k/n$.
As $c_k,c_k'$ are independent of the specific vector $\cO$, the assertion follows.

\section{Proof of \Prop~\ref{Prop_Pcentre}}\label{Sec_Pcentre}

\emph{We keep the notation and the assumptions of \Sec~\ref{Sec_secondMoment}.
	}

\subsection{Outline}

In \Sec~\ref{Sec_functionF} we will establish the following.

\begin{proposition}\label{Prop_functionF}
There exist $C^2$-functions $\cP_\ell(\cdot)$ that range over matrices $\omega=(\omega_{\ell,j})_{\ell\in\cL,j\in\brk k}$
such that
	$$\norm{\omega_{\ell}-\omega^*_{\ell}}_\infty<k^{-12}\quad\mbox{ for all }\quad\ell\in\cL'$$
with the following properties.
\begin{enumerate}
\item For all such $\omega$ we have
		$$\pr\brk{\cS'|\vec\omega\doteq\omega,\,\cB}=
			\exp\brk{O(1)+\sum_{\ell\in\cL'}m(\ell)\cdot\cP_\ell(\omega_\ell)}.$$
\item For each $\ell$, $\cP_\ell$ is a function of the row $\omega_\ell$ only.
\end{enumerate}
\end{proposition}

We need to analyse the functions $\cP_\ell$ from \Prop~\ref{Prop_functionF}.
Crucially, $\omega^*$ turns out to be a stationary point.

\begin{proposition}\label{Prop_calculus1}
The differentials of the functions $\cP_\ell$ from \Prop~\ref{Prop_functionF} satisfy $D\cP_\ell\bc{\omega_\ell^*}=0$
	for all $\ell$.
\end{proposition}

The proof of \Prop~\ref{Prop_calculus1} can be found in \Sec~\ref{Sec_calculus1}.
Furthermore, in \Sec~\ref{Sec_calculus2} we derive the following bound on the second
derivatives of $\cP_\ell$.

\begin{proposition}\label{Prop_calculus2}
The functions $\cP_\ell$ from \Prop~\ref{Prop_functionF} have the following property.
For any $j,j',\ell$ we have
		$$\frac{\partial^2\cP_\ell}{\partial\omega_{\ell,j}\partial\omega_{\ell,j'}}\leq\tilde O(4^{-k})$$
	on the entire domain of $\cP_\ell$.
\end{proposition}

\begin{corollary}\label{Cor_calculus}
For any $\omega$ in the domain of $\cP$ we have
	$$\cP_\ell(\omega_\ell)\leq\cP(\omega_\ell^*)+\tilde O(4^{-k})\norm{\omega_\ell-\omega_\ell^*}_2^2.$$
\end{corollary}
\begin{proof}
This follows directly from \Prop s~\ref{Prop_calculus1} and~\ref{Prop_calculus2} and Taylor's formula.
\qed\end{proof}

\medskip\noindent
Finally, in \Sec~\ref{Sec_PcentreComplete} we will show of  \Prop~\ref{Prop_Pcentre} 
follows from \Prop~\ref{Prop_functionF} and \Cor~\ref{Cor_calculus}.

\subsection{Proof of \Prop~\ref{Prop_functionF}}\label{Sec_functionF}

To construct the functions $\cP_\ell$, we are going to work with the probability space $(\hat\Omega,\hat\pr)$ from \Sec~\ref{Sec_centre} once more;
we are going to define the vector $\vec q$ that determines the measure $\hat\pr$ so as to facilitate the definition of $\cP_\ell$ in due course.
Fix $\omega=(\omega_{\ell,j})_{\ell\in\cL,j\in\brk k}$
such that
	$\norm{\omega_{\ell}-\omega^*_{\ell}}_\infty<k^{-12}$ for all $\ell\in\cL'$.
Let 
$B'=\bigcap_{\ell\in\cL'}B_{\ell}$.
Further, for $\ell\in\cL$ and $j\in\brk k$ let $C_{\ell,j}$ be the event that $\hat\OMEGA_{\ell,j}\doteq\omega_{\ell,j}$.
Let $C_{\ell}=\bigcap_{j\in\brk k}C_{\ell,j}'$ and let $C'=\bigcap_{\ell\in\cL'}C_{\ell}'$.
Finally, 
let $S'=\bigcap_{\ell\in\cL'}S(\ell)$.
The following two facts are direct consequences of the definition of $\hat\pr$.

\begin{fact}
If $\vec q$ is such that 
$\hat\pr\brk{B'\cap C'}>0$, then
 $\hat\pr\brk{\cdot|B'\cap C'}$ is the uniform distribution
over the set $B'\cap C'$.
\end{fact}

\begin{fact}\label{Fact_connection}
Suppose that $\vec q$ is such that 
the conditional distribution $\hat\pr\brk{\cdot|B'\cap C'}$ is uniform.
Then
	$\hat\pr\brk{S'|B',\,C'}=\pr\brk{\cS'|\vec\omega\doteq\omega,\,\cB}.$
\end{fact}

Thus, our goal is pick $\vec q$  such that $\hat\pr\brk{S'|B',\,C'}$ is easy to compute.
Roughly speaking, we are going to accomplish this by choosing $\vec q$ so that $\hat\pr\brk{B',\,C'|S'}$ is as big as possible.
To implement this, we first need to determine the unconditional probabilities $\hat\pr\brk{S'}$, $\hat\pr\brk{B',C'}$ as functions of $\vec q$.

\begin{lemma}\label{Lemma_SATsmm}
Suppose that
$\vec q$ is such that
 $q_{\ell,j}\in(0,1)$ for all $\ell\in\cL'$, $j\in\brk k$. Then
	\begin{equation}\label{eqSatProb}
	\hat\pr\brk{S_i(\ell)}=1-2\prod_{j=1}^k(1-q_{\ell,j})+\prod_{j=1}^k(1-2q_{\ell,j}+q^{11}_{\ell,j})
	\end{equation}
for all $\ell\in\cL'$, $i\in\brk{ m(\ell)}$,
and
	\begin{eqnarray*}
	\frac1n\ln\hat\pr\brk{S'}
		&=&\sum_{\ell\in\cL}\frac{m(\ell)}n\ln\brk{1-2\prod_{j=1}^k(1-q_{\ell,j})+\prod_{j=1}^k(1-2q_{\ell,j}+q^{11}_{\ell,j})}.
	\end{eqnarray*}
\end{lemma}
\begin{proof}
The first statement follows by inclusion/exclusion.
The probability that $\max_{j\in\brk k}\hat\SIGMA_{ij}(\ell)=0$ equals $\prod_{j=1}^k(1-q_{\ell,j})$ as the 
components $\hat\SIGMA_{ij}(\ell)$ are the results of independent $\Be(q_{\ell,j})$ experiments.
For the event $\max_{j\in\brk k}\hat\TAU_{ij}(\ell)=0$ we get the exact same expression.
Furthermore, the probability of $\max_{j\in\brk k}\hat\SIGMA_{ij}(\ell)=\max_{j\in\brk k}\hat\TAU_{ij}(\ell)=0$ 
equals $\prod_{j=1}^k(1-2q_{\ell,j}+q^{11}_{\ell,j})$.
To see this, note that for each individual $j$ we have
	$$\pr\brk{\hat\SIGMA_{ij}(\ell)=\hat\TAU_{ij}(\ell)=0}=1-2q_{\ell,j}+q^{11}_{\ell,j}$$
by inclusion/exclusion, and these events are independent for $j\in\brk k$.
The second one is due to independence over $\ell$ and $i$.
\qed\end{proof}

\begin{lemma}\label{Lemma_intoBinomials}
For any $\vec q$ and any $\ell,j$ we have
	\begin{equation}\label{eqPrCellj}
	\hat\pr\brk{C_{\ell,j}}=\hat\pr\brk{\Bin(m(\ell),q^{11}_{\ell,j})=\omega_{\ell,j}m(\ell)+O(1)}.
	\end{equation}
Furthermore, if $q^{11}(\ell,j)<1$ then
	\begin{eqnarray*}
	\hat\pr\brk{B_{\ell,j}|C_{\ell,j}}&=&\Theta(n^{-1/2})\cdot\hat\pr\brk{\Bin\bc{(1-\omega_{\ell,j})m(\ell),\frac{1-2q_{\ell,j}+q^{11}_{\ell,j}
				}{1-q^{11}_{\ell,j}}}=m(\ell)(1-2\ell_j+\omega_{\ell,j})}.
	\end{eqnarray*}
\end{lemma}
\begin{proof}
Recall that $C_{\ell,j}$ is the event that
	$$\sum_{i\in\brk{m(\ell)}}\hat\SIGMA_{ij}(\ell)\cdot\hat\TAU_{ij}(\ell)=\omega_{\ell,j}m(\ell)+O(1).$$
By construction, the random variables $\hat\SIGMA_{ij}(\ell)\cdot\hat\TAU_{ij}(\ell)$ are independent $\Be(q^{11}_{\ell,j})$ variables,
and thus their sum has distribution $\Bin(m(\ell),q^{11}_{\ell,j})$.
Hence we get~(\ref{eqPrCellj}).

Furthermore, once we condition on the event $C_{\ell,j}$, the remaining $(1-\omega_{\ell,j})m(\ell)$ pairs
	$(\hat\SIGMA_{ij}(\ell),\hat\TAU_{ij}(\ell))$ are chosen conditional on the outcome being different from $(1,1)$.
Hence, by construction each such pair takes the value $(0,0)$ with probability 
	$\frac{1-2q_{\ell,j}+q^{11}_{\ell,j}}{1-q^{11}_{\ell,j}}$ independently (with the numerator resulting from~(\ref{eqqvector3})).
In effect, the probability that the total number of $(0,0)$s equals $m(\ell)(1-2\ell_j+\omega_{\ell,j})$ is just
	$$\pr\brk{\Bin\bc{(1-\omega_{\ell,j})m(\ell),\frac{1-2q_{\ell,j}+q^{11}_{\ell,j}
				}{1-q^{11}_{\ell,j}}}=m(\ell)(1-2\ell_j+\omega_{\ell,j})+O(1)}.$$
Now, given that both this event and $C_{\ell,j}$ occur, the remaining $2(\ell_j-\omega)m(\ell)$ pairs $(\hat\SIGMA_{ij}(\ell),\hat\TAU_{ij}(\ell))$
come up either $(1,0)$ or $(0,1)$ with probability $1/2$.
By Stirling's formula, the probability that both outcomes occur an equal number of times is $\Theta(n^{-1/2})$.
\qed\end{proof}

Note that
	\begin{eqnarray}\label{eqprB'C'independent}
	\hat\pr\brk{B', C'}&=&\prod_{\ell\in\cL'}\hat\pr\brk{B(\ell)\cap C(\ell)}=\prod_{\ell\in\cL'}\prod_{j=1}^k
			\hat\pr\brk{B(t_j,\ell)\cap C(t_j,\ell)}
	\end{eqnarray}
because under $\hat\pr$ the components of the vector $(\hat\SIGMA_{ij}(\ell),\hat\TAU_{ij}(\ell))_{\ell,i,j}$ are independent.

\begin{lemma}\label{Lemma_11fixedpoint}
There exists a vector $\vec q$ such that
	\begin{eqnarray}		\label{eqq}
	\ell_j&=&\frac{q_{\ell,j}-(q_{\ell,j}-q^{11}_{\ell,j})\prod_{h\neq j}(1-q_{\ell,h})
			}{1-2\prod_{h=1}^k(1-q_{\ell,h})+\prod_{h=1}^k(1-2q_{\ell,h}+q^{11}_{\ell,h})},\\
	\omega_{\ell,j}&=&\frac{q^{11}_{\ell,j}}{1-2\prod_{h=1}^k(1-q_{\ell,h})+\prod_{h=1}^k(1-2q_{\ell,h}+q^{11}_{\ell,h})}.
		\label{eqq11}
	\end{eqnarray}
for all $\ell\in\cL',j\in\brk k$.
This vector $\vec q$ satisfies 
	$$q_{\ell,j}=\ell_j-2^{-k-1}+\tilde O(2^{-3k/2})\mbox{ and }q^{11}_{\ell,j}=\omega_{\ell,j}+O(2^{-k}).$$
\end{lemma}
\begin{proof}
This follows from applying the inverse function theorem in a similar way as in the proof of \Lem~\ref{Lemma_fixedpoint}.
\qed\end{proof}

\medskip\noindent
{\bf\em In the rest of this section, we fix $\vec q$ as in \Lem~\ref{Lemma_11fixedpoint}.}

\begin{lemma}\label{Lemma_Pl}
Let
	\begin{eqnarray*}
	\cP_\ell(\omega)&=&
		\ln\brk{1-2\prod_{j=1}^k(1-q_{\ell,j})+\prod_{j=1}^k(1-2q_{\ell,j}+q^{11}_{\ell,j})}\\
		&&-\sum_{j\in\brk k}\brk{\psi(q^{11}_{\ell,j},\omega_{\ell,j})+
				(1-\omega_{\ell,j})\psi\bc{\frac{1-2q_{\ell,j}+q^{11}_{\ell,j}}{1-q^{11}_{\ell,j}},
						\frac{1-2\ell_j+\omega_{\ell,j}}{1-\omega_{\ell,j}}}}.
	\end{eqnarray*}
Furthermore, let
	\begin{equation}\label{eqPomega}
	\cP(\omega)=\sum_{\ell\in\cL'}\frac{m(\ell)}n\cP_\ell(\omega).
	\end{equation}
Then
	$$\hat\pr\brk{S'|B',C'}=\exp\brk{n \cP(\omega)+O(1)}.$$
\end{lemma}
\begin{proof}
The choice of $\vec q$ ensures that for any $\ell$ and $j$,
	\begin{eqnarray}\label{eqLLTsmm1}
	\hat\Erw\brk{\sum_{i\in\brk{m(\ell)}}\hat\SIGMA_{ij}(\ell)\cdot\hat\TAU_{ij}(\ell)\bigg|S'}
		&=&\frac{m(\ell)q^{11}_{\ell,j}}{1-2\prod_{h=1}^k(1-q_{\ell,h})+\prod_{h=1}^k(1-2q_{\ell,h}+q^{11}_{\ell,h})}
			=\omega_{\ell,j}m(\ell);
	\end{eqnarray}
indeed, by~(\ref{eqSatProb}) the denominator in the middle term equals the probability of the event $S_i(\ell)$.
Furthermore, by construction for any $i,j,\ell$ we have
	$$\hat\pr\brk{\hat\SIGMA_{ij}(\ell)=1,\hat\TAU_{ij}(\ell)=0,S_i(\ell)}=q^{10}_{\ell,j}\bc{1-\prod_{h\neq j}(1-q_{\ell,j})}
		=(q_{\ell,j}-q^{11}_{\ell,j})\bc{1-\prod_{h\neq j}(1-q_{\ell,j})}.$$
As a consequence, (\ref{eqq}) ensures that
	\begin{eqnarray}\label{eqLLTsmm2}
	\hat \Erw\brk{\sum_{i\in\brk{m(\ell)}}\hat\SIGMA_{ij}(\ell)|S'}
		&=&\frac{q_{\ell,j}-(q_{\ell,j}-q^{11}_{\ell,j})\prod_{h\neq j}(1-q_h)
			}{1-2\prod_{h=1}^k(1-q_{\ell,h})+\prod_{h=1}^k(1-2q_{\ell,h}+q^{11}_{\ell,h})}=\ell_jm(\ell).
		\end{eqnarray}
By inclusion/exclusion, we obtain from~(\ref{eqLLTsmm1}) and~(\ref{eqLLTsmm2}) that
	\begin{eqnarray}\label{eqLLTsmm3}
	\hat \Erw\brk{\sum_{i\in\brk{m(\ell)}}(1-\hat\SIGMA_{ij}(\ell))\cdot(1-\hat\TAU_{ij}(\ell))\bigg|S'}
		&=&(1-2\ell_j+\omega_{\ell,j})m(\ell).
		\end{eqnarray}
Due to~(\ref{eqLLTsmm1}) and~(\ref{eqLLTsmm3}), 
a repeated application of \Lem~\ref{lem:locallimit} (the local limit theorem) yields
	\begin{eqnarray}\label{eqLLTsmm4}
	\hat\pr\brk{B',C'|S'}&=&\Theta(n^{-3k|\cL'|/2}).
	\end{eqnarray}

Invoking \Lem~\ref{Lemma_intoBinomials}
and using the large deviations principle for the binomial distribution (\Lem~\ref{Lemma_binomial}),
we can easily determine the \emph{unconditional} probability of $B'\cap C'$: we have
	\begin{eqnarray*}
	\hat\pr\brk{B',C'}&=&\prod_{\ell,j}\hat\pr\brk{C'_{\ell,j}}\hat\pr\brk{B'_{\ell,j}|C'_{\ell,j}}\\
		&\hspace{-2cm}=&\hspace{-1cm}\Theta(n^{-k|\cL'|/2})\prod_{\ell,j}\hat\pr\brk{\Bin(m(\ell),q^{11}_{\ell,j})=\omega_{\ell,j}m(\ell)}\\
		&&\qquad\qquad\qquad\qquad		\cdot\hat\pr\brk{\Bin\bc{(1-\omega_{\ell,j})m(\ell),\frac{1-2q_{\ell,j}+q^{11}_{\ell,j}
				}{1-q^{11}_{\ell,j}}}=m(\ell)(1-2\ell_j+\omega_{\ell,j})}\\
		&\hspace{-2cm}=&\hspace{-1cm}\Theta(n^{-3k|\cL'|/2})
			\exp\brk{\sum_{\ell,j}m(\ell)\brk{\psi(q^{11}_{\ell,j},\omega_{\ell,j})+
				(1-\omega_{\ell,j})\psi\bc{\frac{1-2q_{\ell,j}+q^{11}_{\ell,j}}{1-q^{11}_{\ell,j}},
						\frac{1-2\ell_j+\omega_{\ell,j}}{1-\omega_{\ell,j}}}}}.
	\end{eqnarray*}
Thus, 
	\begin{eqnarray*}
	\ln\hat\pr\brk{S'|B',C'}&=&\ln\bcfr{\hat\pr\brk{S'}\hat\pr\brk{B',C'|S'}}{\hat\pr\brk{B',C'}}\\
		&=&O(1)+\ln\hat\pr\brk{S'}\\
		&&\quad	-\sum_{\ell,j}m(\ell)\brk{\psi(q^{11}_{\ell,j},\omega_{\ell,j})+
				(1-\omega_{\ell,j})\psi\bc{\frac{1-2q_{\ell,j}+q^{11}_{\ell,j}}{1-q^{11}_{\ell,j}},
						\frac{1-2\ell_j+\omega_{\ell,j}}{1-\omega_{\ell,j}}}}.
	\end{eqnarray*}
The assertion follows by plugging in the expression for $\pr\brk{S'}$ from \Lem~\ref{Lemma_SATsmm}.
\qed\end{proof}

Finally, \Prop~\ref{Prop_functionF} follows from Fact~\ref{Fact_connection} and \Lem~\ref{Lemma_Pl}.

\subsection{Proof of \Prop~\ref{Prop_calculus1}}\label{Sec_calculus1}

We start with the following observation.

\begin{lemma}\label{Lemma_peak}
Let $\vec q$ be the solution to~(\ref{eqq}) and (\ref{eqq11}) for $\omega=\omega^*$.
There is $\gamma=\gamma(k)>0$ such that for any $\eps>0$ and any $\ell\in\cL'$ we have
	\begin{eqnarray}\label{eqLemmapeak1}
	\hat\pr\brk{\norm{\hat{\vec\omega}_\ell-\omega^*_\ell}_\infty>\eps\,|\,S_\ell,\,B_\ell}&\leq&\exp(-\gamma\eps^2n+o(n))\qquad\mbox{ and}\\
	\hat\pr\brk{\norm{\hat{\vec\omega}_\ell-\omega^*_\ell}_\infty>\eps\,|\,B_\ell}&\leq&\exp(-\gamma\eps^2n+o(n)).
		\label{eqLemmapeak2}
	\end{eqnarray}
\end{lemma}
\begin{proof}
Equation~(\ref{eqLLTsmm4}) from the proof of \Lem~\ref{Lemma_Pl} shows that
	\begin{equation}\label{eqpeak1}
	\hat\pr\brk{B_\ell|S_\ell}=\exp(o(n)).
	\end{equation}
Therefore, it is going to be sufficient to estimate $\hat\pr\brk{\norm{\hat{\vec\omega}_\ell-\omega^*_\ell}_\infty>\eps\,|\,S_\ell}$.
If we just condition on the event $S_\ell$, then the $k$-tuples 
	$(\hat\SIGMA_{ij}(\ell),\hat\TAU_{ij}(\ell))_{j\in\brk k}$
of $0/1$ pairs are mutually independent for all $i\in\brk{m(\ell)}$.
Furthermore, given $S_\ell$ modifying just one such $k$-tuple can alter any entry $\hat{\vec\omega}_{\ell,j}$ by at
most $c/n$, for some number $c=c(k)>0$.
Therefore, Azuma's inequality yields
	\begin{equation}\label{eqpeak2}
	\hat\pr\brk{|\hat\OMEGA_{\ell,j}-\Erw\brk{\hat\OMEGA_{\ell,j}}|>\eps|S_\ell}\leq2\exp(-\gamma\eps^2n),
	\end{equation}
for some $\gamma=\gamma(k)>0$.
Since~(\ref{eqq11}) ensures that $\hat\Erw\brk{\hat\OMEGA_\ell|S_\ell}=\omega_\ell^*$,
(\ref{eqLemmapeak1}) follows from~(\ref{eqpeak1}), (\ref{eqpeak2}) and the union bound.

To obtain (\ref{eqLemmapeak2}), let $\vec q'$ be the vector with entries $\vec q'_{\ell,j}=p(\ell_j)$ for all $\ell,j$.
Then 
	\begin{equation}\label{eqpeak3}
	\hat\pr_{\vec q'}\brk{B_\ell}=\exp(o(n)).
	\end{equation}
Furthermore, applying Azuma's inequality just as in the previous paragraph, we find that
	\begin{equation}\label{eqpeak4}
	\hat\pr_{\vec q'}\brk{|\hat\OMEGA_{\ell,j}-\Erw\brk{\hat\OMEGA_{\ell,j}}|>\eps}\leq2\exp(-\gamma\eps^2n)
	\end{equation}
for some $\gamma=\gamma(k)>0$.
Moreover, $\hat\Erw_{\vec q'}\brk{\hat\OMEGA_\ell}=\omega_\ell^*$ by the choice of $\vec q'$.
Thus, 
(\ref{eqLemmapeak2}) follows from~(\ref{eqpeak3}), (\ref{eqpeak4}) and the union bound.
\qed\end{proof}

\medskip\noindent
\emph{Proof of \Prop~\ref{Prop_calculus1}.}
Let $\ell\in\cL'$.
Let $\vec q$ be the solution to~(\ref{eqq}) and (\ref{eqq11}) for $\omega=\omega^*$.
Then $\hat\pr\brk{\,\cdot\,|B'}$ is the uniform  distribution over pairs $(\hat\sigma,\hat\tau)\in\Omega$
such that $(\hat\sigma,\hat\tau)\in B'$.
Indeed, for $\omega=\omega^*$ the solution $\vec q$ to~(\ref{eqq}) and (\ref{eqq11}) satisfies
$q^{11}_{\ell,j}=q_{\ell,j}^2$ for all $\ell,j$.
Therefore, for any $(\hat\sigma,\hat\tau)\in\Omega$ we have
	\begin{equation}\label{eqPropcalculus1}
	\hat\pr\brk{\hat\SIGMA=\hat\sigma,\,\hat\TAU=\hat\tau}=q_{\ell,j}^{\sum_{\ell,i,j}\hat\sigma_{i,j}(\ell)+\hat\tau_{i,j}(\ell)}
			(1-q_{\ell,j})^{km-\sum_{\ell,i,j}\hat\sigma_{i,j}(\ell)+\hat\tau_{i,j}(\ell)}
	\end{equation}
Since the sums $\sum_{\ell,i,j}\hat\sigma_{i,j}(\ell)+\hat\tau_{i,j}(\ell)$ coincide for all $\hat\sigma,\hat\tau\in B'$,
(\ref{eqPropcalculus1}) shows that
$\hat\pr\brk{\,\cdot\,|B'}$ is uniform.

Let $H(\omega)$ be the number of pairs $(\hat\sigma,\hat\tau)\in\hat\Omega$
such $(\hat\sigma,\hat\tau)\in B'$ and $\hat\OMEGA(\hat\sigma,\hat\tau)=\omega$.
We claim that 
	\begin{equation}\label{eqPropcalculus2}
	\frac1nD\ln H(\omega^*)=o(1).
	\end{equation}
This can be verified either by representing $H(\omega)$ as a product of binomial coefficients and applying Stirlings formula or,
alternatively, by using~(\ref{eqLemmapeak2}).
Indeed, assume that (\ref{eqPropcalculus2}) is false.
Then for small enough $\eps>0$ there is $\delta>0$ such that for some $\omega'$ with $\norm{\omega'-\omega^*}_\infty\sim\eps$ we have
	\begin{equation}\label{eqPropcalculus3}
	\ln H(\omega')\geq\delta n+\max_{\omega:\norm{\omega-\omega^*}_\infty<\eps/2}\ln H(\omega)
	\end{equation}
(with both $\eps,\delta$ possibly dependent on $k$ but not on $n$).
Letting
	$$\bar H=\sum_{(\hat\sigma,\hat\tau)\in B'}H(\hat\omega(\hat\sigma,\hat\tau)),$$
we obtain from~(\ref{eqLemmapeak2}) that
	\begin{eqnarray}\nonumber
	1&\sim&\hat\pr\brk{\norm{\hat\OMEGA-\omega^*}_\infty<\eps/2|B'}
		=\frac1{\bar H}
			\sum_{(\hat\sigma,\hat\tau)\in B'}\vecone_{\norm{\hat\omega(\hat\sigma,\hat\tau)-\omega^*}_\infty<\eps/2}\cdot
					H_{\hat\omega(\hat\sigma,\hat\tau)}\\
		&=&\exp(o(n))\cdot\max_{\omega:\norm{\omega-\omega^*}_\infty<\eps/2}H(\omega)/\bar H.
			\label{eqPropcalculus4}
	\end{eqnarray}
However, combining~(\ref{eqPropcalculus3}) and~(\ref{eqPropcalculus4}) we get
	\begin{eqnarray*}
	\hat\pr\brk{\norm{\hat\OMEGA-\omega^*}_\infty>\eps/2|B'}&\geq&H(\omega')/\bar H
		\geq\exp(\delta n)\max_{\omega:\norm{\omega-\omega^*}_\infty<\eps/2}H(\omega)/\bar H\\
		&\geq&\exp(\delta n-o(n))\hat\pr\brk{\norm{\hat\OMEGA-\omega^*}_\infty<\eps/2|B'}
			>1,
	\end{eqnarray*}
which is a contradiction.
Hence, (\ref{eqPropcalculus2}) follows.
	
Now, assume for contradiction that $D\cP_\ell(\omega^*)\neq0$.
Because the function $\cP_\ell(\,\cdot\,)$ remains fixed as $n\ra\infty$, there exists a fixed
	$\eps'>0$ such that $\norm{D\cP_\ell(\omega^*)}_\infty>\eps'$.
Therefore, (\ref{eqPropcalculus2}) entails that for any $\eps>0$ small enough exist $\omega'$, $\delta>0$ such that
	$\norm{\omega'-\omega^*}_\infty\sim\eps$ and
	\begin{eqnarray}
	\ln H(\omega')+n\cdot\cP_\ell(\omega')
		&\geq&\delta n+\max_{\omega:\norm{\omega-\omega^*}_\infty<\eps/2}
			\ln H(\omega)+n\cdot\cP_\ell(\omega),
			\label{eqPropcalculus5}
	\end{eqnarray}
with $\eps,\delta$ independent of $n$.
Let
	$$\bar H_\ell=\sum_{(\hat\sigma,\hat\tau)\in B'}H(\hat\omega(\hat\sigma,\hat\tau))\exp\brk{n\cP_\ell(\hat\omega(\hat\sigma,\hat\tau))}.$$
Then by~(\ref{eqLemmapeak1}),
		\begin{eqnarray}\nonumber
	1&\sim&\hat\pr\brk{\norm{\hat\OMEGA_\ell-\omega^*_\ell}_\infty<\eps/2|S_\ell,B'}\\
		&=&\frac1{\bar H}
			\sum_{(\hat\sigma,\hat\tau)\in B'}\vecone_{\norm{\hat\omega(\hat\sigma,\hat\tau)-\omega^*}_\infty<\eps/2}\cdot
					H_{\hat\omega(\hat\sigma,\hat\tau)}\exp\brk{n\cP_\ell(\hat\omega(\hat\sigma,\hat\tau))+O(1)}\nonumber\\
		&=&\exp(o(n))\cdot\max_{\omega:\norm{\omega-\omega^*}_\infty<\eps/2}H(\omega)\exp(n\cP_\ell(\omega))/\bar H_\ell.
			\label{eqPropcalculus6}
	\end{eqnarray}
However, combining~(\ref{eqPropcalculus5}) and~(\ref{eqPropcalculus6}) we get
	\begin{eqnarray*}
	\hat\pr\brk{\norm{\hat\OMEGA_\ell-\omega_\ell^*}_\infty>\eps/2|S_\ell,B'}&\geq&\frac{H(\omega')\exp\brk{n\cP_\ell(\omega')+O(1)}}{\bar H_\ell}\\
		&\geq&\exp(\delta n)\max_{\omega:\norm{\omega-\omega^*}_\infty<\eps/2}H(\omega)\exp(n\cP_\ell(\omega))/\bar H_\ell\\
		&\geq&\exp(\delta n-o(n))\hat\pr\brk{\norm{\hat\OMEGA-\omega^*}_\infty<\eps/2|S_\ell,B'}
			>1.
	\end{eqnarray*}
This contradiction shows that $D\cP_\ell(\omega^*)=0$ for all $\ell$.
\qed

\subsection{Proof of \Prop~\ref{Prop_calculus2}}\label{Sec_calculus2}

We need to compute the second derivative of $\cP_\ell$.
In particular, we also need to differentiate $\vec q=\vec q(\omega)$ 
the solution to~(\ref{eqq})--(\ref{eqq11}).
Furthermore, we fix some type $\ell\in\cL$ for the rest of this section.
Let $\cW_\ell$ denote the set of all vectors $\omega_\ell$ such that
	$\abs{\omega_{\ell,j}-\frac14}\leq k^{-4}$ for all $j\in\brk k$.
In \Sec~\ref{Sec_qdiff} we are going to establish the following.

\begin{lemma}\label{Lemma_qdiff}
On $\cW_\ell$ we have
	\begin{eqnarray*}
	\frac{\partial q^{11}_{\ell,h}}{\partial\omega_{\ell,i}}=\vecone_{h=i}
			+\tilde O(2^{-k}),&&
		\frac{\partial^2 q^{11}_{\ell,h}}{\partial\omega_{\ell,i}\partial\omega_{\ell,j}}=\tilde O(2^{-k}),\\
	\frac{\partial q_{\ell,h}}{\partial\omega_{\ell,i}}=\tilde O(2^{-k}),&&
		\frac{\partial^2 q_{\ell,h}}{\partial\omega_{\ell,i}\partial\omega_{\ell,j}}=\tilde O(2^{-k}).
	\end{eqnarray*}
for any $h,i,j\in\brk k$.
\end{lemma}

We split the function $\cP_\ell$ into a sum of various contributions: let
	\begin{eqnarray*}
	\phi_\ell(q)&=&\ln\brk{1-2\prod_{j=1}^k(1-q_{\ell,j})+\prod_{j=1}^k(1-2q_{\ell,j}+q^{11}_{\ell,j})}\quad\mbox{and}\\
	\psi_\ell(\omega,q)&=&\sum_{j\in\brk k}\psi_{\ell,j}(\omega,q)+\tilde\psi_{\ell,j}(\omega,q)\qquad\mbox{with}\\
	\psi_{\ell,j}(\omega,q)&=&\psi(q^{11}_{\ell,j},\omega_{\ell,j}),\\
	\tilde\psi_{\ell,j}(\omega,q)&=&
		(1-\omega_{\ell,j})\psi\bc{\frac{1-2q_{\ell,j}+q^{11}_{\ell,j}}{1-q^{11}_{\ell,j}},
						\frac{1-2\ell_j+\omega_{\ell,j}}{1-\omega_{\ell,j}}}.
	\end{eqnarray*}

\begin{lemma}\label{Lemma_qdiff2}
On $\cW_\ell$ we have
	$$\frac{\partial^2\phi_\ell(\vec q)}{\partial \omega_{\ell,h}\,\partial \omega_{\ell,j}}\leq\tilde O(4^{-k})\qquad\mbox{for all }h,j\in\brk k.$$
\end{lemma}
\begin{proof}
By \Lem~\ref{Lemma_11fixedpoint} for all $\omega\in\cW_\ell$ we have 
 $|q_{\ell,j}-1/2|\leq 1/k^2$ and $|q_{\ell,j}^{11}-1/4|\leq 1/k^2$ for all $j\in\brk k$.
For such vectors $\vec q_\ell$ we obtain the bounds 
	\begin{eqnarray*}
	\frac{\partial\phi_\ell}{\partial q_{\ell,j}},
		\frac{\partial^2\phi_\ell}{\partial q_{\ell,j}\partial q_{\ell,h}},\frac{\partial^2\phi_\ell}{\partial q_{\ell,j}^{11}\partial q_{\ell,h}}
			&=&\tilde O(2^{-k}),\\
		\frac{\partial\phi_\ell}{\partial q_{\ell,j}^{11}},
			\frac{\partial^2\phi_\ell}{\partial q_{\ell,j}^{11}\partial q_{\ell,h}^{11}}&=&\tilde O(4^{-k})
	\end{eqnarray*}
for all $i,j,h\in\brk k$.
Therefore, the assertion follows from \Lem~\ref{Lemma_chainrule} (the chain rule) and \Lem~\ref{Lemma_qdiff}.
\qed\end{proof}

Let $\eps>0$.
We say that $\Psi\in C^2((0,1)^2,\RR)$ is \emph{$\eps$-tame} on $\cY\subset(0,1)^2$ if  the following conditions hold:
	\begin{description}
	\item[T1.] For all $y\in(0,1)$ we have $\Psi(y,y)=0$.
	\item[T2.] On $\cY$  we have
			$\abs{\sum_{i=1}^2\frac{\partial^2\Psi}{\partial z_i\partial z_j}}\leq\eps$ for any $j=1,2$.
	\item[T3.] On $\cY$  we have
		$\abs{\sum_{i,j=1}^2\frac{\partial^2\Psi}{\partial z_i\partial z_j}}\leq\eps^2$.
	\item[T4.] On $\cY$  we have
			$|\frac{\partial^2\Psi}{\partial z_i\partial z_j}|\leq 100$ for any $i,j=1,2$.
	\end{description}

Let $f:(0,1)^k\ra\RR^2$, $(z_1,\ldots,z_k)\mapsto(f_1(z_1,\ldots,z_k),f_2(z_1,\ldots,z_k))$ be a $C^2$-function.
We say that $f$ is \emph{$\eps$-benign} on $\cW$ if the following statements are true  on $\cW$:
\begin{description}
\item[B1.] $\abs{\frac{\partial f_1}{\partial z_1}-\frac{\partial f_2}{\partial z_1}}<\eps$.
\item[B2.] $\abs{\frac{\partial f_i}{\partial z_j}}<\eps$ for any $1<j\leq k$ and $i=1,2$ and $\abs{\frac{\partial f_i}{\partial z_1}}\leq100$.
\item[B3.] $\abs{\frac{\partial^2 f_i}{\partial z_h\partial z_j}}<\eps$ for any $i$ and $(h, j)\neq(1,1)$.
\item[B4.] 	$\abs{\frac{\partial^2 f_1}{\partial z_1^2}-\frac{\partial^2 f_2}{\partial z_1^2}}<\eps$
		and $|\frac{\partial^2 f_1}{\partial z_1^2}|\leq100$.
\end{description}

\begin{lemma}\label{Lemma_tameAndBenign}
There is an absolute constant $C>0$ such that the following is true.
Assume that
$f$ is $\eps$-benign on $\cW$ and that
 $\Psi$ is $\eps$-tame on $f(\cW)$.
Then on $\cW$ we have
	$$\frac{\partial^2 \Psi\circ f}{\partial z_i\partial z_j}\leq C\eps^2\qquad\mbox{for any $i,j\in\brk k$}.$$
\end{lemma}
\begin{proof}
By \Lem~\ref{Lemma_chainrule} (the chain rule), we have
	\begin{eqnarray*}
	\frac{\partial^2 \Psi\circ f}{\partial z_i\partial z_j}&=&\sum_{h=1}^2\frac{\partial \Psi}{\partial y_h}\frac{\partial^2f_h}{\partial z_i\partial z_j}
		+\sum_{a,b=1}^2\frac{\partial^2\Psi}{\partial y_a\partial y_b}\frac{\partial f_a}{\partial z_i}\frac{\partial f_b}{\partial z_j}.
	\end{eqnarray*}
Since by {\bf T4} and Taylor's formula we have $\frac{\partial \Psi}{\partial y_h}=O_k(\eps)$, {\bf B3} implies that for $(i,j)\neq(1,1)$
	$$\sum_{h=1}^2\frac{\partial \Psi}{\partial y_h}\frac{\partial^2f_h}{\partial z_i\partial z_j}=O_k(\eps^2).$$
Furthermore, as $\frac{\partial \Psi}{\partial y_h}=O_k(\eps)$, {\bf B4} yields
	\begin{eqnarray*}
	\sum_{h=1}^2\frac{\partial \Psi}{\partial y_h}\frac{\partial^2f_h}{\partial z_1^2}
		&=&\frac{\partial^2f_1}{\partial z_1^2}\sum_{h=1}^2\frac{\partial \Psi}{\partial y_h}
			+\sum_{h=1}^2\frac{\partial \Psi}{\partial y_h}\brk{\frac{\partial^2f_h}{\partial z_1^2}-\frac{\partial^2f_1}{\partial z_1^2}}
			=O_k(1)\sum_{h=1}^2\frac{\partial \Psi}{\partial y_h}+O_k(\eps^2)=O_k(\eps^2);
	\end{eqnarray*}
the last step follows from {\bf T2} and Taylor's formula.

To deal with the second sum, we consider four cases.
\begin{description}
\item[Case 1: $i\neq1,j\neq1$.] By {\bf B2} we have $\frac{\partial f_a}{\partial z_i}\frac{\partial f_b}{\partial z_j}\leq O_k(\eps^2)$, and thus
		$$\frac{\partial^2\Psi}{\partial y_a\partial y_b}\frac{\partial f_a}{\partial z_i}\frac{\partial f_b}{\partial z_j}=O_k(\eps^2)$$
		by {\bf T4}.
\item[Case 2: $i=1,j\neq1$.]
	We have
		\begin{eqnarray*}
		\sum_{a,b=1}^2\frac{\partial^2\Psi}{\partial y_a\partial y_b}\frac{\partial f_a}{\partial z_1}\frac{\partial f_b}{\partial z_j}
			&=&\sum_{b=1}^2\frac{\partial f_b}{\partial z_j}\sum_{a=1}^2\frac{\partial^2\Psi}{\partial y_a\partial y_b}\frac{\partial f_a}{\partial z_1}
			\,\stacksign{\mbox{\bf B2}}{=}\,
				\sum_{b=1}^2O_k(\eps)
					\sum_{a=1}^2\frac{\partial^2\Psi}{\partial y_a\partial y_b}\frac{\partial f_a}{\partial z_1}\\
			&\stacksign{\mbox{\bf B1, T4}}{=}&O_k(\eps^2)+\frac{\partial f_1}{\partial z_1}\sum_{b=1}^2O_k(\eps)
					\sum_{a=1}^2\frac{\partial^2\Psi}{\partial y_a\partial y_b}\\
			&\stacksign{\mbox{\bf B2}}{=}&O_k(\eps^2)+\sum_{b=1}^2O_k(\eps)
					\sum_{a=1}^2\frac{\partial^2\Psi}{\partial y_a\partial y_b}
					\,\stacksign{\mbox{\bf T2}}{=}\,O_k(\eps^2).
		\end{eqnarray*}
\item[Case 3: $i\neq1,j=1$.] The same argument as in case~2 applies.
\item[Case 4: $i=j=1$.] 
	We have
		\begin{eqnarray*}
		\sum_{a,b=1}^2\frac{\partial^2\Psi}{\partial y_a\partial y_b}\frac{\partial f_a}{\partial z_1}\frac{\partial f_b}{\partial z_1}
			&=&
			\bcfr{\partial f_1}{\partial z_1}^2\sum_{a,b=1}^2\frac{\partial^2\Psi}{\partial y_a\partial y_b}+
					\sum_{a,b=1}^2\frac{\partial^2\Psi}{\partial y_a\partial y_b}\brk{\frac{\partial f_a}{\partial z_1}\frac{\partial f_b}{\partial z_1}
							-\bcfr{\partial f_1}{\partial z_1}^2}\\
			&\hspace{-4cm}\stacksign{\mbox{\bf B2, T3}}{=}&\hspace{-2cm}O_k(\eps^2)+	
				\sum_{a,b=1}^2\frac{\partial^2\Psi}{\partial y_a\partial y_b}\frac{\partial f_a}{\partial z_1}\brk{\frac{\partial f_b}{\partial z_1}
							-\frac{\partial f_1}{\partial z_1}}
								+\sum_{a,b=1}^2\frac{\partial^2\Psi}{\partial y_a\partial y_b}
										\frac{\partial f_1}{\partial z_1}\brk{\frac{\partial f_a}{\partial z_1}
							-\frac{\partial f_1}{\partial z_1}}\\
			&\hspace{-4cm}\stacksign{\mbox{\bf B1}}{=}&\hspace{-2cm}O_k(\eps^2)+	
				\sum_{b=1}^2O_k(\eps)\sum_{a=1}^2
						\frac{\partial^2\Psi}{\partial y_a\partial y_b}\frac{\partial f_a}{\partial z_1}
								+\sum_{a=1}^2O_k(\eps)\sum_{b=1}^2
										\frac{\partial^2\Psi}{\partial y_a\partial y_b}
										\frac{\partial f_1}{\partial z_1}\\	
			&\hspace{-4cm}\stacksign{\mbox{\bf B1}}{=}&\hspace{-2cm}O_k(\eps^2)
								+\sum_{a=1}^2O_k(\eps)\sum_{b=1}^2
										\frac{\partial^2\Psi}{\partial y_a\partial y_b}
										\frac{\partial f_1}{\partial z_1}
						\,\stacksign{\mbox{\bf B1}}{=}\,O_k(\eps^2)
								+\sum_{a=1}^2O_k(\eps)\sum_{b=1}^2
										\frac{\partial^2\Psi}{\partial y_a\partial y_b}
						\,\stacksign{\mbox{\bf T2}}{=}\,O_k(\eps^2).
		\end{eqnarray*}
\end{description}
Hence, in all cases we obtain a bound of $O_k(\eps^2)$.
\qed\end{proof}

\begin{lemma}\label{Lemma_psiIsTame}
The functions $(y_1,y_2)\mapsto \psi(y_1,y_2)$ and $(y_1,y_2)\mapsto (1-y_1)\psi(y_1,y_2)$ are $\tilde O(2^{-k})$-tame on
	$$\cY=\cbc{(y_1,y_2)\in(0,1)^2:|y_1-y_2|\leq  k^32^{-k},\ \max_{i=1,2}|y_i-1/4|\leq 1/k^2}.$$
\end{lemma}
\begin{proof}
It is straightforward to work out the differentials of $\psi$: we have
	\begin{eqnarray*}
	\frac{\partial\psi}{\partial {y_1}}=\frac{y_2}{y_1}-\frac{1-{y_2}}{1-{y_1}},&&
	\frac{\partial\psi}{\partial {y_2}}=-\ln\bcfr{y_2}{y_1}+\ln\bcfr{1-{y_2}}{1-{y_1}},\\
	\frac{\partial^2\psi}{\partial {y_1}^2}=-\frac{y_2}{{y_1}^2}-\frac{1-{y_2}}{(1-{y_1})^2},&&
	\frac{\partial^2\psi}{\partial {y_1}\partial{y_2}}\psi=\frac 1{{y_1}}+\frac{1}{1-{y_1}},\qquad
	\frac{\partial^2\psi}{\partial {y_2}^2}=-\frac1{y_2}-\frac1{1-{y_2}}.
	\end{eqnarray*}

Differentiating once more with respect to $y_1$, we get
	\begin{eqnarray*}
	\frac{\partial^3\psi}{\partial y_1^3}&=&\frac{2y_2}{y_1^3}-\frac{2(1-y_2)}{(1-y_1)^3},\quad
	\frac{\partial^3\psi}{\partial y_1^2\partial y_2}=-\frac 1{y_1^2}+\frac{1}{(1-y_1)^2},\quad
	\frac{\partial^3\psi}{\partial y_1\partial y_2^2}=0.
	\end{eqnarray*}
Therefore, at $y_1=y_2+\eps$ the second derivatives work out to be
	\begin{eqnarray*}\label{eqd2psiperturbed1}
	\frac{\partial^2\psi}{\partial y_1^2}(y_2+\eps,y_2)&=&-\frac1{y_2}-\frac1{1-{y_2}}+2\eps\bc{\frac1{{y_2}^2}-\frac1{(1-{y_2})^2}}+O(\eps^2),\\
	\frac{\partial^2\psi}{\partial y_1\partial {y_2}}(y_2+\eps,y_2)&=&\frac1{y_2}+\frac1{1-{y_2}}+
			\eps\bc{-\frac1{{y_2}^2}+\frac1{(1-{y_2})^2}}+O(\eps^2),\label{eqd2psiperturbed2}\\
	\frac{\partial^2\psi}{\partial {y_2}^2}(y_2+\eps,y_2)&=&-\frac1{y_2}-\frac1{1-{y_2}}.\label{eqd2psiperturbed3}
	\end{eqnarray*}
Hence, $\psi$ is tame.
Furthermore,
differentiating $(y_1,y_2)\mapsto(1-y_2)\psi(y_1,y_2)$ yields
	\begin{eqnarray*}
	\frac{\partial}{\partial y_1}(1-y_2)\psi(y_1,y_2)&=&
		(1-y_2)\frac{\partial}{\partial y_1}\psi(y_1,y_2),\\
	\frac{\partial}{\partial y_2}(1-y_2)\psi(y_1,y_2)&=&
		(1-y_2)\frac{\partial}{\partial y_2}\psi(y_1,y_2)-\psi(y_1,y_2),\\
	\frac{\partial^2}{\partial y_1^2}(1-y_2)\psi(y_1,y_2)&=&
		(1-y_2)\frac{\partial^2}{\partial y_1^2}\psi(y_1,y_2),\\
	\frac{\partial^2}{\partial y_2^2}(1-y_2)\psi(y_1,y_2)&=&
		(1-y_2)\frac{\partial^2}{\partial y_2^2}\psi(y_1,y_2)-2\frac{\partial}{\partial y_2}\psi(y_1,y_2),\\
	\frac{\partial^2}{\partial y_1\partial y_2}(1-y_2)\psi(y_1,y_2)&=&
		(1-y_2)\frac{\partial^2}{\partial y_1\partial y_2}\psi(y_1,y_2)-\frac{\partial}{\partial y_1}\psi(y_1,y_2).
	\end{eqnarray*}
Hence, the fact that $(1-y_2)\psi(y_1,y_2)$ is $\eps$-tame follows from the fact that $\psi$ is.
\qed\end{proof}

\begin{lemma}\label{Lemma_zetaIsBenign}
With $\vec q=\vec q(\omega)$
the functions
	\begin{eqnarray*}
	\xi_{\ell,j}&:&\omega\mapsto\bc{q_{\ell,j}^{11},\omega_{\ell,j}},\\
	\zeta_{\ell,j}&:&\omega\mapsto
		(\zeta_{1,\ell,j},\zeta_{2,\ell,j})=\bc{\frac{1-2q_{\ell,j}+q^{11}_{\ell,j}}{1-q^{11}_{\ell,j}},\frac{1-2\ell_j+\omega_{\ell,j}}{1-\omega_{\ell,j}}}
	\end{eqnarray*}
are $\tilde O(2^{-k})$-benign on 
$\cW=\cbc{\omega:\norm{\omega-\frac14\vecone}_{\infty}\leq k^{-4}}$.
\end{lemma}
\begin{proof}
The fact that $\xi_{\ell,j}$ is benign follows directly from \Lem~\ref{Lemma_qdiff}.
With respect to $\zeta_{\ell,j}$ we have
	\begin{eqnarray*}
	\frac{\partial\zeta_{2,\ell,j}}{\partial\omega_{\ell,j}}&=&
			\frac{2(1-\ell_j)}{(1-\omega_{\ell,j})^2},
		\qquad\frac{\partial^2\zeta_{2,\ell,j}}{\partial\omega_{\ell,j}^2}=\frac{4(1-\ell_j)}{(1-\omega_{\ell,j})^3},\\
	\frac{\partial\zeta_{2,\ell,j}}{\partial\omega_{\ell,h}}&=&0,
		\qquad\frac{\partial^2\zeta_{2,\ell,j}}{\partial\omega_{\ell,h}\partial\omega_{\ell,i}}=0\qquad
			(h\neq j),\\
	\frac{\partial\zeta_{1,\ell,j}}{\partial\omega_{\ell,j}}&=&
		\frac{(1-q_{\ell,j}^{11})\brk{-2\frac{\partial q_{\ell,j}}{\partial\omega_{\ell,j}}+\frac{\partial q_{\ell,j}^{11}}{\partial\omega_{\ell,j}}}
				+\frac{\partial q_{\ell,j}^{11}}{\partial\omega_{\ell,j}}(1-2q_{\ell,j}+q_{\ell,j}^{11})}
				{(1-q_{\ell,j}^{11})^2}=\frac{2(1-q_{\ell,j})}{(1-q_{\ell,j}^{11})^2}+\tilde O(2^{-k}),\\
	\frac{\partial^2\zeta_{1,\ell,j}}{\partial\omega_{\ell,j}^2}&=&\frac{4(1-q_{\ell,j})}{(1-q_{\ell,j})^4}+\tilde O(2^{-k}),\\
	\frac{\partial\zeta_{1,\ell,j}}{\partial\omega_{\ell,h}}&=&\tilde O(2^{-k}),
		\qquad\frac{\partial^2\zeta_{1,\ell,j}}{\partial\omega_{\ell,h}\partial\omega_{\ell,i}}=\tilde O(2^{-k})\qquad
			(h\neq j).
	\end{eqnarray*}
Since $|q_{\ell,j}-\ell_j|\leq\tilde O(2^{-k})$ and $|q_{\ell,j}^{11}-\omega_{\ell,j}|\leq\tilde O(2^{-k})$ by \Lem~\ref{Lemma_11fixedpoint},
the assertion follows.
\qed\end{proof}

Finally, \Prop~\ref{Prop_calculus2}
follows directly from \Lem s \ref{Lemma_qdiff2}, \ref{Lemma_tameAndBenign}, \ref{Lemma_psiIsTame} and~\ref{Lemma_zetaIsBenign}.

\subsection{Proof of \Lem~\ref{Lemma_qdiff}}\label{Sec_qdiff}

Let
	\begin{eqnarray*}
	P_{\ell,j}&:&\vec q\mapsto\frac{q_{\ell,j}-(q_{\ell,j}-q^{11}_{\ell,j})\prod_{h\neq j}(1-q_{\ell,h})
			}{1-2\prod_{h=1}^k(1-q_{\ell,h})+\prod_{h=1}^k(1-2q_{\ell,h}+q^{11}_{\ell,h})},\\
	\Omega_{\ell,j}&:&\vec q \mapsto\frac{q^{11}_{\ell,j}}{1-2\prod_{h=1}^k(1-q_{\ell,h})+\prod_{h=1}^k(1-2q_{\ell,h}+q^{11}_{\ell,h})}.
	\end{eqnarray*}
A straightforward calculation shows that for $\vec q$ such that $|q_{\ell,j}-1/2|\leq 1/k^2$ and $|q_{\ell,j}^{11}-1/4|\leq 1/k^2$ we have
	\begin{eqnarray*}
	\frac{\partial P_{\ell,j}}{\partial q_{\ell,h}}=\vecone_{j=h}+\tilde O(2^{-k}),&& \frac{\partial P_{\ell,j}}{\partial q^{11}_{\ell,h}}=\tilde O(2^{-k}),\\
	\frac{\partial \Omega_{\ell,j}}{\partial q_{\ell,h}}=\tilde O(2^{-k}),&&
		\frac{\partial \Omega_{\ell,j}}{\partial q^{11}_{\ell,h}}=\vecone_{j=h}+\tilde O(2^{-k})
	\end{eqnarray*}
for any $j,h\in\brk k$.
Let $F:\vec q\mapsto\bink{\bc{P_{\ell,j}(\vec q)}_{j\in\brk k}}{\bc{\Omega_{\ell,j}(\vec q)}_{j\in\brk k}}$.
Then the differential of $F$ satisfies
	\begin{equation}\label{eqDF}
	DF=\brk{\begin{array}{cc}
		{\bc{\bc{\frac{\partial P_{\ell,j}}{\partial q_{\ell,h}}}_{h\in\brk k},\bc{\frac{\partial P_{\ell,j}}{\partial q_{\ell,h}^{11}}}_{h\in\brk k}
			}_{j\in\brk k}}\\
		{\bc{\bc{\frac{\partial P_{\ell,j}}{\partial q_{\ell,h}}}_{h\in\brk k},\bc{\frac{\partial P_{\ell,j}}{\partial q_{\ell,h}^{11}}}_{h\in\brk k}
			}_{j\in\brk k}}
		\end{array}}
			=\id+\tilde O(2^{-k})\vecone,
	\end{equation}
where $\id$ is the matrix with ones on the diagonal and zeros elsewhere, and $\vecone$ signifies the matrix with all entries equal to one.
By the inverse function theorem, we have $D(F^{-1})=(DF)^{-1}$.
Furthermore, by (\ref{eqDF}) and Cramer's rule,
	\begin{equation}\label{eqDFinv}
	(DF)^{-1}=\id+\tilde O(2^{-k})\vecone.
	\end{equation}
Since $\vec q(\omega)$ is the solution to $F(\vec q)=\bink{(p(\ell_j))_{j\in\brk k}}{(\omega_{\ell,j})_{j\in\brk k}}$,
(\ref{eqDFinv}) yields the assertions on the first derivatives
	$\frac{\partial q^{11}_{\ell,h}}{\partial\omega_{\ell,i}}$, $\frac{\partial q_{\ell,h}}{\partial\omega_{\ell,i}}$
 in \Lem~\ref{Lemma_qdiff}.

Proceeding to the second derivative, we highlight the following (folklore) fact.

\begin{lemma}\label{Lemma_determinants}
Let $\eps,\delta=\exp(-\Omega(k))$.
Let $\cA$ be the set of all $k\times k$ matrices $A=(A_{ij})$ such that
	$|A_{ii}-1|<\eps$ for all $i$ and $|A_{ij}|<\delta$ for all $i\neq j$.
Then $A$ is regular and the operator $\inv:A\in\cA\mapsto A^{-1}=(\inv_{st}A)_{s,t=1,\ldots,k}$ satisfies
	$$\frac{\partial\inv_{st}}{\partial a_{ij}}\bigg|_A\leq \tilde O(\delta)-\vecone_{i=j=s=t}(1+\tilde O(\eps))\quad\mbox{for any $i,j,s,t\in\brk k$.}$$
\end{lemma}
\begin{proof}
This is a simple consequence of Cramer's rule.
Indeed, let $A_{ij}'$ be the matrix obtained from $A$ by omitting row $i$ and column $j$.
Then
	$$\inv_{st}A=(-1)^{s+t}\frac{\det A_{ts}'}{\det A}.$$
Thus, we need to differentiate $\det A_{ts}'$ and $\det A$.
For any $i\neq j$ we have 
	\begin{eqnarray*}
	\frac{\partial}{\partial a_{ii}}\det A&=&\prod_{h\neq i} a_{hh}+\tilde O(\delta)=1+\tilde O(\eps)+\tilde O(\delta),\quad
	\frac{\partial}{\partial a_{ij}}\det A=\tilde O(\delta).
	\end{eqnarray*}
Similarly, for $i\neq j$ and $s\neq t$ we have
	\begin{eqnarray*}
	\frac{\partial}{\partial a_{ii}}\det A_{tt}'&=&\vecone_{i\neq t}\cdot(1+\tilde O(\eps)),\quad
	\frac{\partial}{\partial a_{ii}}\det A_{ts}'=\tilde O(\delta),\quad
	\frac{\partial}{\partial a_{ij}}\det A_{ts}'=\tilde O(\delta).
	\end{eqnarray*}
Thus, the assertion follows from the quotient rule.
\qed\end{proof}

A direct calculation shows that 
for $\vec q$ such that $|q_{\ell,j}-1/2|\leq 1/k^2$ and $|q_{\ell,j}^{11}-1/4|\leq 1/k^2$ we have
	\begin{eqnarray*}
	\frac{\partial^2 P_{\ell,j}}{\partial q_{\ell,h}\partial q_{\ell,i}},\frac{\partial^2 P_{\ell,j}}{\partial q_{\ell,h}\partial q_{\ell,i}^{11}},
		\frac{\partial^2 P_{\ell,j}}{\partial q_{\ell,h}^{11}\partial q_{\ell,i}^{11}}
		&=&\tilde O(2^{-k}),\\
	\frac{\partial^2 \Omega_{\ell,j}}{\partial q_{\ell,h}\partial q_{\ell,i}},
		\frac{\partial^2 \Omega_{\ell,j}}{\partial q_{\ell,h}\partial q_{\ell,i}^{11}},
			\frac{\partial^2 \Omega_{\ell,j}}{\partial q_{\ell,h}^{11}\partial q_{\ell,i}^{11}}
		&=&\tilde O(2^{-k})
	\end{eqnarray*}
for any $h,i,j\in\brk k$.
Thus, 
	\begin{equation}\label{eqD2Fgeneral}
	\norm{D^2F}_{\infty}\leq\tilde O(2^{-k}).
	\end{equation}
Because by the chain rule
	$D(\inv\circ DF)=(D\inv)\circ(D^2F)$,
		the assertion on the second derivatives follows from 
			\Lem~\ref{Lemma_determinants}, (\ref{eqDFinv}) and (\ref{eqD2Fgeneral}).

\subsection{Completing the proof of \Prop~\ref{Prop_Pcentre}}\label{Sec_PcentreComplete}

The first assertion is a direct consequence of \Prop~\ref{Prop_functionF} and \Cor~\ref{Cor_calculus}.
Similarly, the second assertion follows from \Prop~\ref{Prop_functionF} because
	$\cP_\ell(\omega)\leq-\Omega_k(2^{-k})$ for all $\ell$.

Finally, let $\omega=\omega^*$.
It is straightforward to verify that by letting $q_{\ell,j}$ be as in \Lem~\ref{Lemma_fixedpoint}
and by setting $q_{\ell,j}^{11}=q_{\ell,j}^2$ we obtain the unique solution to~(\ref{eqq})--(\ref{eqq11}).
We need to plug this solution into $\cP(\omega)$:
we have
	\begin{eqnarray}\nonumber
	\ln\brk{1-2\prod_{j=1}^k(1-q_{\ell,j})+\prod_{j=1}^k(1-2q_{\ell,j}+q^{11}_{\ell,j})}&=&
		\ln\brk{1-2\prod_{j=1}^k(1-q_{\ell,j})+\prod_{j=1}^k(1-q_{\ell,j})^2}\\
		&=&2\ln\bc{1-\prod_{j=1}^k1-q_{\ell,j}}.
			\label{eqdismal1}
	\end{eqnarray}
Moreover, 
	\begin{eqnarray}\nonumber
	\psi(q^{11}_{\ell,j},\omega_{\ell,j})&=&\psi(q_{\ell,j}^2,\ell_j^2)=
		-2\ell_j^2\ln\bcfr{\ell_j}{q_{\ell,j}}-(1-\ell_j^2)\ln\bcfr{1-\ell_j^2}{1-q_{\ell,j}^2}\\
		&=&-2\ell_j^2\ln\bcfr{\ell_j}{q_{\ell,j}}-(1-\ell_j^2)\brk{\ln\bcfr{1-\ell_j}{1-q_{\ell,j}}+\ln\bcfr{1+\ell_j}{1+q_{\ell,j}}}.
						\label{eqdismal2}
	\end{eqnarray}
Further,
	\begin{eqnarray}
	(1-\ell_j^2)\psi\bc{\frac{1-2q_{\ell,j}+q_{\ell,j}^{11}}{1-q_{\ell,j}^{11}},\frac{1-2\ell_j+\omega_{\ell,j}}{1-\omega_{\ell,j}}}&=&
		(1-\ell_j^2)\psi\bc{\frac{(1-q_{\ell,j})^2}{1-q_{\ell,j}^2},\frac{(1-\ell_j)^2}{1-\ell_j^2}}\nonumber\\
		&\hspace{-14cm}=&\hspace{-7cm}
			(1-\ell_j^2)\psi\bc{\frac{1-q_{\ell,j}}{1+q_{\ell,j}},\frac{1-\ell_j}{1+\ell_j}}\nonumber\\
		&\hspace{-14cm}=&\hspace{-7cm}-(1-\ell_j)^2\ln\bcfr{1-\ell_j}{1-q_{\ell,j}}
				-(1-\ell_j^2)\ln\bcfr{1+q_{\ell_j}}{1+\ell_j}-2\ell_j(1-\ell_j)\ln\bcfr{\ell_j}{q_{\ell,j}}.
							\label{eqdismal3}
	\end{eqnarray}
Summing up~(\ref{eqdismal1})--(\ref{eqdismal3}), we find
	\begin{eqnarray*}
	\frac{n\cP(\omega)}2&=&\sum_{\ell\in\cL}m(\ell)\brk{\ln\bc{1-\prod_{j=1}^k1-q_{\ell,j}}-\sum_{j\in\brk k}\psi(q_{\ell,j},\ell_j)}.
	\end{eqnarray*}
Therefore, the third assertion follows from Remark~\ref{Remark_weNeedThis}.

\section{Enumeration of Assignments with $p$-Marginals}
\label{sec:enumeration}

In this section we will prove Lemma~\ref{Lemma_entropy} and Proposition~\ref{Prop_entropy}. Before we present the actual details we will introduce an appropriate framework, which will enable us to perform the enumeration of assignments with $p$-marginals, and pairs of such assignments with a given overlap.

In \Sec~\ref{XSec_theRandomVar} we said that an assignment $\sigma\in\cbc{0,1}^V$ has \emph{$\pd$-marginals} if
 for any type $t\in\cT$ we have
	$$\sum_{l\in L:\cT(l)=t}\vecone_{\sigma(l)=1}\cdot \frac{d_l}{km}\doteq p(t)\pi(t).$$
In words, the fraction of literal occurrences of type $t$ that are true under $\sigma$ equals $p(t)$ up to an error of $O(1/n)$. However, due to technical reasons and because it simplifies some of our calculations significantly, we will actually work with a slightly refined definition. Let us say that a signature $(s, d^+,d^-)$ is \emph{good}, if $d^+, d^- < 3kr/4$ and $0 < (d^+ - d)^2 \le 100 k 2^k \ln k$. Instead of requiring that the fraction of literal occurrences of type $t$ equals $p(t)$, we require that this is true for \emph{every good signature}. That is, we say that an assignment $\sigma\in\cbc{0,1}^V$ has \emph{$\pd$-marginals} if for any good $s\in T$
\[
	\sum_{l\in L: T(l)=s}\vecone_{\sigma(l)=1}\cdot \frac{d_l}{km} = p(s) \, \sum_{l\in L: T(l)=s} \frac{d_l}{km}, 
\]
and moreover, that fraction of literal occurrences of all other variables is $1/2$, i.e.,
\[
		\sum_{l\in L: p(l)=1/2}\vecone_{\sigma(l)=1}\cdot \frac{d_l}{km} = \frac12 \, \sum_{l\in L: p(l)=1/2} \frac{d_l}{km}.
\]
We are going to prove Lemma~\ref{Lemma_entropy} and Proposition~\ref{Prop_entropy} with this modified definition.
It is easily checked that this modification does not affect any of the arguments in the previous sections.

Let $s\in T$ be any signature and set $L_s = \{\ell \in L : T(\ell)=s\}$. Moreover, denote by $V_s = \{|\ell| : \ell \in L_t\}$ and observe that $V_s = V_{\neg s}$. 
For any $\sigma \in \{0,1\}^n$ let us denote by the $s$-weight $w_s(\sigma)$ the number of satisfied literal occurrences, where only literals of signature $s$ are considered, i.e.,
\[
	w_s(\sigma) = \sum_{\ell \in L_s} \vecone_{[\sigma(\ell)=1]} d_\ell.
\]
Let us also define similar quantities with respect to the types. Let $t\in \cal T$ and set, as previously, $L_t = \{\ell \in L : {\cal T}(\ell) = t\}$. Denote by $\neg t\in \cal T$ the type satisfying $p(\neg t) = 1 - p(t)$. Note that $\neg t$ exists, and we have $L_{\neg t} = \{\neg \ell : \ell \in L_t\}$. Moreover, note that if $p(t) \neq 1/2$ we have $L_t \cap L_{\neg t} = \emptyset$, and $L_t = L_{\neg t}$ otherwise. Finally, set $V_t = \{|\ell| : \ell \in L_t\} = \{|\ell| : \ell \in L_{\neg t}\}$. In accordance with the case of signatures, let us for any $\sigma \in \{0,1\}^n$ denote by the $t$-weight $w_t(\sigma)$ the number of satisfied literal occurrences, where only literals of type $t$ are considered, i.e.,
\[
	w_t(\sigma) = \sum_{\ell \in L_t} \vecone_{[\sigma(\ell)=1]} d_\ell.
\]
Let $t_{1/2}$ be the type such that $p(t_{1/2}) = 1/2$. Since $L_{t_{1/2}} = L_{\neg t_{1/2}}$ it follows that in this special case
\begin{equation}
\label{eq:wt12}
	w_{t_{1/2}}(\sigma) = \sum_{v \in V_{t_{1/2}}} \vecone_{[\sigma(v)=1]} d_v + \vecone_{[\sigma(v)=0]} d_{\neg v}.
\end{equation}
With the above notation, an assignment $\sigma$ has $p$-marginals if and only if
\[
	\forall s\in T\setminus t_{1/2}: ~ w_s(\sigma) = p(s) \pi(s) km
	\quad \text{ and }\quad
	w_{t_{1/2}}(\sigma) = \frac12 \pi(t_{1/2}) km.
\]
The next proposition is the first step towards the estimation of the total number of assignments with $p$-marginals, c.f.\ Lemma~\ref{Lemma_entropy}. We denote by $H(x) = -x\ln x - (1-x)\ln(1-x)$ the entropy of $x$, and with $[z^n]f(z)$ the $n$-th coefficient in the Taylor series expansion of an analytic function $f$ around 0.
\begin{proposition}
W.h.p.\ $\bf d$ chosen from $\bf D$ has the following property. There is a constant $C > 0$ such that if we denote by $\cal S$ the set of signatures $s\in T$ with the property $p(s) > 1/2$, then
\begin{equation}
\label{eq:H}
	|{\cal H}|
	= (C+o(1)) \, n^{-|{\cal S}|/2} \, \exp\left \{ \sum_{s\in {\cal S}} |V_s|H(p(s))\right\} \cdot [z^{\pi(t_{1/2})km/2}] \prod_{v \in V_{t_{1/2}}}(z^{d_v} + z^{d_{\neg v}}).
\end{equation}
\end{proposition}
\begin{proof}
First of all, note that if for an assignment $\sigma$ and a signature $s\in T$ with $p(s) > 1/2$ we have $w_s(\sigma) = \pi(s)km$, then the fraction of variables in $V_s$ that are set to true is $p(s)$. Thus, the fraction of variables set to false is $1-p(s) = p(\neg s)$, and we infer that
\[
	w_{\neg s}(\sigma)
	= \sum_{\ell \in L_{\neg s}} \vecone_{[\sigma(\ell)=1]} d_\ell
	= \sum_{v \in V_s} \vecone_{[\sigma(v)=0]} d_{\neg v}
	= p(\neg s) \pi(\neg s) km.
\]
Consequently, for any such $s$ the number of partial assignments $\sigma_s:V_s \to \{0,1\}$, with the property that the fraction of satisfied variables is $p(s)$ is
\[
	\binom{|V_s|}{p(s)|V_s|} = \frac1{\sqrt{2\pi p(s)(1-p(s)) |V_s|}}e^{|V_s| H(p(s))}.
\]
Since w.h.p.\ $\bf d$ is such that $|V_s| = (1+o(1))\alpha_s n$ for some $\alpha_s = \alpha_s(k)$, this provides the exponential terms in~\eqref{eq:H}.

It remains to bound the number of partial assignments $\sigma':V_{t_{1/2}} \to \{0,1\}$ such that $w_{t_{1/2}}=\frac12\pi(t_{1/2})km$. Define the generating function
\[
	F(z) = \sum_{\sigma':V_{t_{1/2}} \to \{0,1\}} z^{w_{t_{1/2}}(\sigma')}
\]
By definition, the sought quantity is $[z^{\pi(t_{1/2})km/2}]F(z)$. Moreover, the definition of $F(z)$ and~\eqref{eq:wt12} imply that
\[
	F(z) = \sum_{\sigma':V_{t_{1/2}} \to \{0,1\}} \prod_{v \in V_{t_{1/2}}}(\vecone_{[\sigma'(v)=1]}z^{d_v} + \vecone_{[\sigma'(v)=0]}z^{d_{\neg v}})
\]
The assertion follows.
\qed
\end{proof}

Lemma~\ref{Lemma_entropy} follows immediately from the next statement, which is shown in Section~\ref{ssec:coeff_extract_simple}.
\begin{proposition}
\label{prop:coeff_extract_simple}
W.h.p.\ $\bf d$ chosen from $\bf D$ has the following property. There is a constant $C = C(k) > 0$ such that is we write $N = |V_{t_{1/2}}|$, then
\[
	[z^{\pi(t_{1/2})km/2}] \prod_{v \in V_{t_{1/2}}}(z^{d_v} + z^{d_{\neg v}}) = (C + o(1))N^{-1/2} 2^{N}.
\]
\end{proposition}
We proceed with the proof of Proposition~\ref{Prop_entropy}, i.e., we want to enumerate pairs of assignments with $p$-marginals that have a specific overlap. Let $s\in T$ be a signature. For any $\sigma, \tau\in\{0,1\}^n$ denote the by the $s$-overlap $o_s(\sigma,\tau)$ the number of literal occurrences that are satisfied in both $\sigma$ and $\tau$, where we consider only literals of signature $s$, i.e.,
\[
	o_s(\sigma, \tau) = \sum_{\ell \in L_s} \vecone_{[\sigma(\ell) = \tau(\ell) = 1]} d_\ell.
\]
Similarly, for any type $t\in \cal T$ we denote by $o_t(\sigma, \tau)$ the number of satisfied literal occurrences in both $\sigma$ and $\tau$, where only literals of type $t$ are considered. Note that $o_t(\sigma, \tau) = {\cal O}(\sigma, \tau)_t \pi(t) km$, where $\cal O$ is defined in Section~\ref{ssec:secMomOutline}. For the special case $t = t_{1/2}$ it follows
\begin{equation}
\label{eq:o12}
	o_{t_{1/2}}(\sigma, \tau) = \sum_{v \in V_{t_{1/2}}} \vecone_{[\sigma(v) = \tau(v) = 1]} d_v + \vecone_{[\sigma(\neg v) = \tau(\neg v) = 0]}.
\end{equation}
Let us begin with a simple observation. Let $s\in t$ such that $p(s)>1/2$, and let $\sigma,\tau$ be two assignments with $p$-marginals. Note that if $w_s(\sigma,\tau) = (1+\delta)p(s)^2\pi(s)km$, for some $\delta \ge -1$,  then the fraction of variables in $V_s$ that are set to true in $\sigma$ and $\tau$ is $(1+\delta)p(s)^2$. Consequently, the number of variables that are set to false in both assignments is $(1-p(s))|V_s| - (p(s)|V_s| - (1+\delta)p(s)^2|V_s|)$, and therefore
\[
\begin{split}
	w_{\neg s}(\sigma, \tau)
	& = (1-p(s))\pi(\neg s)km - \left(p(s)\pi(\neg s)km - (1+\delta)p(s)^2\pi(\neg s)km\right)\\
	& = \left(1 - \delta\frac{(1-p(\neg s))^2}{p(\neg s)^2}\right)p(\neg s)^2\pi(\neg s)km.
\end{split}
\]
In words, the overlap in $s$ determines the overlap in $\neg s$. However, note that the $s'$-overlap, for any $s' \neq s,\neg s$, is not affected by the quantities $w_s(\sigma, \tau)$ and $w_{\neg s}(\sigma, \tau)$.

Let $t\in \cal T$ be a type. With the previous observation at hand we are able to estimate the number of pairs of $p$-satisfying assignments with a given $t$- and $\neg t$-overlap. The proof can be found in Section~\ref{ssec:toverlap}.
\begin{proposition}
\label{prop:toverlap}
There is a $c > 0$ such that the following is true. Let $\eps, \eps' > 0$. Let $t\in\cal T$ be a type such that $p(t)\neq 1/2$. Denote by ${\cal H}^2_{t, \neg t}(\eps,\eps')$ the set of pairs $\sigma$, $\tau$ of assignments with $p$-marginals, such that
\[
	|w_t(\sigma,\tau) - p(t)^2\pi(t)km| \ge \eps p(t)^2\pi(t)km
	\quad \text{ and }\quad
	|w_{\neg t}(\sigma,\tau) - p(\neg t)^2\pi(\neg t)km| \ge \eps' p(\neg t)^2\pi(\neg t)km.
\]	
Then, 
\[
	|{\cal H}^2_t(\eps, \eps')| \le |{\cal H}|^2 \cdot \exp\left\{-cn \, \big(\eps^2 \pi(t) + \eps'^2 \pi(\neg t)\big)\right\}.
\]
\end{proposition}
What remains is to enumerate pairs of $p$-satisfying assignments with a given $t_{1/2}$-overlap. The next proposition provides this number as the coefficient of an appropriately defined generating function.
\begin{proposition}
Let $\eps\in(-1/4,1/4)$. Let ${\cal H}^2_{1/2}(\eps)$ denote the set of pairs $\sigma',\tau'$ of assignments to the variables in $V_{t_{1/2}}$ such that
\[
	o_{t_{1/2}}(\sigma', \tau') = \left(\frac14 + \eps\right) \pi(t_{1/2})km
\]
and
\[
	w_{t_{1/2}}(\sigma') = w_{t_{1/2}}(\sigma') = \pi(t_{1/2}km) / 2.
\]
Then ${\cal H}^2_{1/2}(\eps) = [(xy)^{\pi(t_{1/2})km/2}\, u^{(1/4 + \eps)\pi(t_{1/2})km}]F(x,y,u)$, where
\[
	F(x,y,u) = \prod_{v \in V_{t_{1/2}}}	\left((xyu)^{d_v} + (xyu)^{d_{\neg v}} + x^{d_v}y^{d_{\neg v}} + x^{d_{\neg v}}y^{d_v}\right).
\]
\end{proposition}
\begin{proof}
Assign to a pair of assignments $\sigma',\tau'$ to the variables in $V_{t_{1/2}}$ the weight $x^{w_{t_{1/2}}(\sigma')} \, y^{w_{t_{1/2}}(\tau')} \, u^{o_{t_{1/2}}(\sigma', \tau')}$. Then, by using~\eqref{eq:wt12} and~\eqref{eq:o12} 
\[
\begin{split}
	 &\sum_{\sigma', \tau': V_{t_{1/2}} \to \{0,1\}} x^{w_{t_{1/2}}(\sigma')} \, y^{w_{t_{1/2}}(\tau')} \, u^{o_{t_{1/2}}(\sigma', \tau')} \\
	= & \sum_{\sigma', \tau': V_{t_{1/2}} \to \{0,1\}} \prod_{v\in V_{t_{1/2}}}
		\vecone_{[\sigma(v)=\tau(v)=1]}(xyu)^{d_v}
		+ \vecone_{[\sigma(v)=\tau(v)=0]}(xyu)^{d_{\neg v}} \\
		& \qquad \qquad \qquad \qquad \qquad \qquad + \vecone_{[\sigma(v)=1, \tau(v)=0]}x^{d_v}y^{d_{\neg v}}
		+ \vecone_{[\sigma(v)=0,\tau(v)=1]}x^{d_{\neg v}}y^{d_v}.
\end{split}
\]
Summing this expression up yields the claimed statement.
\qed
\end{proof}

The next statement provides the asymptotic value of the sought coefficients of $F(x,y,u)$ from the previous proposition. The proof can be found in Section~\ref{ssec:coeff_extract_triple}.
\begin{proposition}
\label{prop:coeff_extract_triple}
W.h.p.\ $\bf d$ chosen from $\bf D$ has the following property. There is a constant $C = C(k,\eps) > 0$ such that if we write $N = |V_{t_{1/2}}|$ and $M = \pi(t_{1/2})km$, then
\[
	[(xy)^{M/2}\, u^{(1/4 + \eps)M}]F(x,y,u) = (C + o(1)) \cdot E \cdot N^{-3/2},
\]
where
\begin{equation}
\label{eq:expGrowth_triple}
	E = \rho^{-(1-4\eps)M/2} \prod_{v\in V_{t_{1/2}}} (2 + 2\rho^{d_v + d_{\neg v}})
\end{equation}
and $\rho$ is the solution to the equation
\begin{equation}
\label{eq:rho}
	(1/4+\eps)M = \sum_{v \in V_{t_{1/2}}} \frac{d_v + d_{\neg v}}{2 +2\rho^{d_v + d_{\neg v}}}.
\end{equation}
\end{proposition}
In order to complete the proof of Proposition~\ref{Prop_entropy} we will estimate the exponential term in the previous statement as a function of $\eps$. Note that if $\eps = 0$, then clearly $\rho = 1$ and $E = 4^N$. Let $|\eps| < 1/100$. We begin with providing bounds for the value of $\rho$ from Equation~\eqref{eq:rho}. Let $f_g(\rho) = g/(2 + 2\rho^g)$, where $g \ge 3$. Then $f_g(1) = g/4, f_g'(1)=-g^2/8$ and
\[
	f_g''(\rho) = {\frac {{g}^{2} \left( g{\rho}^{2\,g-2}-g{\rho}^{g-2}+{\rho}^{g-2}+{\rho}^{2
\,g-2} \right) }{2 \left( 1+{\rho}^{g} \right) ^{3}}}.
\]
Note that if $0 \le \rho \le 1$, then, with room to spare, $|f_g''(\rho)| \le g^3$. Moreover, if $\rho > 1$, then we may estimate $f_g''$ as follows:
\[
	|f_g''(\rho)| < \frac{{g}^{2} \left( (g+1){\rho}^{2\,g-2}+g{\rho}^{g-2}+{\rho}^{g-2}\right) }{2\rho^{3g}} \le \frac{g^3}{\rho^g} \le g^3.
\]
Let us write $\rho = 1 + \delta$. Taylor's theorem then implies that $|f_g(\rho) - (g/4 - g^2\delta/8 )| \le g^2\delta^2$.
By writing $g_v = d_v + d_{\neg v}$ and recalling that $M = \sum_{v\in V_{t_{1/2}}} g_v$ we infer from~\eqref{eq:rho}
\[
	-\delta \frac{S_2}{8} - \delta^2 S_3 \le \eps M \le -\delta \frac{S_2}{8} + \delta^2 S_3,
	\quad \text{ where }\quad
	S_i = \sum_{v \in V_{t_{1/2}}} g_v^i, \quad \text{for } i\in\{2,3\}.
\]
In view of these inequalities we might expect that whenever $\eps$ is not too large, then $\delta \approx -\eps{8M}/{S_2}$. This can be made precise as follows. By solving the quadratic equations explicitly we infer that $\delta$ satisfies 
\[
	\frac1{16}\frac{-S_2 + \sqrt{S_2^2 - 256S_3 \eps M}}{S_3} \le \delta \le -\frac1{16}\frac{-S_2 + \sqrt{S_2^2 +256S_3 \eps M}}{S_3}
\]
Note that $\bf d$ is such that w.h.p.\ $S_2 = \Theta(kr M)$  and $S_3 = \Theta((kr)^2 M)$.
Thus, for sufficiently large $k$
\[
	\sqrt{S_2^2 +256S_3 \eps M} = S_2 \sqrt{1 + \frac{256S_3 \eps M}{S_2^2}}
	= S_2 + \frac{128S_3 \eps M}{S_2} + O\left(\frac{S_3^2 \eps^2 M^2}{S_2^3}\right).
\]
The square-root with the minus sign can be estimated analogously. We infer that
\begin{equation}
\label{eq:rho_eps}
	\rho = 1 + \delta, \quad \text{where}\quad \delta = -\eps \frac{8M}{S_2} + O((kr)^{-1} \eps^2 ).
\end{equation}
With the approximate value of $\rho$ at hand we can proceed with estimating the exponential term in~\eqref{eq:expGrowth_triple}. First of, we rearrange terms to obtain
\begin{equation}
\label{eq:E_probabilistic}
	E
	= \rho^{-(1-4\eps)M/2} \prod_{v\in V_{t_{1/2}}} (2 + 2\rho^{d_v + d_{\neg v}})
	= 4^N \cdot \rho^{2\eps M} \cdot \prod_{v\in V_{t_{1/2}}} (\rho^{-g_v/2} + \rho^{g_v/2})/2.
\end{equation}
The bounds on $\rho$ imply that
\begin{equation}
\label{eq:boundrhoinE}
	\rho^{2\eps M}
	= \left(1 -\eps \frac{8M}{S_2} + O((kr)^{-1} \eps^2 )\right)^{2\eps M}
	\le \exp\left\{-16 \eps^2 \frac{M^2}{S_2} + O(\eps^3 (kr)^{-1}M)\right\}.
\end{equation}
Regarding the last term involving the product in~\eqref{eq:E_probabilistic}, we bound it by the following probabilistic considerations. Note that
\[
	\prod_{v\in V_{t_{1/2}}} (\rho^{-g_v/2} + \rho^{g_v/2})/2
	= \sum_{(s_v) : v\in V_{t_{1/2}}, s_v \in \{-1,+1\}} 2^{-N} \rho^{-1/2 \sum_{v} s_v g_v}.
\]
Let $(S_v)_{v \in V_{t_{1/2}}}$ be a family of independent random variables, which are uniformly distributed in $\{-1,+1\}$. Then the last expression in the previous display is equal to the expected value of $\rho\left\{-1/2 \sum_{v} s_v g_v\right\}$. We obtain 
\[
	\mu := \Exp\brk{\rho^{-\frac12\, \sum_{v \in V_{t_{1/2}}} S_v g_v}}
	\le 2 \sum_{t \ge 0} \pr\brk{|\sum_{v \in V_{t_{1/2}}} S_v g_v| = t} (\rho^{t/2} + \rho^{-t/2})
\]
Note that since either $\rho^{t/2} \ge 1$ or $\rho^{-t/2} \ge 1$ we may assume without loss of generality that $\rho \ge 1$. The advantage of the above formulation is that we can estimate rather easily the probability for a large deviation of the sum $S = \sum_{v} S_v g_v$. Indeed, if we change the value of any $S_v$ to obtain a new sum $S'$, then $|S - S'| = 2g_v$. By applying Azuma-Hoeffding we obtain
\[
	\pr\brk{|\sum_{v \in V_{t_{1/2}}} S_v g_v| = t} \le \exp\left\{-2t^2/\sum_v (2g_v)^2\right\} = \exp\{-t^2/2S_2\}.
\]
Thus, by using~\eqref{eq:rho_eps} and noting that $\eps \le 0$ due to our assumption $\rho \ge 1$ we obtain the bound
\[
	\mu \le 4 \sum_{t \ge 0} e^{-t^2 / 2S_2} \cdot \rho^{t/2}
	\le 4 \sum_{t \ge 0} e^{-t^2 / 2S_2} \cdot \left(1 - \eps\frac{8M}{S_2}\right)^{t/2}
	\le 4 \sum_{t \ge 0} \exp\left\{-\frac{t^2}{2S_2} + |\eps|\frac{4Mt}{S_2}\right\}.
\]
Since the exponent is convex in $t$, it can easily be seen that it is maximized at $t = 4M|\eps|$, where its value equals
\[
	-\frac{(4M|\eps|)^2}{2S_2} + |\eps|\frac{4M(4M|\eps|)}{S_2} = 8\eps^2 \frac{M^2}{S_2}.
\]
Thus, $\mu = O(\sqrt{N}) e^{8\eps^2 \frac{M^2}{S_2}}$, and by combining~\eqref{eq:E_probabilistic} and~\eqref{eq:boundrhoinE} we infer that $E \le \sqrt{N} e^{-8\eps^2 \frac{M^2}{S_2}}$. But since $S_2 = \Theta(kr M)$ and $M = \Theta(kr N)$, this is at most $\sqrt{N} e^{-c \eps^2 N}$, for some $c > 0$.

Proposition~\ref{Prop_entropy} then follows immediately from Propositions~\ref{prop:toverlap}-\ref{prop:coeff_extract_triple}, and the (aforementioned) observation that the $t$- and $t'$-overlap of $\SIGMA,\TAU$ are independent for $t \neq t', \neg t$.

\subsection{Proof of Proposition~\ref{prop:coeff_extract_simple}}
\label{ssec:coeff_extract_simple}

Set $M = \pi(t_{1/2})km$.
By the virtue of Cauchy's integral formula we obtain
\[
	I := [z^{M/2}] F(z) = \frac{1}{2\pi i}\oint_C F(z) z^{-M/2-1} dz.
\]
Since $F$ is analytic in $\mathbf{C}$, $C$ can be any curve enclosing the origin. To estimate the integral we will use the saddle point method, which is commonly used to determine the asymptotic behavior of integrals that involve a large
parameter, and are simultaneously subject to huge variations. For an excellent overview and numerous applications we refer the reader to~\cite{FlajSed}.

The main idea is to choose $C$ such that the integrand 'peaks' at a unique point, so that the main contribution to the integral comes from a small neighborhood of this maximum. We choose $C$ to be the unit circle centered at the origin, i.e., $C = \{e^{i \theta}: -\pi < \theta < \pi \}$. Moreover, let $\theta_0 = \theta_0(n) = N^{-2/5}$, and write $C_0 = \{e^{i \theta} : |\theta| \le \theta_0(n)\}$ for the restriction of $C$ to the segment with $|\theta| \le \theta_0(n)$. Then we may write $I = I_0 + I_1$, where
\[
	I_0 = \frac{1}{2\pi i}\oint_{C_0 } F(z) z^{-M/2-1} dz
	~\text{ and }~
	I_1 = \frac{1}{2\pi i}\oint_{C \setminus C_0} F(z) z^{-M/2-1} dz.
\]
By changing variables, the first integral becomes
\begin{equation}
\label{eq:main_simple}
	I_0 =
	\frac{1}{2\pi}\int_{-\theta_0}^{\theta_0} H(\theta) d\theta,
	\quad\text{where}\quad
	H(\theta)
	= e^{-i\theta M/2} \cdot \prod_{v\in V_{t_{1/2}}} (e^{i\theta d_v} + e^{i\theta d_{\neg v}}).
\end{equation}
Moreover, by using the trivial bound for complex integrals and the fact $|z| = 1$ on $C$ we obtain
\begin{equation}
\label{eq:tails_simple}
	I_1 \le 2\pi \cdot \sup_{z \in C \setminus C_0 } \left| F(z) \right|
\end{equation}
Our subsequent proof strategy is as follows. We will first compute the asymptotic value of the integral over the 'central region'; in particular, we show that
\begin{equation}
\label{eq:main_simple_asympt}
	I_0 = (c+o(1)) N^{-1/2}2^N
\end{equation}
for an appropriate $c > 0$. Then, by using~\eqref{eq:tails_simple} we show that $I_1 = o(I_0)$. The two statements combined yield then immediately the conclusion of the proposition.

We proceed with showing~\eqref{eq:main_simple_asympt}. Recall that $|\theta| \le \theta_0 = N^{-2/5}$, and note that for any $d, d'$, by applying Taylor's Theorem
\[
	e^{i\theta d} + e^{i \theta d'} = 2 + i(d + d')\theta - \frac{1}{2}(d^2 + d'^2)\theta^2 + O\Big((1+i)(d^3 + d'^3)\theta^3\Big)
	\quad \text{ uniformly for all } d,d'\in \mathbf{N}, |\theta| \le \theta_0.
\]
Let us write
\[
	S_2 = \frac14 \, \sum_{v\in V_{t_{1/2}}} d_v^2 + d_{\neg v}^2 + (d_v + d_{\neg v})^2
	\quad\text{and}\quad
	S_j = \sum_{v\in V_{t_{1/2}}} (d_v^j + d_{\neg v}^j).
	\quad\text{ for $j \ge 3$. }
\]
Observe that $\bf d$ is w.h.p.\ such that $S_j = (1+o(1))c_j N$ for some $c_j = c_j(k) > 0$, where $2 \le j \le 9$. Using~\eqref{eq:main_simple} we infer that the integrand satisfies
\[
\begin{split}
	H(\theta)
	& = e^{-i\theta M/2} \cdot \prod_{v\in V_{t_{1/2}}}\left(2 + i(d_v + d_{\neg v})\theta - \frac12(d_v^2 + d_{\neg v}^2)\theta^2 + O\Big((1+i)(d_v^3 + d_{\neg v}^3)\theta^3\Big)\right) \\
	& = 2^N \, \exp\left\{-S_2 \theta^2 + O\left((1+i)(S_3 \theta^3 + S_4\theta^4 +\dots + S_9 \theta^9)\right)\right\} \\
	& = (1+o(1)) 2^N \, \exp\Big\{-S_2 \theta^2 \Big\}, \qquad \text{since } \theta \le N^{-2/5}.
\end{split}
\]
Thus,
\[
\begin{split}
	(2\pi)I_0
	= \int_{-\theta_0}^{\theta_0} H(\theta)d\theta
	& = (1+o(1))\, 2^N \, \int_{-\theta_0}^{\theta_0} e^{-S_2 \theta^2}d\theta \\
	& = (1+o(1))\, \frac{2^N}{\sqrt{c_2 N}} \, \int_{-\sqrt{c_2}N^{1/10}}^{\sqrt{c_2} N^{1/10}} e^{-x^2}dx
	= (1+o(1))\, \frac{2^N}{\sqrt{2\pi c_2 N}}.
\end{split}
\]

This proves~\eqref{eq:main_simple_asympt}. To complete the proof we will show that $\sup_{z \in C \setminus C_0 } \left| F(z) \right|$ is asymptotically negligible compared to $I_0$.
First, for any $v \in V_{t_{1/2}}$
\[
	f_v(\theta) := |e^{i\theta d_v} + e^{i\theta d_{\neg v}}| = \sqrt {2+2\,\cos \left( \theta \left( d_v-d_{\neg v} \right)  \right) }
\]
Let us collect some basic properties of $f_v$. Note that if $d_v = d_{\neg v}$, then $f_v(\theta) = 2$ for any $-\pi < \theta < \pi$. Otherwise, $f$ is maximized for any
\[
	\theta \in {\cal M}_{d_v - d_{\neg v}} = \left\{ j \frac{2\pi}{|d_v - d_{\neg v}|} :  |j| < \frac{|d_v - d_{\neg v}|}2\right\},
\]
where $f(\theta) = 2$.

For a pair $(d_+, d_-) \in \mathbf{N}^2$ let $V_{d_+, d_-} \subseteq V_{t_{1/2}}$ denote the set of variables $v$ such that $d_v = d_+$ and $d_{\neg v} = d_-$, and write $N_{d_+, d_-} = |V_{d_+, d_-}|$. Then,
\[
	|F(e^{i\theta})| = \prod_{v \in V_{t_{1/2}}} f_v(\theta) = \prod_{s=(d_+, d_-)} \big(2 + 2\cos(\theta(d_+ - d_-))\big)^{N_s / 2}.
\]
Note that $\sum_{s=(d_+, d_-)} N_s = N$. Thus, $|F(e^{i\theta})| \le 2^N$ for all $\theta$. However, this bound is achieved only if all factors are maximized simultaneously. We will argue in the sequel that if $|\theta| \in (\theta_0, \pi)$, then a linear (in $N$) fraction of the factors is $\le 2 - O(N^{-4/5})$. It follows for some $\alpha > 0$ that
\[
	|F(e^{i\theta})|
	\le 2^{(1-\alpha)N} \cdot (2 - O(N^{-4/5}))^{\alpha N}
	= 2^N \cdot e^{-O(N^{1/5})}
	= o(N^{-1/2} 2^N) = o(I_0).
\]
To see the claim, consider the specific pair $(d'_+, d'_-) = (kr, kr-1)$, and note that if $k$ is sufficiently large, then $kr-1 > kr/2 + 10\sqrt{k 2^k \ln k}$. So, indeed $V_{d_+, d_-} \subseteq V_{t_{1/2}}$. Furthermore, $\bf d$ is such that w.h.p.\ there is a constant $\alpha = \alpha(k) >0$ such that $N_{d_+, d_-} \ge \alpha N$. It follows that for all variables $v\in V_{d'_+, d'_-}$
\[
	f_v(\theta) = \sqrt{2 + 2\cos(\theta)}.
\]
It can easily be verified that $f_v$ is monotone increasing for $-\pi < \theta < 0$ and decreasing for $0 < \theta < \pi$. Thus, for any $|\theta| \in (\theta_0, \pi)$ we have $f_v(\theta) \le \max\{f_v(\theta_0), f_v(-\theta_0)\}$. By using the Taylor series expansion of the cosine and the square root we obtain that
\[
	f_v(\eps) = 2 - \frac{\theta^2}{4} + O(\theta^4), \quad \text{uniformly for all } -\pi < \theta < \pi.
\]
We conclude that $f_v(\theta) \le 2 - O(n^{-4/5})$ for at least $\alpha N$ variables $v$, and the proof is completed.

\subsection{Proof of Proposition~\ref{prop:toverlap}}
\label{ssec:toverlap}

We will exploit a concentration inequality due to McDiarmid~\cite{McD02}. We present it here in a simplified form that is appropriate for our purpose.  Given a finite non-empty set $B$, we denote by $Sym(B)$ the set of all $|B|!$ permutations of the elements of $B$. Let $B_1, \dots, B_N$ be a family of finite non-empty sets, and denote by $\Omega = Sym(B_1) \times \dots \times Sym(B_N)$. Moreover, let $\vec\pi = (\pi_1, \dots, \pi_N)$ be a family of independent random permutations, where $\pi_i$ is drawn uniformly from $Sym(B_i)$. 
\begin{theorem}\label{thm:McDiarmid} 
Let $c$ and $r$ be positive constants. Suppose that $h:\Omega \to \mathbf{R}_+$ is such that for any $\pi\in\Omega$ the following conditions are satisfied.
\begin{itemize}
\item If $\pi'$ can be obtained from $\pi$ by swapping two elements, then $|h(\pi) - h(\pi')| \le c$.
\item If $h(\pi) \ge s$, then there is a set of at most $rs$ coordinates such that $h(\pi')\geq s$ for any $\pi' \in \Omega$ that agrees with $\pi$ on these coordinates.  
\end{itemize}
Let $Z = h(\vec\pi)$ and let $m$ be the median of $Z$. Then, for any $t>0$ 
\begin{equation*}
\pr\brk{|Z-m|>t } \leq 4 \exp \left(- {t^2 \over 16rc^2(m+t)} \right). 
\end{equation*}
\end{theorem}

Let us proceed with the proof of Proposition~\ref{prop:toverlap}. We will assume without loss of generality that $t$ is such that $p(t) > 1/2$. We will abbreviate $p = p(t)$, $q = p(\neg t)$. Let $\sigma$ be an arbitrary assignment with $p$-marginals. Moreover, denote by $\TAU$ an assignment that is obtained by selecting for any signature $s\in t$ uniformly at random $p |V_s|$ variables from $V_s$ and setting them to true, and setting all other variables in $V\setminus V_t$ arbitrarily so that $\TAU$ has $p$-marginals. Equivalently, we may generate $\TAU$ by permuting the variables in $V_s$ randomly, and setting the first $p|V_s|$ variables in that permutation to true, for all $s\in t$. With this notation we obtain
\[
	|{\cal H}^2_{t, \neg t}(\eps, \eps')| \le |{\cal H}|^2 \cdot \pr\brk{|w_t(\sigma, \TAU) - p^2\pi(t)km| \ge \eps \pi(t)km}
\]
The latter probability can be estimated with Theorem~\ref{thm:McDiarmid}. Indeed, note that
\begin{itemize}
	\item if $\tau, \tau'$ have $p$-marginals and can be obtained by swapping the truth assignment of two variables, then $$|w_t(\sigma, \tau) - w_t(\sigma, \tau')| \le 2\max_{v \in V_t}d_v \le 4kr.$$
	\item if $w_t(\sigma, \tau) \ge s$, then there is a set $S$ of $\le s/ \min_{v \in V_t}d_v \le 2s/kr$ variables that are set to true, and any $\tau'$ with $p$-marginals that sets all variables is $S$ to true satisfies $w_t(\sigma, \tau') \ge s$.
\end{itemize}
We thus may apply Theorem~\ref{thm:McDiarmid} with $c =  4kr$ and $r = 2/kr$. Moreover, trivially $\Erw\brk{w_t(\sigma, \TAU)} \le \pi(t)km$. We infer that
\[
	\frac{|{\cal H}^2_{t, \neg t}(\eps, \eps')|}{|{\cal H}|^2} \le 4 \exp \left(- \Theta(1)\frac{(\eps \pi(t) k m)^2}{kr \cdot \pi(t) k m} \right) = 4 \exp \left(- \Theta(1)\, \eps^2 \pi(t) n \right). 
\]
Exactly the same argument, where we interchange the roles of $t$ and $\neg t$, shows that also
\[
	\frac{|{\cal H}^2_{t, \neg t}(\eps, \eps')|}{|{\cal H}|^2} \le 4 \exp \left(- \Theta(1)\frac{(\eps' \pi(\neg t) k m)^2}{kr \cdot \pi(\neg t) k m} \right) = 4 \exp \left(- \Theta(1)\, \eps^2 \pi(\neg t) n \right). 
\]
The claim follows.

\subsection{Proof of Proposition~\ref{prop:coeff_extract_triple}}
\label{ssec:coeff_extract_triple}

Set $M = \pi(t_{1/2})km$. By applying Cauchy's integral formula we obtain
\[
	I := [(xy)^{M/2}\, u^{(1/4 + \eps)M}] F(x,y,u) = \frac{1}{(2\pi i)^3}\oint_{C_1} \oint_{C_2} \oint_{C_o} F(x,y,u) (xy)^{-M/2-1} u^{-(1/4+\eps)M-1} du dy dx.
\]
The function $F$ is analytic in $\mathbf{C}^3$, implying that $C_1, C_2, C_o$ can be any curves enclosing the origin. We choose
\[
	C_1 = \{\rho e^{i\theta} : |\theta|<\pi\}, \quad C_2 = \{ \rho e^{i\varphi}: |\varphi|<\pi \}, \quad C_o =\{ \rho^{-2} e^{i \psi} : |\psi| < \pi\},
\]
where $\rho$ is the solution to the Equation~\eqref{eq:rho}. Some remarks are in place here. The choice of the integration paths may seem arbitrary at this point. Note, however, that $F$ is symmetric with respect to $x$ and $y$, and thus it is natural to assume similar integration curves for them. Moreover, the choice of $\rho$ is guided by the general principles of the saddle-point method and is such that the integrand has a unique maximum at $(\theta,\varphi,\psi) = (0,0,0)$. Indeed, as we will show subsequently, the integrand is around $(0,0,0)$ of elliptic type; this allows us to reduce the estimation of the main terms to the evaluation of a 3-dimensional Gaussian integral.

Denote by $\cal C$ the restriction of the circles $C_1, C_2, C_o$ to a small region around the origin, i.e.,
\[
	{\cal C} = \{\rho e^{i\theta}: |\theta| < N^{-2/5}\} \times \{\rho e^{i\varphi} : |\varphi| < N^{-2/5}\} \times \{\rho^{-2}e^{i\psi} : |\psi| < N^{-2/5}\}.
\]
Then we may write $I = I_0 + I_1$, where
\[
	I_0 = \frac{1}{(2\pi i)^3}\oint_{\cal C} F(x,y,u)\,(xy)^{-M/2-1}z^{-(1/4+\eps)M-1} dzdydx,
\]
and $I_1$ is the integral over $(C_1\times C_2 \times C_o)\setminus \cal C$. By changing variables we obtain
\begin{equation}
\label{eq:I0triple}
	I_0 = \frac1{(2\pi)^3}\int\limits_{[-N^{-2/5},N^{-2/5}]^3} H(\theta,\varphi,\psi) d\psi d\varphi d\theta,
	\text{ where }
	H = \rho^{-\frac{(1 - 4\eps)M}2} e^{-i\frac{(\theta + \varphi)M}2 - i\psi(1/4+\eps)M} \prod_{v\in V_{t_{1/2}}} h_v(\theta,\varphi,\psi),
\end{equation}
and
\[
	h_v(\theta,\varphi,\psi) = e^{i(\theta + \varphi + \psi)d_v} + e^{i(\theta + \varphi + \psi)d_{\neg v}} + \rho^{d_v + d_{\neg v}}e^{i\theta d_v + i\varphi d_{\neg v}} + \rho^{d_v + d_{\neg v}}e^{i\theta d_{\neg v} + i\varphi d_v}.
\]
Regarding $I_1$, we will use the trivial bound
\begin{equation}
\label{eq:trival_triple}
	I_1 \le (2\pi)^3 \sup_{(x,y,u) \in (C_1 \times C_2 \times C_o) \setminus {\cal C}} |H(x,y,u)|
\end{equation}
to show that $I_1 = o(I_0)$.

We begin with estimating $I_0$ by providing an appropriate asymptotic expansion of it for points around the origin. First of all, note that for any $v\in V_{t_{1/2}}$ we have $h_v(0,0,0) = 2 + 2\rho^{d_v + d_{\neg v}}$ and thus
\[
	H(0,0,0) = \rho^{-(1-4\eps)M/2} \, \prod_{v_\in V_{t_{1/2}}} (2 + 2\rho^{d_v + d_{\neg v}}) = E.
\]
Moreover,
\[
	\frac{\partial}{\partial \theta}h_v(0,0,0) = \frac{\partial}{\partial \varphi}h_v(0,0,0) = (2 + 2\rho^{d_v + d_{\neg v}})\, \frac{i}2(d_v + d_{\neg v}), 
	\quad\text{and}\quad
	\frac{\partial}{\partial \psi}h_v(0,0,0) = i(d_v + d_{\neg v}).
\]
The second derivatives at $(0,0,0)$ are given by
\[
	\frac{\partial^2}{\partial \theta^2}h_v = \frac{\partial^2}{\partial \varphi^2}h_v = -(d_v^2+d_{\neg v}^2)(1 + \rho^{d_v + d_{\neg v}}),
	\quad\text{and}\quad
	\frac{\partial^2}{\partial \psi^2}h_v = -(d_v^2+d_{\neg v}^2).
\]
Furthermore, the mixed second derivatives are
\[
	\frac{\partial^2}{\partial\theta\partial\varphi}h_v = -(d_v^2+d_{\neg v}^2+2d_vd_{\neg v}\rho^{d_v + d_{\neg v}})
	\quad\text{and}\quad
	\frac{\partial^2}{\partial\theta\partial\psi}h_v = \frac{\partial^2}{\partial\phi\partial\psi}h_v= -(d_v^2+d_{\neg v}^2).
\]
We will also need crude bounds for the third-order derivatives in order to establish an accurate approximation for $H$ around the origin. Note that $h_v$ linearly exponential in $\theta,\varphi,\psi$ and $d_v, d_{\neg v}$. Thus, every time we take a derivative with respect to some variable, the norm of each single term in the expression of $h_v$ can increase by at most $m_v = \max\{d_v, d_{\neg v}\}$. Thus, uniformly for $(\theta,\varphi,\psi)\in [-N^{2/5},N^{2/5}]$ we have that
\[
	\left|\frac{\partial^3}{\partial \xi_1 \partial \xi_2 \partial \xi_3}h_v\right| \le 2(1 + \rho^{d_v + d_{\neg v}}) (d_v + d_{\neg v})^3,
	\quad \text{ where }\quad
	\xi_1,\xi_2,\xi_3 \in \{\theta,\varphi,\psi\}.
\]
By using the uniform estimate $1+x = e^{x - x^2/2 + \Theta(x^3)}$, where we set $1 + x = h_v(\theta,\varphi,\psi)/h_v(0,0,0)$ we infer that
\begin{equation}
\label{eq:hvfirstorder}
	\ln \frac{h_v(\theta,\varphi,\psi)}{h_v(0,0,0)} = \frac{i}{2}(d_v + d_{\neg v})(\theta + \phi) + i\frac{d_v+d_{\neg v}}{2+2\rho^{d_v + d_{\neg v}}}\psi + \text{ 2nd order }  + \text{error},
\end{equation}
where the 2nd order terms are
\[
	-\frac{(d_v - d_{\neg v})^2}{8}(\theta^2+\phi^2)
	- \frac{(d_v-d_{\neg v})^2 + 2\rho^{d_v + d_{\neg v}}(d_v^2+d_{\neg v}^2)}{2(2+2\rho^{d_v + d_{\neg v}})^2}\psi^2
	+ \frac{(d_v-d_{\neg v})^2(\rho^{d_v+d_{\neg v}}-1)}{2(2+2\rho^{d_v + d_{\neg v}})}\theta\varphi
	- \frac{(d_v-d_{\neg v})^2}{4+4\rho^{d_v + d_{\neg v}}}(\theta + \varphi)\psi.
\]
Finally, since $(\theta,\varphi,\psi)\in [-N^{2/5},N^{2/5}]$ the error term is of order at most $(d_v + d_{\neg v})^3N^{-6/5}$. In order to obtain an approximation for $H$ we form the product over all $v\in V_{t_{1/2}}$. Observe that the (linear in the variables) exponential factor $e^{-i(\theta+\varphi)M/2 - i\psi(1/4+\eps)M}$ cancels exactly with the first order terms in~\eqref{eq:hvfirstorder}. By abbreviating
\[
	S_{\theta,\theta} = \sum_{v\in V_{t_{1/2}}} \frac{(d_v - d_{\neg v})^2}8,
	\quad S_{\psi,\psi} = \sum_{v\in V_{t_{1/2}}} \frac{(d_v - d_{\neg v})^2 + 2\rho^{d_v + d_{\neg v}}(d_v^2+d_{\neg v}^2)}{2(2+2\rho^{d_v + d_{\neg v}})^2},
\]
and
\[
	S_{\theta,\phi} = \sum_{v\in V_{t_{1/2}}}\frac{(d_v - d_{\neg v})^2(\rho^{d_v+d_{\neg v}}-1)}{4+4\rho^{d_v+d_{\neg v}}},
	\quad
	S_{\theta,\psi} = \sum_{v\in V_{t_{1/2}}} \frac{(d_v - d_{\neg v})^2}{4+4\rho^{d_v + d_{\neg v}}},
	\quad
	S_3 = \sum_{v\in V_{t_{1/2}}}(d_v + d_{\neg v})^3
\]
we obtain uniformly for any $(\theta,\varphi,\psi) \in [-N^{-2/5},N^{-2/5}]^3$
\[
\ln\left(\frac{H}E\right) = -S_{\theta,\theta}(\theta^2+\varphi^2) - S_{\psi,\psi}\psi^2 + S_{\theta,\phi}\theta\phi - S_{\theta,\psi}(\theta + \varphi)\psi + O(S_3 N^{-6/5}).
\]
Observe that $\bf d$ is such that w.h.p.\ all quantities $S_{.,.}$ and $S_3$ are linear in $N$. Thus, we are left with computing
\[
	I_0 = (1+o(1)) E \cdot \int_{[-N^{-2/5},N^{-2/5}]^3} e^{-S_{\theta,\theta}(\theta^2+\varphi^2) - S_{\psi,\psi}\psi^2 + S_{\theta,\phi}\theta\varphi - S_{\theta,\psi}(\theta + \varphi)\psi} d\psi d\varphi d\theta.
\]
In order to compute this integral we rescale each variable with $N^{-1/2}$. By writing $s_{.,.}$ for $S_{.,.}/N$ we obtain
\[
	I_0 = (1+o(1)) E \cdot N^{-3/2} \cdot \int_{[-N^{1/10},N^{1/10}]^3} e^{-s_{\theta,\theta}(\theta^2+\varphi^2) - s_{\psi,\psi}\psi^2 + s_{\theta,\varphi}\theta\varphi - s_{\theta,\psi}(\theta + \varphi)\psi} d\psi d\varphi d\theta.
\]
A termwise comparison and elementary algebraic manipulations yield that
\[
	4S_{\theta,\theta}^2 - S_{\theta,\varphi}^2 \ge 0
	\quad \text{and} \quad
	2S_{\psi,\psi}S_{\theta,\theta} - S_{\theta,\psi}^2 - S_{\psi,\psi}S_{\theta,\varphi} \ge 0
\]
Thus, the squares can be completed and the integral in the above expression equals a constant depending on the family $s_{.,.}$; this shows that asymptotically $I_1$ is proportional to $N^{-3/2} \cdot E$.

In order to complete the proof we will use~\eqref{eq:trival_triple} to show that $I_1$ is asymptotically negligible compared to $I_0$. Recall the definition of $H$ from~\eqref{eq:I0triple}. It follows that the absolute value of $H$ is given by 
\[
	\rho^{-(1-4\eps)M/2} \cdot \prod_{v\in V_{t_{1/2}}} f_v(\theta,\varphi,\psi),
	\quad \text{where}\quad
	f_v(\theta,\varphi,\psi) = |h_v(\theta,\varphi,\psi)|.
\]
Let us abbreviate $D_v = d_v - d_{\neg v}$. A lengthy calculation, which can be performed easily with the help of MAPLE, yields that
\[
\begin{split}
	f_v(\theta,\varphi,\psi)^2
	& = 2 + 2\rho^{2(d_v + d_{\neg v})}
	+ 2\cos\big(D_v(\theta+\varphi+\psi)\big)
	+ 2\rho^{2(d_v + d_{\neg v})}\cos\big(D_v(\theta-\varphi)\big) \\
	& + 2\rho^{d_v + d_{\neg v}}
		\left(
			\cos\big(D_v\varphi + d_v\psi\big)
			+ \cos\big(D_v\theta + d_v\psi\big)
			+ \cos\big(D_v\theta - d_{\neg v}\psi\big)
			+ \cos\big(D_v\varphi - d_{\neg v}\psi\big)
		\right).
\end{split}
\]
Note that we can get an upper bound for $f_v$ if we replace all terms involving a cosine by one; this implies that $|H| \le \rho^{-(1-4\eps)M/2}\prod_v (2 + 2\rho^{d_v + d_{\neg v}}) = E$. Moreover, the bound is achieved only if all factors are maximized simultaneously, and this happens for example when we choose $(\theta,\varphi,\psi)=(0,0,0)$. We will argue in the sequel that if $(\theta,\varphi,\psi) \in (C_1\times C_2 \times C_o)\setminus \cal C$, i.e., at least one of the variables $\theta, \varphi, \psi$ is assigned a value not lying in $[-N^{-2/5}, N^{-2/5}]$, then there is a subset of variables $V' \subset V_{t_{1/2}}$ such that $|V'| \ge \alpha N$ for some $\alpha > 0$ and for all $v\in V'$ it holds $f_v(0,0,0) \le f_v(0,0,0) - O(N^{-4/5})$. Indeed, if this is true, then
\[
	|H| \le \rho^{-(1-4\eps)M/2} \prod_{v\in V_{t_{1/2}} \setminus V'} (2 + 2\rho^{d_v + d_{\neg v}}) \prod_{v\in V'}(2 + 2\rho^{d_v + d_{\neg v}} - O(N^{-4/5})).
\]
Since $\rho$ is bounded and $\bf d$ is such that w.h.p.\ $d_v + d_{\neg v} = o(\log n)$, it follows that $|H|$ smaller that $E$ by an exponential factor, which shows with~\eqref{eq:trival_triple} that $I_1 = o(I_0)$.

To see that a set $V'$ with the desired properties exists, let us assume that at least one of $\theta,\varphi,\psi$ is in absolute value at least $N^{-2/5}$. For a pair $(d_+, d_-) \in \mathbf{N}^2$ let $V_{d_+, d_-} \subseteq V_{t_{1/2}}$ denote the set of variables $v$ such that $d_v = d_+$ and $d_{\neg v} = d_-$, and write $N_{d_+, d_-} = |V_{d_+, d_-}|$. Consider the specific pair $(d_+,d_-) = (kr, kr-1)$, and note that for all such variables we have $D_v=1$. Furthermore, $\bf d$ is such that w.h.p.\ there is a constant $\beta = \beta(k) >0$ such that $N_{d_+, d_-} \ge \beta N$. Then we may assume that 
\[
	\text{for all } v\in N_{d_+, d_-}: ~ f_v(\theta, \varphi, \psi) \ge (2 + 2\rho^{2kr-1} - O(N^{-4/5})),
\]
as otherwise there is nothing to show. This impliesthat the arguments of all cosines appearing in the expression of $f_v$ are close to multiples of $2\pi$, and in particular, 
\begin{equation}
\label{eq:estimatesangles}
	|\theta + \phi + \psi|, \quad | \theta - \varphi|, \quad | \varphi + d_+\psi| = O(N^{-2/5}) \quad (\bmod~2\pi); 
\end{equation}
this follows directly from the series expansion of the cosine around integer multiples of $2\pi$, which lack a linear term. Next, consider the pair $(d'_+,d'_-) = (kr, kr-2)$; again $\bf d$ is such that w.h.p.\ there is a constant $\beta' = \beta'(k) >0$ such that $N_{d'_+, d'_-} \ge \beta' N$. Note that for these variables we have $D_v=2$. Then, as previously, we may also assume that 
\[
	\text{for all } v\in N_{d'_+, d'_-}: ~ f_v(\theta, \varphi, \psi) \ge (2 + 2\rho^{2kr-2} - O(N^{-4/5})),
\]
But then, by the same argument as above, $|2\varphi + d'_+\psi| = O(N^{-2/5})~(\bmod~2\pi)$. Since $d_+ = d'_+$ and, by assumption, $|\varphi|<\pi$, by combining this with the third term in~\eqref{eq:estimatesangles}, we infer that $|\varphi| = O(N^{-2/5})$. In turn, together with the second term in~\eqref{eq:estimatesangles}, this implies that also $|\theta| = O(N^{-2/5})$. Finally, the fact  $|\theta + \varphi + \psi| = O(N^{-2/5})~(\bmod~2\pi)$ from~\eqref{eq:estimatesangles} then also implies that $|\delta| = O(N^{-2/5})$. Everything together yields that $(\theta,\varphi,\psi)\in\cal C$, a contradiction.

\section{Proof of Corollary~\ref{XCor_asymmetric}}

As a direct consequence of our second moment argument, the Paley-Zygmund inequality,
and a concentration result on the number of satisfying assignments from~\cite{Barriers} we obtain the following.

\begin{proposition}\label{Prop_partition}
For $r$ as in \eqref{eq:rrho} we have 
	$\abs{\cS(\PHI)}\geq\Erw\abs{\cS(\PHI)}\cdot\exp\brk{-\frac{nr}{k^94^k}}$ \whp
\end{proposition}

We consider the following ``planted model'':
let $\Lambda=\Lambda_k(n,m)$ be the the of all pairs $(\Phi,\sigma)$ of $k$-CNFs
$\Phi$ over $V$ with $m$ clauses and satisfying assignments $\sigma\in\cS(\Phi)$.
Let $\pr_\Lambda$ signify the uniform distribution over $\Lambda$;
	$\pr_\Lambda$ is sometimes called the \emph{planted model}.
Moreover, let $\pr_G$ be the distribution on $\Lambda$ obtained by first choosing a random formula $\PHI$ and then
	a uniformly random $\sigma\in\cS(\PHI)$ (provided that $\PHI$ is satisfiable);
		$\pr_G$ is sometimes called the \emph{Gibbs distribution}.
Combining \Prop~\ref{Prop_partition} with an argument from~\cite{}, we obtain the following ``transfer result''.

\begin{corollary}\label{Cor_partition}
For any $\cB\subset\Lambda$ the following is true.
If $\pr_\Lambda\brk{\cB}\leq\exp\brk{-\frac{2nr}{k^94^k}}$, then $\pr_G\brk{\cB}=o(1)$.
\end{corollary}

\noindent
Thus, in order to show that some `bad' event $\cB$ is unlikely under $\pr_G$, we ``just'' need to show
that $\pr_\Lambda\brk{\cB}\leq\exp\brk{-\frac{2nr}{k^94^k}}$ is exponentially small.

\begin{lemma}\label{Lemma_towardMaj}
There is a number $\delta=\delta(k)>0$ such that
	\begin{eqnarray*}
	\pr_\Lambda\brk{\dist(\sigma,\sigmaMAJ)>\frac12-\delta}&\leq&\exp\brk{-\frac{2nr}{k^94^k}}.
	\end{eqnarray*}
\end{lemma}
\begin{proof}
We can generate a pair $(\Phi,\sigma)$ from the planted model as follows:
	first, choose $\sigma\in\cbc{0,1}^V$ uniformly;
		then, generate $m$ clauses that are satisfied under $\sigma$ uniformly and independently.
Without loss of generality, we may assume that $\sigma=\vecone$ is the all-true assignment.
We need to study the distribution $\vec d=(d_l)_{l\in L}$ of literal degrees.
To this end, let $(e_l)_{l\in L}$ be a family of independent Poisson variables such that
	$\Erw\brk{e_l}=\Erw\brk{d_l}$ for all $l$.
It is easily verified that there is $\zeta=\Theta(2^{-k})$ such that
	\begin{equation}\label{eqtowardMaj1}
	\Erw\brk{d_x}=\frac{kr}2(1+\zeta),\quad\Erw\brk{d_{\neg x}}=\frac{kr}2(1-\zeta)
	\end{equation}
for all $x\in V$.
Furthermore, if we let $\cE$ be the event that $\sum_{l\in L}e_l=km$, then $\vec e=(e_l)_{l\in L}$ given $\cE$ has the
same distribution as $\vec d$.
Moreover,
	\begin{equation}\label{eqtowardMaj2}
	\pr\brk{\cE}=\Omega(n^{-1/2}).
	\end{equation}
Let
	$$Y=\frac1n\sum_{x\in V}\vecone_{e_x>e_{\neg x}}+\frac12\vecone_{e_x=e_{\neg x}}.$$
Viewing the difference $e_x-e_{\neg x}$ as a random walk of length $\Po(kr)$ and using
limit theorems for resulting distribution (the  Skellam distribution),
we obtain from~(\ref{eqtowardMaj1}) that $\Erw\brk Y\geq\frac12+\Omega(\sqrt{kr}/2^k)$.
Further, applying Chernoff bounds to $Y$ (which is a sum of independent contributions), we find that for
a certain $\delta=\Omega(\sqrt{kr}/2^k)$ 
	\begin{equation}\label{eqtowardMaj3}
	\pr\brk{Y<\frac12+\delta}\leq\exp\brk{-\Omega(\sqrt{kr}/2^k)^2n}\leq \exp\brk{-\frac{3nr}{k^94^k}}.
	\end{equation}
Finally, the assertion follows from~(\ref{eqtowardMaj2}) and~(\ref{eqtowardMaj3}).
\qed\end{proof}

\section{Proof of Lemma~\ref{XLemma_wmaj}}

The expected majority weight in $\PHI$ is easily computed.
In $\PHI$, for each $x$ the numbers $d_x$, $d_{\neg x}$ of positive/negative occurrences are
asymptotically independently Poisson with mean $kr/2$.
Therefore, for any $d=\Theta(kr)$ we obtain
	$$\Erw\brk{\abs{d_x-d_{\neg x}}\,|\,d_x+d_{\neg x}=d}=\sqrt{2d/\pi}+O_k(1).$$
In effect,
	\begin{equation}\label{eqwmaj1}
	\Erw\brk{w_{maj}(\PHI)}\sim\frac12+\sqrt{\frac{2}{\pi kr}}+O_k(1/kr).
	\end{equation}

By comparison, given that, say, the all-true assignment is satisfying,
the number $d_x$ of positive occurrences has distribution $\Po((1+1/(2^k-1))kr/2)$, while 
$d_{\neg x}$ has distribution $\Po((1-1/(2^k-1))kr/2)$.
The normal approximation to the  Poisson distribution yields for $d=\Theta(kr)$,
	$$\Erw\brk{\abs{d_x-d_{\neg x}}\,|\,\vecone\in\cS(\PHI),d_x+d_{\neg x}=d}=\sqrt{2d/\pi}+\Theta(4^{-k}d^{3/2})+O_k(1).$$
for a certain constant $c>0$.
Consequently,
	\begin{equation}\label{eqwmaj2}
	\Erw\brk{w_{maj}(\PHI)\,|\,\vecone\in\cS(\PHI)}\sim\frac12+\sqrt{\frac{2}{\pi kr}}+\Theta(4^{-k}(kr)^{1/2}).
	\end{equation}

Both with and without conditioning on $\vecone\in\cS(\PHI)$, $w_{maj}$ enjoys the following Lipschitz property:
	changing one single clause can alter the value of $w_{maj}$ by at most $k/(km)=1/(rn)$.
Therefore, Azuma's inequality yields
	\begin{eqnarray*}
	\pr\brk{\abs{w_{maj}-\Erw\brk{w_{maj}}}>\lambda}&\leq&2\exp\brk{-\frac{(r\lambda n)^2}{2m}}=2\exp\brk{-\frac{r\lambda^2 n}{2}},\\
	\pr\brk{\abs{w_{maj}-\Erw\brk{w_{maj}}}>\lambda|\vecone\in\cS(\PHI)}&\leq&2\exp\brk{-\frac{r\lambda^2n}{2}}.
	\end{eqnarray*}
In effect, for a certain constant $\zeta>0$ we have
	\begin{eqnarray}\label{eqwmaj3}
	\pr\brk{w_{maj}\geq\frac12+\sqrt{\frac{2}{\pi kr}}+\zeta4^{-k}(kr)^{1/2}
		}&\leq&\exp\brk{-\Omega\bc{k/4^k}n},\\
	\pr\brk{w_{maj}\leq\frac12+\sqrt{\frac{2}{\pi kr}}+\zeta4^{-k}(kr)^{1/2}
		|\vecone\in\cS(\PHI)}&\leq&\exp\brk{-\Omega\bc{k/4^k}n}.\label{eqwmaj4}
	\end{eqnarray}
Combining~(\ref{eqwmaj3}) and~(\ref{eqwmaj4}) with a simple counting argument yields Lemma 2 from the extended abstract.

\bigskip
\noindent{\bf Acknowledgment.}
The first author thanks Dimitris Achlioptas 
 for helpful discussions on the second moment
method. 
We also thank Charilaos Efthymiou for helpful comments
that have led to an improved presentation.

\end{document}